\documentclass[11pt,reqno]{amsart}

\usepackage[T1]{fontenc}
\usepackage[utf8]{inputenc}
\usepackage{lmodern}
\usepackage{comment}
\usepackage{amsmath,amssymb,amsthm,mathtools}
\usepackage{mathrsfs}
\usepackage{xcolor}

\usepackage[colorlinks=true,citecolor=blue,linkcolor=blue,urlcolor=blue]{hyperref}

\numberwithin{equation}{section}
\newtheorem{theorem}{Theorem}[section]
\newtheorem{lemma}[theorem]{Lemma}
\newtheorem{corollary}[theorem]{Corollary}
\newtheorem{proposition}[theorem]{Proposition}

\newtheorem*{theoremA}{Theorem A}
\newtheorem*{theoremB}{Theorem B}
\newtheorem*{theoremC}{Theorem C}
\newtheorem*{theoremD}{Theorem D}
\theoremstyle{definition}
\newtheorem{definition}[theorem]{Definition}
\newtheorem{remark}[theorem]{Remark}

\DeclareMathOperator{\dist}{dist}

\newcommand{\Scal}{\mathcal S}

\begin{document}

\title[Singular Dimension Descent]{Positive Scalar Curvature Obstructions via Singular Dimension Descent}
\author[Y. Bi]{Yuchen Bi}
\address[Yuchen Bi]{Mathematical Institute, Department of Pure Mathematics, University of Freiburg, Ernst-Zermelo-Stra{\ss}e 1, D-79104 Freiburg im Breisgau, Germany}
\email{yuchen.bi@math.uni-freiburg.de}

\author[J. Zhu]{Jintian Zhu}
\address[Jintian Zhu]{Institute for Theoretical Sciences, Westlake University, 600 Dunyu Road, 310030, Hangzhou, Zhejiang, People's Republic of China}
\email{zhujintian@westlake.edu.cn}
\hypersetup{pdftitle={Positive Scalar Curvature Obstructions via Singular Dimension Descent},pdfauthor={}}

\subjclass[2020]{Primary 53C21; Secondary 53C23, 49Q05}

\begin{abstract}
In light of recent advances in conformal blow-up methods for the positive mass theorem, including He--Shi--Yu, Bi--Hao--He--Shi--Zhu, and Brendle--Wang, we develop a Schoen--Yau type singular dimension descent method for positive scalar curvature obstructions in arbitrary dimensions.  We prove obstructions to positive scalar curvature on enlargeable manifolds and establish the corresponding cubical width inequalities and two-systole estimates.  The method also applies to enlargeable AM--PI spaces, giving a positive scalar curvature obstruction when the singular set has Assouad codimension greater than \(3-2/n\).
\end{abstract}

\maketitle

\section{Introduction}
\label{sec:introduction}

The Geroch conjecture says that the torus \(\mathbb T^n\) admits no metric of positive scalar curvature.  Schoen--Yau \cite{SchoenYau1979b} proved this for \(n\le 7\) using minimal hypersurface descent, and Gromov--Lawson \cite{GromovLawson1980} proved the all-dimensional case using the spin method. 

In this paper, we consider the corresponding extension of the Geroch theorem to the overtorical setting.  Following Gromov \cite{Gromov2018}, a closed oriented \(n\)-manifold is called \emph{overtorical} if it admits a continuous map
\[
   F:M\longrightarrow \mathbb T^n
\]
of nonzero degree.  We prove the following result.

\begin{theoremA}
Let \(M^n\) be a closed oriented overtorical manifold.  Then \(M\) admits no metric of positive scalar curvature.
\end{theoremA}

A basic family of examples comes from connected sums.  If \(N\) is any closed oriented \(n\)-manifold, then
\[
    \mathbb T^n\# N
\]
is overtorical, since collapsing the \(N\)-summand gives a degree-one map to \(\mathbb T^n\).  Thus Theorem A, together with Lohkamp's compactification argument \cite{Lohkamp1999}, implies the corresponding positive mass theorem in dimension \(n\). 

The proof belongs to the Schoen--Yau minimal hypersurface tradition \cite{SchoenYau1979a,SchoenYau1979b}.  In that argument one constructs a suitable area-minimizing hypersurface.  Its stability inequality, together with the Gauss equation, passes the scalar-curvature inequality to the hypersurface, and the construction is then repeated in one lower dimension.  The iteration ends with a contradiction to the Gauss--Bonnet theorem. When the ambient dimension is at most seven, the minimizing hypersurfaces are smooth and this descent argument can be carried out directly.  In higher dimensions they may have singular sets.  An inductive use of the method must then control the singular set just produced, as well as possible contacts with singular sets from earlier steps. Schoen and Yau \cite{SchoenYau2022} proposed a singular dimension descent for extending the minimal hypersurface method through these singular sets. Lohkamp developed another treatment of minimal hypersurface singularities, see \cite{LohkampSecret}.

Another approach to the higher-dimensional problem is to avoid singularities in the descent by perturbing the ambient metric so that the minimizing hypersurfaces are smooth. Such generic regularity was established by Smale~\cite{Smale1993} in ambient dimension \(8\). Chodosh--Mantoulidis--Schulze~\cite{ChodoshMantoulidisSchulze} extended this to dimensions \(9\) and \(10\), and Chodosh--Mantoulidis--Schulze--Wang~\cite{ChodoshMantoulidisSchulzeWang} to dimension \(11\). For an excellent survey of this approach, see~\cite[Section~14]{ChodoshICM}.

The approach of the present paper is a different singular descent strategy, closely related to the recent conformal blow-up methods for the positive mass theorem. He--Shi--Yu \cite{HeShiYu2026}  used conformal blow-up to continue the Schoen--Yau dimension reduction on asymptotically flat manifolds with arbitrary ends \cite{LesourdUngerYau,ZhuArbitraryEnds}.  Bi--Hao--He--Shi--Zhu \cite{BiHaoHeShiZhu2026} further developed this approach and proved the Riemannian positive mass theorem up to dimension \(19\). The same work also established the Geroch conjecture up to dimension \(12\). Brendle--Wang \cite{BrendleWang2026} extended the positive mass theorem to arbitrary dimensions by introducing a new descent datum that is stable under more general conformal changes.

In this approach, one constructs an area-minimizing hypersurface, or more generally a soap bubble  (\(\mu\)-bubble), in the regular part.  If a singular set is formed, one applies a conformal blow-up along this set and pushes it to infinity. The next dimension reduction step then proceeds on the resulting complete smooth manifold, completely avoiding any interaction with the old singular set.

For the overtorical problem, the descent cannot be carried out purely in the regular part.  This obstacle is illustrated by the following simple model. Suppose that a first descent step on \(\mathbb T^9\) produces an area-minimizing current \(T\) in a coordinate \(\mathbb T^8\)-class.  Imagine we are in the bad situation where \(\operatorname{spt}T\) is homeomorphic to \(\mathbb T^8\), and its singular set \(S\) is the image of the \(1\)-skeleton in the standard cubical quotient
\[
   \mathbb T^8=[0,1]^8/\!\sim .
\]
Although \(T\) still represents the required nonzero homology class in \(\mathbb T^9\), its regular part 
\(\operatorname{reg}T=\operatorname{spt}T\setminus S\) satisfies
\[
   H_7(\operatorname{reg}T;\mathbb Z)=0 .
\]
This means that the regular part alone cannot preserve the necessary topological information from which to start the next minimizing step.

To overcome this issue, we combine {\it transversality} with {\it cylindrical conformal blow-ups} to carry the singular set through the descent. Returning to the example, choose a regular level set
\[
   L=\{\theta=a\}\subset \mathbb T^9
\]
of another torus coordinate, transverse to both \(\operatorname{reg}T\) and \(S\).  Then \(L\cap\operatorname{reg}T\) is a smooth \(7\)-dimensional hypersurface in \(\operatorname{reg}T\), and its closure meets \(S\) only in finitely many points.

After this transverse choice, we blow up the metric along \(S\) via a conformal deformation.  The novelty here is the {\it cylindrical rate} of the conformal factor.  Near \(S\), the resulting metric is locally comparable to \(d(\cdot,S)^{-2}g\). The packing estimates of Naber--Valtorta~\cite{NaberValtorta2020} play an essential role here, particularly in proving the upper bound. This specific rate simultaneously pushes \(S\) to infinite distance and makes \(\operatorname{reg}T\) complete, while preserving the weighted scalar-curvature lower bound up to a prescribed small loss.

In the resulting metric, \(L\cap\operatorname{reg}T\) is proper. We take a finite metric band around it and solve a weighted \(\mu\)-bubble problem there. This can keep both curvature and topological information. On the one hand, the stability inequality passes the same lower bound to the \(\mu\)-bubble, with an additional error bounded by the inverse square of the band width. After passing to a suitable lift, the width can be made arbitrarily large, so this error is arbitrarily small. On the other hand, roughly speaking, the closure of this $\mu$-bubble is homologous to $L$.

The \(\mu\)-bubble may still have singular set but with non-decreasing codimension. Since it was constructed in the blown-up band, its closure can approach the old singular set only at the finitely many points fixed in the previous step. New singular set may form in the $\mu$-bubble as well. In this example, the \(\mu\)-bubble is \(7\)-dimensional, so its new singular set is at most countable. We can therefore choose a further regular level set of another torus coordinate avoiding both the old intersection points and the new singular set. We then blow up these new singularities in the same way, again preserving the weighted scalar-curvature lower bound up to an arbitrarily small error. The chosen level set avoids all singularities carried from the previous steps, so the next stage is the usual smooth Schoen--Yau descent. We can then iterate the argument.

We now specify the weighted scalar curvature used in the descent. It is close in spirit to the descent datum of Brendle--Wang~\cite{BrendleWang2026}.  For an \(m\)-dimensional Riemannian manifold \((X^m,g)\) and a parameter \(\lambda\ne 1/m\), set
\[
  a_m(\lambda)=\frac{1-(m-1)\lambda}{1-m\lambda},
  \qquad
  \Scal_f^{m,\lambda}(g)
  =
  R_g+2\Delta_g f-a_m(\lambda)|df|_g^2 .
\]
For \(\lambda=1/m\), we always take \(f\) to be constant and set
\[
   \Scal_f^{m,1/m}(g):=R_g .
\]
If \(\lambda\ne0\), write
\[
   k=\lambda^{-1}-m .
\]
Then
\[
  a_m(\lambda)=\frac{k+1}{k},
  \qquad
  \Scal_f^{m,\lambda}(g)
  =
  R_g+2\Delta_g f-\frac{k+1}{k}|df|_g^2 .
\]
When \(k\) is a positive integer, this is the scalar curvature of the warped product
\[
   g+e^{-2f/k}h
\]
with scalar-flat \(k\)-dimensional fiber. In this case \(\lambda^{-1}=m+k\) is the total dimension. More generally,
\(\lambda^{-1}\) may be viewed as a formal effective dimension, so that \(\lambda<0\) corresponds to negative effective dimension.

The same method proves a cubical width inequality over enlargeable bases.  Recall that a complete oriented Riemannian manifold \((Y^r,h_Y)\) is enlargeable, in the sense of Gromov--Lawson~\cite{GromovLawson1980}, if for every \(\varepsilon>0\), there are a Riemannian cover \(\widehat Y\to Y\) and a smooth \(\varepsilon\)-Lipschitz map
\[
   \widehat Y\to S^r
\]
which is constant outside a compact set and has nonzero degree.

Write \(I=[-1,1]\) and set
\[
  \Lambda_n:=
  \begin{cases}
    \displaystyle \frac1n, & n\le13,\\[0.8em]
    \displaystyle
    \frac{51-2n}{23n+26}, & n\ge14 .
  \end{cases}
\]

\begin{theoremB}\label{thm:complete-enlargeable-base}
Let \(Y^r\) be either a point, in which case \(r=0\), or a complete oriented enlargeable Riemannian \(r\)-manifold.  Let \(k\ge0\), set \(n=r+k\), and let \((X^n,\partial X,g)\) be a complete oriented smooth Riemannian manifold with boundary.  Let
\[
   F=(F_Y,\tau_1,\ldots,\tau_k):X\to Y\times I^k
\]
be proper, continuous, and smooth on \(X^\circ\), and of nonzero relative degree as a map
\[
   (X,\partial X)\to(Y\times I^k,Y\times\partial I^k).
\]
Assume \(F_Y\) is globally Lipschitz.  For \(i=1,\ldots,k\), set
\[
   B_i^\pm
   =
   F^{-1}\bigl(Y\times I^{i-1}\times\{\pm1\}\times I^{k-i}\bigr),
   \qquad
   d_i=\operatorname{dist}_g(B_i^-,B_i^+).
\]
Assume \(\lambda\le\Lambda_n\), and \(\Scal_f^{n,\lambda}(g)\ge\sigma\) on \(X^\circ\).  Then
\[
   \sigma
   \le
   4\pi^2(1-\lambda)\sum_{i=1}^{k}d_i^{-2}.
\]
For \(k=0\), the sum is empty, and the conclusion is \(\sigma\le0\).
\end{theoremB}

Taking \(k=0\) gives the positive scalar curvature obstruction for enlargeable manifolds, and taking \(Y=\mathbb T^n\) in this case gives Theorem~A.

The next consequence is a \(2\)-systole estimate. Let \(\omega_{S^2}\in H^2(S^2;\mathbb Z)\) be the generator satisfying
\[
   \int_{S^2}\omega_{S^2}=1 .
\]
For a map \(\rho:X\to S^2\), set \(u_\rho=\rho^*\omega_{S^2}\in H^2(X;\mathbb Z)\). For a compact Riemannian manifold \(X\), define
\[
   \operatorname{sys}_2(X,u_\rho;g)
   =
   \inf\{\mathbf M_g(T):T\text{ is an integral }2\text{-cycle in }X,\
   \langle u_\rho,[T]\rangle\ne0\}.
\]

\begin{theoremC}\label{thm:two-systole}
Let \((X^n,\partial X,g)\) be a compact oriented smooth Riemannian manifold with boundary, \(n\ge2\), and let
\[
   F=(\rho,\tau_1,\ldots,\tau_{n-2}):X\to S^2\times I^{n-2}
\]
be continuous, smooth on \(X^\circ\), and of nonzero relative degree as a map
\[
   (X,\partial X)\to(S^2\times I^{n-2},S^2\times\partial I^{n-2}).
\]
For \(i=1,\ldots,n-2\), set
\[
   B_i^\pm
   =
   F^{-1}\bigl(S^2\times I^{i-1}\times\{\pm1\}\times I^{n-2-i}\bigr),
   \qquad
   d_i=\operatorname{dist}_g(B_i^-,B_i^+).
\]
Assume \(\lambda\le\Lambda_n\), and \(\Scal_f^{n,\lambda}(g)\ge\sigma\).  Denote \(\sigma_*= \sigma-4\pi^2(1-\lambda)\sum_{i=1}^{n-2}d_i^{-2}\). If \(\sigma_*>0\), then
\[
   \operatorname{sys}_2(X,u_\rho;g)\le\frac{8\pi}{\sigma_*}.
\]
\end{theoremC}

In this singular dimension descent, \(\mu\)-bubbles are naturally treated as local almost-manifold PI spaces (AM--PI spaces).  This viewpoint leads to a positive scalar curvature obstruction for such singular spaces themselves. Roughly speaking, an AM--PI space is a metric measure space
\[
   X=\mathcal R\sqcup\mathcal S
\]
whose regular part \(\mathcal R\) is a smooth Riemannian manifold with metric \(g\), and \(X\) is locally Ahlfors regular and supports a local Poincar\'e inequality.  The precise definition is given in Section~\ref{sec:conformal-blowup-exceptional}. As part of the definition, we assume that the underlying measure restricted to \(\mathcal R\) takes the form \(e^{-f}d\mu_g\) for some smooth function \(f\).  For a closed set \(\mathcal S\subset X\), write
\[
   c_A^X(\mathcal S)=n-\dim_A^X(\mathcal S)
\]
for its Assouad codimension. As in the smooth manifold case, we call \(X\) enlargeable if for every \(\varepsilon>0\), some cover of \(X\) admits an \(\varepsilon\)-Lipschitz map to \(S^n\), constant outside a compact set, of nonzero degree.

\begin{theoremD}
Let \(n\geq 3\) and
\[
   X=\mathcal R\sqcup\mathcal S
\]
be a compact connected oriented enlargeable \(n\)-dimensional AM--PI space. Assume that \(f\) is bounded on \(\mathcal R\).

If
\[
   c_A^X(\mathcal S)>3-\frac2n,
\]
then, for every \(\lambda\leq 1/n\), the weighted scalar curvature \(\Scal_f^{n,\lambda}(g)\) cannot be nonnegative on \(\mathcal R\) and positive on a nonempty open subset of \(\mathcal R\).

If, in addition,
\[
   c_A^X(\mathcal S)>3-\frac1{n-1},
\]
then \(\Scal_f^{n,\lambda}(g)\ge0\) on \(\mathcal R\) implies that  \(f\) is constant and that the length space \( X^\ell=\overline{(\mathcal R,d_g^\ell)}\) is a compact flat manifold.
\end{theoremD}

\begin{remark}
Consider the torical band
\[
   \left(
   \left(-\frac{\pi}{n},\frac{\pi}{n}\right)\times\mathbb T^{n-1},
   \; dt^2+\cos^{4/n}\!\left(\frac{nt}{2}\right)g_{\mathbb T^{n-1}}
   \right).
\]
It has scalar curvature \(n(n-1)\) and width \(2\pi/n\).  Near either end, with \(\rho=\pi/n-|t|\), the metric is asymptotic to
\[
   d\rho^2+\rho^{4/n}g_{\mathbb T^{n-1}},
\]
and the volume of an end-neighborhood is of order \(\rho^{3-2/n}\).  This suggests \(3-2/n\) is the natural threshold for the obstruction part.
\end{remark}

\begin{remark}
Around the time when the first version of this paper appeared, Wang--Wang--Xie \cite{WangWangXie2026} independently obtained related positive scalar curvature obstructions for \(L^\infty\)-metrics on tori, using spinorial methods.  In their setting, the size of the singular set is measured by Hausdorff dimension.  It is natural to ask whether the Assouad dimension hypothesis in Theorem~D can be weakened to a Hausdorff dimension hypothesis.
\end{remark}

The paper is organized as follows.  Section~\ref{sec:weighted-scalar-curvature} computes the conformal transformation law for the weighted scalar curvature and the descent inequality for weighted \(\mu\)-bubbles. Section~\ref{sec:conformal-blowup-exceptional} constructs conformal blow-ups near singular sets, preserving the weighted scalar curvature lower bound up to an arbitrarily small loss.  Section~\ref{sec:cube-inequality} proves Theorem~B in the case \(Y=\{\mathrm{pt}\}\).  Section~\ref{sec:enlargeable-bases} proves Theorems~A,~B, and~C. Section~\ref{sec:ampi-obstruction-rigidity} treats the positive scalar curvature obstruction and flat rigidity theorem for AM--PI spaces, completing the proof of Theorem~D.

\subsection*{Acknowledgements}
The second-named author is partially supported by National Key R\&D Program of China 2023YFA1009900, NSFC Grant No. 12401072, Zhejiang Provincial Natural Science Foundation of China under Grant No. LQKWL26A0101 and the start-up fund from Westlake University.

\section{Weighted scalar curvature}
\label{sec:weighted-scalar-curvature}

For any real coefficient \(a\), write
\begin{equation}\label{eq:coefficient-scalar-expression}
  \mathscr S_a^g(f):=R_g+2\Delta_g f-a|df|_g^2 .
\end{equation}
For \(\lambda<1/m\), set
\begin{equation}\label{eq:weighted-scalar-curvature}
  a_m(\lambda)
  :=
  \frac{1-(m-1)\lambda}{1-m\lambda},
  \qquad
  \Scal_f^{m,\lambda}(g)
  :=
  \mathscr S_{a_m(\lambda)}^g(f).
\end{equation}
The two useful identities used below are
\begin{equation}\label{eq:parameter-shift}
  a_{m-1}(\lambda)
  =
  2-\frac1{a_m(\lambda)}
\end{equation}
and
\begin{equation}\label{eq:a-difference}
  a_m(\lambda)-a_{m-1}(\lambda)
  =
  \frac{\lambda^2}{(1-m\lambda)(1-(m-1)\lambda)}
  \ge0 .
\end{equation}
Thus a lower bound for \(\Scal_f^{m,\lambda}(g)\) also gives the same lower bound for \(\mathscr S_{a_{m-1}(\lambda)}^g(f)\).

\subsection{Conformal changes}

Let \((N^m,g)\) be smooth, \(m\ge3\), and let \(w>0\).  We consider
\[
   \widetilde g=w^\alpha g,
   \qquad
   \widetilde f=f+\beta\log w .
\]
All quantities on the right hand side are computed with respect to \(g\).

\begin{lemma}\label{lem:general-conformal-formula-weighted-scalar}
For any real \(a\),
\begin{equation}\label{eq:general-conformal-formula-f}
\begin{split}
w^\alpha\mathscr S_a^{\widetilde g}(\widetilde f)
&=
\mathscr S_a^g(f)
+\bigl(2\beta-(m-1)\alpha\bigr)w^{-1}\Delta_g w  \\
&\quad
+\bigl(\alpha(m-2)-2a\beta\bigr)
 w^{-1}\langle dw,df\rangle_g  \\
&\quad
+\Bigl[
 (m-1)\alpha-\frac{(m-1)(m-2)}4\alpha^2
 -2\beta+\alpha\beta(m-2)-a\beta^2
 \Bigr]w^{-2}|dw|_g^2 .
\end{split}
\end{equation}
\end{lemma}

\begin{proof}
Write
\[
   \widetilde g=e^{2\varphi}g,
   \qquad
   \varphi=\frac{\alpha}{2}\log w .
\]
The scalar curvature transformation formula gives
\[
\begin{split}
  w^\alpha R_{\widetilde g}
  &=
  R_g-(m-1)\alpha w^{-1}\Delta_g w  \\
  &\quad+
  \left(
  (m-1)\alpha
  -
  \frac{(m-1)(m-2)}4\alpha^2
  \right)w^{-2}|dw|_g^2 .
\end{split}
\]
Moreover,
\[
\begin{split}
  w^\alpha\Delta_{\widetilde g}\widetilde f
  &=
  \Delta_g f
  +\beta w^{-1}\Delta_g w
  +\frac{\alpha(m-2)}2w^{-1}\langle dw,df\rangle_g  \\
  &\quad+
  \left(
  -\beta+\frac{\alpha\beta(m-2)}2
  \right)w^{-2}|dw|_g^2 ,
\end{split}
\]
and
\[
  w^\alpha |d\widetilde f|_{\widetilde g}^2
  =
  |df|_g^2
  +2\beta w^{-1}\langle dw,df\rangle_g
  +\beta^2w^{-2}|dw|_g^2 .
\]
Substituting these identities into
\[
   \mathscr S_a^{\widetilde g}(\widetilde f)
   =
   R_{\widetilde g}
   +2\Delta_{\widetilde g}\widetilde f
   -a|d\widetilde f|_{\widetilde g}^2
\]
gives \eqref{eq:general-conformal-formula-f}.
\end{proof}

We now specialize to \(a=a_m(\lambda)\).  For \(\lambda<1/m\), define
\[
  q_m(\lambda)
  :=
  \sqrt{
  \frac{(m-1)(1-(m-1)\lambda)}
  {1-\lambda}
  } .
\]
Then \(q_m(\lambda)>1\), and
\[
  a_m(\lambda)
  =
  \frac{(m-2)q_m(\lambda)^2}
  {(m-1)(q_m(\lambda)^2-1)} .
\]
Set
\begin{equation}\label{eq:optimal-alpha-beta}
  \alpha_m(\lambda)
  :=
  \frac{2(1+q_m(\lambda))}{m-2},
  \qquad
  \beta_m(\lambda)
  :=
  \frac{m-1}{m-2}
  \left(
  q_m(\lambda)-\frac1{q_m(\lambda)}
  \right).
\end{equation}
Put
\begin{equation}\label{eq:pointwise-K-gamma}
   \mathcal K_m(\lambda)
   :=
   \frac{\alpha_m(\lambda)(m-1)}{q_m(\lambda)},
   \qquad
   \gamma_m(\lambda)
   :=
   \frac{2}{\mathcal K_m(\lambda)} .
\end{equation}
We write
\begin{equation}\label{eq:lambda-drift-operator}
  L_{f,\lambda}^{(m)}u
  :=
  \Delta_g u
  -
  \gamma_m(\lambda)
  \langle\nabla u,\nabla f\rangle_g .
\end{equation}

\begin{proposition}\label{prop:lambda-conformal-formula}
Let \(\lambda< 1/m\), and set
\[
   \alpha=\alpha_m(\lambda),
   \qquad
   \beta=\beta_m(\lambda),
   \qquad
   q=q_m(\lambda).
\]
If
\[
   \widetilde g=w^\alpha g,
   \qquad
   \widetilde f=f+\beta\log w ,
\]
then
\begin{equation}\label{eq:lambda-conformal-formula}
  w^\alpha
  \Scal_{\widetilde f}^{m,\lambda}(\widetilde g)
  =
  \Scal_f^{m,\lambda}(g)
  -
  \mathcal K_m(\lambda)
  w^{-1}L_{f,\lambda}^{(m)}w.
\end{equation}
\end{proposition}

\begin{proof}
With the above choice of \(\alpha\) and \(\beta\),
\[
  2\beta-(m-1)\alpha
  =
  -\frac{\alpha(m-1)}{q},
\]
and
\[
  \alpha(m-2)-2a_m(\lambda)\beta=2.
\]
It remains only to check that the coefficient of \(w^{-2}|dw|^2\) in \eqref{eq:general-conformal-formula-f} vanishes.   The first two terms give
\[
  (m-1)\alpha-\frac{(m-1)(m-2)}4\alpha^2
  =
  -\frac{m-1}{m-2}(q^2-1),
\]
while the next two terms give
\[
  -2\beta+\alpha\beta(m-2)
  =
  \frac{2(m-1)}{m-2}(q^2-1).
\]
Thus the four terms not involving \(a_m(\lambda)\beta^2\) combine to
\[
  \frac{m-1}{m-2}(q^2-1).
\]
On the other hand,
\[
  a_m(\lambda)\beta^2
  =
  \frac{m-1}{m-2}(q^2-1).
\]
Hence the coefficient of \(w^{-2}|dw|^2\) is zero.  Substituting the two remaining coefficients into \eqref{eq:general-conformal-formula-f} gives \eqref{eq:lambda-conformal-formula}.
\end{proof}

We shall also use a quadratic-form version of the conformal formula.  This quadratic form is the descent datum used in Brendle--Wang \cite{BrendleWang2026}.  

Let \((N^m,g)\) be smooth, \(m\ge3\).  For a smooth function \(f\), a constant \(\sigma\), and a real number \(a\), define
\begin{equation}\label{eq:quadratic-form-descent-datum}
   Q_{a,\sigma}^{g,f}(\varphi)
   :=
   \int_N e^{-f}|\nabla\varphi|_g^2\,d\mu_g
   +
   \frac12
   \int_N e^{-f}
   \bigl(\mathscr S_a^g(f)-\sigma\bigr)\varphi^2\,d\mu_g .
\end{equation}

Assume \(\lambda<\frac1{m+1}\). Define
\begin{equation}\label{eq:qbar-m-lambda}
   \overline q_m(\lambda)
   :=
   \sqrt{
      \frac{m(1-(m-1)\lambda)}
           {2(1-\lambda)}
   } .
\end{equation}
Set
\begin{equation}\label{eq:qf-alpha-beta}
   \overline\alpha_m(\lambda)
   :=
   \frac{2(1+\overline q_m(\lambda))}{m-2},
   \qquad
   \overline\beta_m(\lambda)
   :=
   \frac{\overline q_m(\lambda)}
        {2a_{m+1}(\lambda)-1}.
\end{equation}

Put
\begin{equation}\label{eq:qf-s-N}
   \overline s_m(\lambda)
   :=
   \frac12
   \left(
      \frac{\overline\alpha_m(\lambda)(m-2)}2
      -
      \overline\beta_m(\lambda)
   \right),
   \qquad
   \overline{\mathcal K}_m(\lambda)
   :=
   \frac{\overline\alpha_m(\lambda)m}{2}
   -
   \overline\beta_m(\lambda).
\end{equation}
 We write
\begin{equation}\label{eq:qf-drift-operator}
   \overline L_{f,\lambda}^{(m)}u
   :=
   \Delta_g u
   -
   \overline\gamma_m(\lambda)
   \langle\nabla u,\nabla f\rangle_g,\qquad \overline\gamma_m(\lambda)
   :=
   \frac{1}{\overline{\mathcal K}_m(\lambda)}.
\end{equation}

\begin{proposition}
\label{prop:qf-conformal-formula}
Let \(w>0\) be smooth, and set
\[
   \widetilde g=w^{\overline\alpha_m(\lambda)}g,
   \qquad
   \widetilde f=f+\overline\beta_m(\lambda)\log w .
\]
Then, for every \(\varphi\in C_c^\infty(N)\),
\begin{equation}\label{eq:qf-conformal-formula}
\begin{split}
Q_{a_{m+1}(\lambda),\sigma}^{\widetilde g,\widetilde f}
\left(w^{-\overline s_m(\lambda)}\varphi\right)  = &
Q_{a_{m+1}(\lambda),\sigma}^{g,f}(\varphi)
-
\frac{\overline{\mathcal K}_m(\lambda)}2
\int_N e^{-f}
w^{-1}\overline L_{f,\lambda}^{(m)}w\,\varphi^2\,d\mu_g  \\
&
-
\frac12
\int_N e^{-f}\sigma
\bigl(w^{\overline\alpha_m(\lambda)}-1\bigr)
\varphi^2\,d\mu_g .
\end{split}
\end{equation}
\end{proposition}

\begin{proof}
Write
\[
   \overline q=\overline q_m(\lambda),
   \qquad
   \overline\alpha=\overline\alpha_m(\lambda),
   \qquad
   \overline\beta=\overline\beta_m(\lambda),
   \qquad
   \overline s=\overline s_m(\lambda),
   \qquad
   \overline{\mathcal K}=\overline{\mathcal K}_m(\lambda),
\]
and \(a=a_{m+1}(\lambda)\).  By definition,
\[
   2\overline s
   =
   \frac{\overline\alpha(m-2)}2-\overline\beta .
\]
Thus
\[ \int_N e^{-\widetilde f}
\left|
\nabla_{\widetilde g}
\left(w^{-\overline s}\varphi\right)
\right|_{\widetilde g}^2
d\mu_{\widetilde g}  =\int_N e^{-f} w^{2\overline s}
\left| \nabla_g \left(w^{-\overline s}\varphi\right) \right|_g^2 d\mu_g .
\]
Expanding the integrand,
\[
   w^{2\overline s}
   \left|
   \nabla\left(w^{-\overline s}\varphi\right)
   \right|^2
   =
   |\nabla\varphi|^2
   -
   \overline s w^{-1}\langle dw,\nabla(\varphi^2)\rangle
   +
   \overline s^2w^{-2}|dw|^2\varphi^2 .
\]
Integrating the middle term by parts with respect to \(e^{-f}d\mu_g\), we get
\begin{equation}\label{eq:qf-gradient-part}
\begin{split}
&\int_N e^{-\widetilde f}
\left|
\nabla_{\widetilde g}
\left(w^{-\overline s}\varphi\right)
\right|_{\widetilde g}^2
d\mu_{\widetilde g}\\
=& \int_N e^{-f}|\nabla\varphi|^2\,d\mu_g +
\int_N e^{-f}
\left[
   \overline s w^{-1}\Delta_g w
   -
   \overline s w^{-1}\langle dw,df\rangle_g
   +
   (\overline s^2-\overline s)w^{-2}|dw|^2
\right]\varphi^2\,d\mu_g .
\end{split}
\end{equation}

For the zero-order part, note that
\[
   \frac{\overline\alpha m}{2}
   -
   \overline\beta
   -
   2\overline s
   =
   \overline\alpha .
\]
Therefore
\[ \frac12 \int_N e^{-\widetilde f} \bigl( \mathscr S_a^{\widetilde g}(\widetilde f)-\sigma\bigr) w^{-2\overline s}\varphi^2\,d\mu_{\widetilde g} = \frac12 \int_N e^{-f}
\left( w^{\overline\alpha} \mathscr S_a^{\widetilde g}(\widetilde f) - \sigma w^{\overline\alpha}\right)\varphi^2\,d\mu_g . \]
Subtracting the original zero-order term gives
\[\frac12 \int_N e^{-f}\left( w^{\overline\alpha}\mathscr S_a^{\widetilde g}(\widetilde f)
   -\mathscr S_a^g(f)\right)\varphi^2\,d\mu_g -\frac12 \int_N e^{-f}\sigma\bigl(w^{\overline\alpha}-1\bigr)\varphi^2\,d\mu_g .\]
We now apply the general conformal formula \eqref{eq:general-conformal-formula-f}, in dimension \(m\), with \(a=a_{m+1}(\lambda)\).

The coefficient of \(w^{-1}\Delta_gw\) in the total quadratic form is
\[
   \overline s
   +
   \frac12\bigl(2\overline\beta-(m-1)\overline\alpha\bigr)
   =
   -\frac12
   \left(
      \frac{\overline\alpha m}{2}-\overline\beta
   \right)
   =
   -\frac{\overline{\mathcal K}}{2}.
\]
The coefficient of \(w^{-1}\langle dw,df\rangle_g\) is
\[
   -\overline s
   +
   \frac12\bigl(\overline\alpha(m-2)-2a\overline\beta\bigr).
\]
Using
\[
   \overline\alpha(m-2)=2(1+\overline q),
   \qquad
   \overline\beta=\frac{\overline q}{2a-1},
\]
this coefficient becomes
\[
   -\frac12(1+\overline q-\overline\beta)
   +
   (1+\overline q)-a\overline\beta
   =
   \frac12 .
\]
Thus the first-order terms combine to
\[
   -\frac{\overline{\mathcal K}}{2}
   w^{-1}
   \left(
      \Delta_g w
      -
      \frac1{\overline{\mathcal K}}
      \langle dw,df\rangle_g
   \right)
   =
   -\frac{\overline{\mathcal K}}{2}
   w^{-1}\overline L_{f,\lambda}^{(m)}w .
\]

It remains to check the \(w^{-2}|dw|^2\)-coefficient.  Since
\[
   \overline\alpha(m-2)=2(1+\overline q),
   \qquad
   \overline\beta=\frac{\overline q}{2a-1},
   \qquad
   \overline s=\frac12(1+\overline q-\overline\beta),
\]
the coefficient is
\[
\begin{split}
\mathcal C
&=
\overline s^2-\overline s
+
\frac12
\left[
 (m-1)\overline\alpha
 -\frac{(m-1)(m-2)}4\overline\alpha^2
 -2\overline\beta
 +\overline\alpha\overline\beta(m-2)
 -a\overline\beta^2
\right]  \\
&=
\frac{m}{4(m-2)}(1-\overline q^2)
+
\frac{\overline q^2}{4(2a-1)} .
\end{split}
\]
Using
\[
   2a-1=\frac{1-(m-1)\lambda}{1-(m+1)\lambda},
   \qquad
   \overline q^2=
   \frac{m(1-(m-1)\lambda)}{2(1-\lambda)},
\]
we get
\[
   1-\overline q^2
   =
   -\frac{(m-2)(1-(m+1)\lambda)}{2(1-\lambda)},
   \qquad
   \frac{\overline q^2}{2a-1}
   =
   \frac{m(1-(m+1)\lambda)}{2(1-\lambda)} .
\]
Therefore \(\mathcal C=0\). This proves \eqref{eq:qf-conformal-formula}.
\end{proof}

\begin{remark}\label{rem:weighted-to-ordinary-scalar}
For \(\widehat g=e^{-2f/(m-1)}g\), a direct conformal calculation gives
\[
  R_{\widehat g}
  =
  e^{2f/(m-1)}
  \left(
    \Scal_f^{m,\lambda}(g)
    +
    \frac{1-\lambda}{(m-1)(1-m\lambda)}|df|_g^2
  \right)
\]
for finite \(\lambda\).  Thus, when \(\lambda<\frac{1}{m}\) or \(\lambda\ge1\), positivity of \(\Scal_f^{m,\lambda}(g)\) implies positive scalar curvature of
\(\widehat g\).  The same conclusion holds at \(\lambda=\infty\) under the convention \(a_m(\infty)=(m-1)/m\).
\end{remark}

\subsection{Stability under Descent}
\label{subsec:stability-under-descent}

Let \((N^m,g)\) be smooth, and let
\[
   \Sigma^{m-1}\subset N
\]
be an oriented two-sided smooth hypersurface, with chosen unit normal \(\nu\). We use the convention $H=\mathrm{div}_\Sigma\nu$ for mean curvature. Let \(f\in C^\infty(N)\), set
\[
   F=f|_\Sigma,
\]
and let \(\Phi\) be smooth near \(\Sigma\).  Assume
\begin{equation}\label{eq:weighted-mu-stationary}
   H-\langle\nabla f,\nu\rangle=\Phi
\end{equation}
on \(\Sigma\), and assume the weighted stability inequality
\begin{equation}\label{eq:weighted-stability-inequality}
\begin{split}
\int_\Sigma e^{-F}|\nabla^\Sigma\varphi|^2
\ge
\int_\Sigma e^{-F}
\bigl(
  \operatorname{Ric}(\nu,\nu)
  +|A_\Sigma|^2
  +D^2f(\nu,\nu)
  +\langle\nabla\Phi,\nu\rangle
\bigr)\varphi^2
\end{split}
\end{equation}
for all \(\varphi\in C_c^\infty(\Sigma)\).  

\begin{proposition}\label{prop:smooth-descent-pointwise}
Assume \(\lambda<1/m\) and
\[
   \mathscr S_{a_m(\lambda)}^g(f)\ge \sigma
\]
near \(\Sigma\).  Define
\begin{equation}\label{eq:descended-sigma-lambda}
  \sigma_\Sigma
  :=
  \sigma|_\Sigma
  +2\langle\nabla\Phi,\nu\rangle
  +\frac{1}{1-\lambda}\Phi^2 .
\end{equation}
Then the quadratic form
\begin{equation}\label{eq:descended-schrodinger-form}
\begin{split}
Q_\Sigma(\varphi)
:=
\int_\Sigma e^{-F}|\nabla^\Sigma\varphi|^2
+
\frac12\int_\Sigma e^{-F}
\bigl(
  \mathscr S_{a_m(\lambda)}^{g_\Sigma}(F)-\sigma_\Sigma
\bigr)\varphi^2
\end{split}
\end{equation}
is nonnegative on \(C_c^\infty(\Sigma)\). 
\end{proposition}

\begin{proof}
Set
\[
   u=\langle\nabla f,\nu\rangle .
\]
By the stability inequality \eqref{eq:weighted-stability-inequality}, it is enough to prove
\begin{equation}\label{eq:I-estimate-needed}
I:=
\operatorname{Ric}(\nu,\nu)
+|A_\Sigma|^2
+D^2f(\nu,\nu)
+\langle\nabla\Phi,\nu\rangle 
+\frac12\mathscr S_{a_m(\lambda)}^{g_\Sigma}(F)
\ge
\frac12\sigma_\Sigma .
\end{equation}

With our convention for \(H\), the Gauss equation and the restriction formula are
\[
   R_{g_\Sigma}
   =
   R_g-2\operatorname{Ric}(\nu,\nu)+H^2-|A_\Sigma|^2
\]
and
\[
   \Delta_\Sigma F
   =
   \Delta_g f-D^2f(\nu,\nu)-Hu .
\]
Combining these identities gives
\begin{equation}\label{eq:I-expanded}
I=
\frac12(R_g+2\Delta_g f)
-\frac{a_m(\lambda)}2|\nabla^\Sigma F|^2
+\frac12|A_\Sigma|^2
+\frac12H^2-Hu
+\langle\nabla\Phi,\nu\rangle .
\end{equation}
Using
\[
   R_g+2\Delta_g f-a_m(\lambda)|\nabla f|^2\ge\sigma
\]
and
\[
   |\nabla f|^2=|\nabla^\Sigma F|^2+u^2,
\]
we get
\[
  \frac12(R_g+2\Delta_g f)
  -\frac{a_m(\lambda)}2|\nabla^\Sigma F|^2
  \ge
  \frac\sigma2+\frac{a_m(\lambda)}2u^2 .
\]
Since \(|A_\Sigma|^2\ge H^2/(m-1)\), it follows that
\begin{equation}\label{eq:I-lower-before-square}
I\ge
\frac\sigma2+\langle\nabla\Phi,\nu\rangle
+\frac{a_m(\lambda)}2u^2
+\frac{m}{2(m-1)}H^2-Hu .
\end{equation}
Using \(H=u+\Phi\), the last three terms are
\[
\frac{a_m(\lambda)}2u^2+\frac{m}{2(m-1)}(u+\Phi)^2-(u+\Phi)u .
\]
Set
\[
  B_{m,\lambda}
  :=
  a_m(\lambda)-\frac{m-2}{m-1}
  =
  \frac{1-\lambda}{(m-1)(1-m\lambda)}
  >0 .
\]
Then
\begin{equation}\label{eq:lambda-square-completion-correct}
\begin{split}
&\frac{a_m(\lambda)}2u^2
+\frac{m}{2(m-1)}(u+\Phi)^2
-(u+\Phi)u  \\
=&
\frac{B_{m,\lambda}}2
\left(
  u+\frac{\Phi}{(m-1)B_{m,\lambda}}
\right)^2
+\frac{1}{2(1-\lambda)}\Phi^2 .
\end{split}
\end{equation}
Therefore
\[
   I
   \ge
   \frac\sigma2
   +\langle\nabla\Phi,\nu\rangle
   +\frac{1}{2(1-\lambda)}\Phi^2
   =
   \frac12\sigma_\Sigma .
\]
This proves the nonnegativity of \(Q_\Sigma\).
\end{proof}

For a smooth manifold \((Y,g)\), a smooth function \(f\), and a parameter \(\lambda\leq 1/m\), define
\begin{equation}\label{eq:weighted-schrodinger-operator}
  \mathcal L_{m,\lambda,\sigma}^{g,f}(\cdot)
  :=
  -\Delta_g(\cdot)+\langle\nabla f,\nabla(\cdot)\rangle_g
  +\frac12\bigl(\mathscr S_{a_m(\lambda)}^g(f)-\sigma\bigr)(\cdot).
\end{equation}
Here \(m\) is only a parameter to determine the coefficient \(a_m(\lambda)\); the manifold \(Y\) need not have dimension \(m\).

\begin{lemma}\label{lem:positive-supersolution}
Let \(\psi>0\) be smooth.  Set
\[
  h=\log\psi,
  \qquad
  f_\psi=f-h=f-\log\psi .
\]
Then
\begin{equation}\label{eq:positive-supersolution-identity}
\begin{split}
\mathscr S_{a_{m-1}(\lambda)}^g(f_\psi)-\sigma 
=&
2\left(
  -\Delta_g h-|dh|_g^2+
  \langle df,dh\rangle_g
  +\frac12(\mathscr S_{a_m(\lambda)}^g(f)-\sigma)
  \right)  \\
&+
\frac1{a_m(\lambda)}
\left|dh+\bigl(a_m(\lambda)-1\bigr)df\right|_g^2 .
\end{split}
\end{equation}
Hence, if we suppose \(\mathcal L_{m,\lambda,\sigma}^{g,f}\psi\ge0\), then
\begin{equation}\label{eq:positive-supersolution-gives-pointwise}
  \mathscr S_{a_{m-1}(\lambda)}^g(f_\psi)\ge\sigma .
\end{equation}
\end{lemma}

\begin{proof}
Expanding \(f_\psi=f-h\) gives
\[
\begin{split}
\mathscr S_{a_{m-1}(\lambda)}^g(f_\psi)-\sigma
=&
\mathscr S_{a_m(\lambda)}^g(f)-\sigma
-2\Delta_g h  \\
&+
\bigl(a_m(\lambda)-a_{m-1}(\lambda)\bigr)|df|_g^2
+2a_{m-1}(\lambda)\langle df,dh\rangle_g
-a_{m-1}(\lambda)|dh|_g^2 .
\end{split}
\]
Using \(a_{m-1}(\lambda)=2-a_m(\lambda)^{-1}\), the last three terms can be written as
\[
  -2|dh|_g^2
  +2\langle df,dh\rangle_g
  +
  \frac1{a_m(\lambda)}
  \left|dh+\bigl(a_m(\lambda)-1\bigr)df\right|_g^2 .
\]
This gives \eqref{eq:positive-supersolution-identity}. Since \(\psi=e^h\), we have
\[
  \psi^{-1}\mathcal L_{m,\lambda,\sigma}^{g,f}\psi
  =
  -\Delta_g h-|dh|_g^2+
  \langle df,dh\rangle_g
  +\frac12\bigl(\mathscr S_{a_m(\lambda)}^g(f)-\sigma\bigr).
\]
If \(\mathcal L_{m,\lambda,\sigma}^{g,f}\psi\ge0\), then the first term on the right-hand side of \eqref{eq:positive-supersolution-identity} is nonnegative. Since \(a_m(\lambda)>0\), the square term is also nonnegative, and the conclusion follows.
\end{proof}

\section{Conformal blow-up near singular sets}
\label{sec:conformal-blowup-exceptional}

\subsection{AM--PI structure, weights, and packing}

Throughout this section we assume \(m\ge3\).  We use the almost-manifold convention as in Bi--Hao--He--Shi--Zhu~\cite{BiHaoHeShiZhu2026}.  Let \((X,d)\) be a complete  
metric space with a regular-singular decomposition
\[
   X=\mathcal R\sqcup\mathcal S,
\]
where \(\mathcal R\) is an \(m\)-dimensional smooth manifold equipped with a Riemannian metric \(g\), and \(\mathcal S\) is closed.
 
For an open set \(U\subset\mathcal R\), let \(d_{\ell,d}^{U}\) be the length metric on \(U\) induced by \(d\).  We assume that, for every
\(p\in\mathcal R\), there is an open neighborhood \(U\subset\mathcal R\) of \(p\) such that
\[
   d_{\ell,d}^{U}=d_g^U,
\]
where \(d_g^U\) is the intrinsic Riemannian distance in \(U\).

Let \(\omega\in C^\infty(\mathcal R)\) be positive, and define a measure \(\mu_\omega\) on \(X\) by
\[
   \mu_\omega(A)=\int_{A\cap\mathcal R}\omega\,d\mu_g ,
\]
where \(d\mu_g\) is the Riemannian volume measure of \((\mathcal R,g)\). In particular, \(\mu_\omega(\mathcal S)=0\).

\begin{definition}\label{def:admissible-weight-section3}
We say that \(\omega\) is an admissible weight, and write
\[
   \omega\in\mathcal W_{\rm adm}(X),
\]
if 
\[
   (X,d,\mathcal R,g,\mu_\omega)
\]
has AM--PI structure in the following compact-local sense.  For every compact \(K\Subset X\), there exist constants
\[
   C_K\ge1,
   \qquad
   \lambda_K\ge1,
   \qquad
   r_K>0,
\]
such that, for every \(x\in K\) and \(0<r<r_K\),
\begin{equation}\label{eq:local-ahlfors-section3}
   C_K^{-1}r^m
   \le
   \mu_\omega(B_r(x))
   \le
   C_Kr^m ,
\end{equation}
and for every Lipschitz function \(\varphi\) on \(B_{\lambda_K r}(x)\),
\begin{equation}\label{eq:local-pi-section3}
   \frac{1}{\mu_\omega(B_r(x))}
   \int_{B_r(x)}
   |\varphi-\varphi_{B_r(x),\omega}|\,d\mu_\omega
   \le
   C_Kr
   \left(
      \frac{1}{\mu_\omega(B_{\lambda_K r}(x))}
      \int_{B_{\lambda_K r}(x)\cap\mathcal R}
      |\nabla_g\varphi|^2\,d\mu_\omega
   \right)^{1/2}.
\end{equation}
Here
\[
   \varphi_{B_r(x),\omega}
   =
   \frac{1}{\mu_\omega(B_r(x))}
   \int_{B_r(x)}\varphi\,d\mu_\omega .
\]
When an admissible weight has been fixed, we say that \(X=\mathcal R\sqcup\mathcal S\) is an AM--PI space.
\end{definition}

For \(\omega\in\mathcal W_{\rm adm}(X)\), set
\[
   L_\omega u
   =
   \omega^{-1}\operatorname{div}_g(\omega\nabla_g u)
   \qquad\text{on }\mathcal R .
\]
Then, for \(u\in C^\infty(\mathcal R)\) and \(\phi\in C_c^\infty(\mathcal R)\),
\[
   \int_{\mathcal R}
   \langle\nabla_g u,\nabla_g\phi\rangle_g\,d\mu_\omega
   =
   \int_{\mathcal R}(-L_\omega u)\phi\,d\mu_\omega .
\]
\begin{lemma}\label{lem:bounded-weight-preserves-ampi-section3}
Let \(\omega\in\mathcal W_{\rm adm}(X)\), and let \( \eta\in C^\infty(\mathcal R)\), \(\eta>0 \). Assume  for every compact \(K\Subset X\), there are constants
\(0<c_{\eta,K}\le C_{\eta,K}<\infty\) such that
\[
   c_{\eta,K}\le \eta\le C_{\eta,K}
   \qquad\text{on }K\cap\mathcal R .
\]
Then \(\eta\omega\in\mathcal W_{\rm adm}(X)\).
\end{lemma}

\begin{proof}
On every compact \(K\Subset X\), the measures \(\mu_{\eta\omega}=\eta\mu_\omega\) and \(\mu_\omega\) are comparable.  Hence the local Ahlfors estimates follow immediately.

For the Poincar\'e inequality, let \(B=B_r(x)\) and \(\lambda B=B_{\lambda_K r}(x)\) be a ball on which the AM--PI constants for
\(d\mu_\omega\) are fixed.  Denote averages with respect to \(d\mu_\omega\) and \(d\mu_{\eta\omega}\) by \(u_{B,\omega}\) and \(u_{B,\eta\omega}\).  By the triangle inequality,
\[
   \int_B |u-u_{B,\eta\omega}|\,d\mu_{\eta\omega}
   \le
   2\int_B |u-u_{B,\omega}|\,d\mu_{\eta\omega}
   \le
   2C_{\eta,K}\int_B |u-u_{B,\omega}|\,d\mu_\omega .
\]
Dividing by \(\mu_{\eta\omega}(B)\ge c_{\eta,K}\mu_\omega(B)\), applying the Poincar\'e inequality for \(d\mu_\omega\), and using comparability again on \(\lambda B\), we obtain
\[
   \frac{1}{\mu_{\eta\omega}(B)}
   \int_B |u-u_{B,\eta\omega}|\,d\mu_{\eta\omega}
   \le
   Cr
   \left(
      \frac{1}{\mu_{\eta\omega}(\lambda B)}
      \int_{\lambda B\cap\mathcal R}
      |\nabla_g u|^2\,d\mu_{\eta\omega}
   \right)^{1/2}.
\]
Thus \(\eta\omega\) is admissible.
\end{proof}

We now record the local packing condition used below.

\begin{definition}\label{def:local-packing-condition-section3}
Let \(Z\subset X\) be closed and let \(0\le d\le m\).  We write \(\mathsf P_d(Z)\) if, for every compact \(K\Subset X\), there exist constants \(C_K<\infty\) and \(r_K>0\) such that
\begin{equation}\label{eq:local-packing-condition-section3}
   \#\bigl(P_s\cap B_X(x,R)\bigr)
   \le
   C_K\left(\frac Rs\right)^d
\end{equation}
whenever \(x\in K\), \(0<s<R<r_K\), and \(P_s\subset Z\cap K\) is \(s\)-separated.  If \(d=m-c\), we say that \(Z\) has local packing codimension \(c\). If we want to emphasize the ambient space $X$, we write \(\mathsf P_d^X(Z)\).
\end{definition}

Set
\begin{equation}\label{eq:packing-threshold-section3}
   d_{\rm pack}(Z)
   =
   \inf\bigl\{d: \mathsf P_d(Z)\text{ holds}\bigr\},
   \qquad
   \tau_{\rm pack}(Z)
   =
   m-2-d_{\rm pack}(Z).
\end{equation}
This is the packing formulation of the Assouad dimension \cite{Assouad1983}.  Thus
\[
   \dim_A(Z)=d_{\rm pack}(Z),
   \qquad
   \operatorname{codim}_A(Z)=m-\dim_A(Z).
\]
If we want to emphasize the ambient space $X$, we also write \(\dim_A^X(Z)\) and \(\tau_{\rm pack}^X(Z)\).

\begin{lemma}\label{lem:packing-zero-capacity-section3}
If \(\mathcal S\) satisfies \(\mathsf P_{m-2}(\mathcal S)\), then for every \(\omega\in\mathcal W_{\rm adm}(X)\) and every \(K\Subset U\Subset X\) there exist \(\chi_j\in W^{1,2}_0(U)\), \(0\le\chi_j\le1\), such that \(\chi_j=1\) near \(\mathcal S\cap K\) and
\[
   \int_U |\nabla\chi_j|^2\,d\mu_\omega\longrightarrow0 .
\]
In particular, \(\operatorname{Cap}_{2,\omega}(\mathcal S\cap K,U)=0\).
\end{lemma}

\begin{proof}
Fix \(K\Subset U\).  If \(\mathcal S\cap K=\varnothing\), there is nothing to prove. Choose \(\rho_0>0\) so that the \(2\rho_0\)-neighborhood of \(K\) is contained in \(U\) and the local Ahlfors and packing estimates apply there.  Then
\begin{equation}\label{eq:packing-tube-volume-section3}
   \mu_\omega\bigl(\{x\in U:\operatorname{dist}(x,\mathcal S\cap K)<\rho\}\bigr)
   \le C\rho^2,
   \qquad 0<\rho<\rho_0 .
\end{equation}

For \(0<\varepsilon<\rho_0/4\), set
\[
   l_\varepsilon=\log(\rho_0/\varepsilon),
   \qquad
   r_\mathcal S(x)=\operatorname{dist}(x,\mathcal S\cap K),
\]
and define
\[
   \chi_\varepsilon(x)
   =
   \begin{cases}
      1, & r_\mathcal S(x)\le\varepsilon,\\ \frac{\log(\rho_0/r_\mathcal S(x))}{l_\varepsilon},
         & \varepsilon<r_\mathcal S(x)<\rho_0,\\
      0, & r_\mathcal S(x)\ge\rho_0 .
   \end{cases}
\]
Then \(\chi_\varepsilon\in W^{1,2}_0(U)\), \(\chi_\varepsilon=1\) near \(\mathcal S\cap K\), and, on \(\mathcal R\),
\[
   |\nabla\chi_\varepsilon|
   \le
   \frac{1}{l_\varepsilon r_\mathcal S}
   \qquad\text{a.e. on }\{\varepsilon<r_\mathcal S<\rho_0\}.
\]
Let \(\rho_j=2^{-j}\rho_0\) and \(J_\varepsilon = \left\lceil \log_2\frac{\rho_0}{\varepsilon}\right\rceil\). Then
\[
   \rho_{J_\varepsilon}\le\varepsilon<\rho_{J_\varepsilon-1},
   \qquad
   J_\varepsilon\le C l_\varepsilon .
\]
Using \eqref{eq:packing-tube-volume-section3} on the dyadic annuli
\[
   \rho_{j+1}<r_\mathcal S\le\rho_j,
   \qquad
   j=0,\ldots,J_\varepsilon-1,
\]
gives
\[
   \int_U |\nabla\chi_\varepsilon|^2\,d\mu_\omega
   \le
   \frac{C}{l_\varepsilon^2}
   \sum_{j=0}^{J_\varepsilon-1}
   \rho_{j+1}^{-2}
   \mu_\omega\bigl(\{r_\mathcal S<\rho_j\}\bigr)  
   \le
   \frac{C}{l_\varepsilon}
   \longrightarrow0 .
\]
\end{proof}

\subsection{Conformal blow-up potentials}

Throughout this subsection we fix an admissible weight \(\omega\in\mathcal W_{\rm adm}(X)\) and work on the almost-manifold
\[
   (X,d,\mathcal R,g,\mu_\omega),
\]
which has AM--PI structure in the sense of Definition~\ref{def:admissible-weight-section3}. Throughout the arguments below, we assume \(m>2\). 

For \(\delta\ge0\), write
\[
   P_{\omega,\delta}:=-L_\omega+\delta .
\]

\begin{lemma}\label{Lem:Green representation}
Let \(U\Subset X\) be open, and assume \(\mathcal S\cap U\) has zero \(2\)-capacity. There exists \(\rho_U>0\) such that, for every \(\delta\ge 0\) and every ball
\[
\Omega=B_r(x_0)\Subset U,\qquad 0<r<\rho_U,
\]
the operator \(P_{\omega,\delta}\) admits a nonnegative symmetric Dirichlet Green function \(G_{\Omega,\delta}^\omega\).

Moreover, if \(f\in\operatorname{Lip}(\Omega)\) and \(u\in W_0^{1,2}(\Omega)\) is the weak solution of 
\[
P_{\omega,\delta}u=f\qquad\text{in }\Omega,
\]
then
\[
u(p)=\int_\Omega G_{\Omega,\delta}^\omega(p,y)f(y)d\mu_\omega(y),\qquad p\in\Omega.
\]
\end{lemma}
\begin{proof}
Since \(\mathcal S\cap U\) has zero \(2\)-capacity, the AM--PI space agrees locally with the usual PI-space.  Hence the local theory on PI spaces applies here; see textbooks in this direction, for example, by Bj\"orn--Bj\"orn~\cite{BjornBjorn2011} and by Heinonen--Koskela--Shanmugalingam-Tyson \cite{HKST2015}. Choose \(\rho_U>0\) so that the local Sobolev and Poincar\'e inequalities and the De Giorgi--Nash--Moser estimates hold uniformly at scales below \(\rho_U\) in \(U\).

Consider the Dirichlet form
\[
\mathcal E_{\Omega,\delta}(v,\varphi)=\int_{\Omega\cap\mathcal R}\langle\nabla_gv,\nabla_g\varphi\rangle_gd\mu_\omega+\delta\int_\Omega v\varphi d\mu_\omega
\]
on \(W_0^{1,2}(\Omega)\). By the Lax--Milgram theorem, for every \(h\in L^2(\Omega)\) there is a unique \(R_{\Omega,\delta}h\in W_0^{1,2}(\Omega)\) satisfying
\[
\mathcal E_{\Omega,\delta}(R_{\Omega,\delta}h,\varphi)=\int_\Omega h\varphi d\mu_\omega,\qquad \varphi\in W_0^{1,2}(\Omega).
\]
The operator \(R_{\Omega,\delta}\) is self-adjoint on \(L^2(\Omega)\).

Fix \(p\in\Omega\). Choose nonnegative functions \(\rho_k\in\operatorname{Lip}_c(\Omega)\) such that 
\[
\operatorname{supp}\rho_k\subset B_{\varepsilon_k}(p),\qquad\int_\Omega\rho_kd\mu_\omega=1,\qquad\varepsilon_k\longrightarrow 0,
\]
and set \(G_k=R_{\Omega,\delta}\rho_k\). By the maximum principle, \(G_k\ge 0\).

For \(t>0\), let \(T_t(G_k)=\min\{G_k,t\}\). Then 
\[
\int_{\Omega\cap\mathcal R}|\nabla_gT_t(G_k)|^2d\mu_\omega+\delta\int_\Omega G_kT_t(G_k)d\mu_\omega=\int_\Omega\rho_kT_t(G_k)d\mu_\omega\le t.
\]
In particular, \(\int_{\Omega\cap\mathcal R}|\nabla_gT_t(G_k)|^2d\mu_\omega\le t\). By the Sobolev inequality,
\[
\mu_\omega(\{G_k\ge t\})\le Ct^{-\frac{m}{m-2}}.
\]
Thus, for \(t,\lambda>0\),
\[
\begin{aligned}
    \mu_\omega(\{|\nabla_gG_k|\ge\lambda\})\le&\mu_\omega(\{G_k\ge t\})+\lambda^{-2}\int_{\{G_k<t\}\cap\mathcal R}|\nabla_gG_k|^2d\mu_\omega\\
    \le&Ct^{-\frac{m}{m-2}}+\lambda^{-2}t.
\end{aligned}
\]
Taking \(t=\lambda^{\frac{m-2}{m-1}}\), we obtain
\[
\mu_\omega(\{|\nabla_gG_k|\ge\lambda\})\le C\lambda^{-\frac{m}{m-1}}.
\]
It follows that
\[
\|G_k\|_{W^{1,s}(\Omega)}\le C_s,\qquad 1<s<\frac{m}{m-1}.
\]
After passing to a subsequence, there is a nonnegative function
\[
G_{\Omega,\delta}^\omega(p,\cdot)\in W_0^{1,s}(\Omega),\qquad 1<s<\frac{m}{m-1},
\]
such that
\[
G_k\rightharpoonup G_{\Omega,\delta}^\omega(p,\cdot)\qquad\text{weakly in }W_0^{1,s}(\Omega).
\]
On every compact subset of \(\Omega\setminus\{p\}\), the functions \(G_k\) are \(P_{\omega,\delta}\)-harmonic for all sufficiently large \(k\). The convergence is therefore locally uniform on \(\Omega\setminus\{p\}\).

Passing to the limit in the weak equation gives
\[
\int_{\Omega\cap\mathcal R}\left\langle\nabla_yG_{\Omega,\delta}^\omega(p,y),\nabla_y\varphi(y)\right\rangle_g d\mu_\omega(y)+\delta\int_\Omega G_{\Omega,\delta}^\omega(p,y)\varphi(y)d\mu_\omega(y)=\varphi(p)
\]
for every \(\varphi\in\operatorname{Lip}_c(\Omega)\). Thus \(G_{\Omega,\delta}^\omega(p,\cdot)\) is a Dirichlet Green function with pole at \(p\).

Let \(f\in\operatorname{Lip}(\Omega)\) and \(u=R_{\Omega,\delta}f\). The De Giorgi--Nash--Moser estimate gives a continuous representative of \(u\). By self-adjointness,
\[
\begin{aligned}
    \int_\Omega G_{\Omega,\delta}^\omega(p,y)f(y)d\mu_\omega(y)=&\lim_{k\to\infty}\int_\Omega G_k(y)f(y)d\mu_\omega(y)\\
    =&\lim_{k\to\infty}\int_\Omega\rho_k(y)u(y)d\mu_\omega(y)=u(p).
\end{aligned}
\]
Finally, let \(p\ne y\) and approximate \(\delta_y\) by functions \(\rho_{k,y}\) as above. Then
\[
R_{\Omega,\delta}\rho_{k,y}(p)=\int_\Omega G_{\Omega,\delta}^\omega(p,z)\rho_{k,y}(z)d\mu_\omega(z).
\]
The right-hand side converges to \(G_{\Omega,\delta}^\omega(p,y)\), while the left-hand side converges to \(G_{\Omega,\delta}^\omega(y,p)\). Hence \(G_{\Omega,\delta}^\omega\) is symmetric.
\end{proof}

\begin{lemma}\label{lem:local-green-estimate-section3}
Let \(U\) be as in Lemma~\ref{Lem:Green representation}, and fix \(\delta\ge 0\). There exist \(r_U>0\), \(C_U\ge 1\) and \(\Lambda_U>10\) such that if \(\Omega=B_{\Lambda_UR}(x_0)\Subset U\) and \(0<R<r_U\), then
\begin{equation}\label{eq:local-green-estimate-section3}
C_U^{-1}d(p,y)^{2-m}\le G_{\Omega,\delta}^\omega(p,y)\le C_Ud(p,y)^{2-m}
\end{equation}
for \(p\in B_R(x_0)\), \(y\in B_R(x_0)\cap\mathcal R\), and \(0<d(p,y)<2R\).
\end{lemma}

\begin{proof}
Let \(\rho_U>0\) be as in Lemma~\ref{Lem:Green representation}. Choose \(r_U>0\) and \(\Lambda_U>10\), with \(\Lambda_Ur_U<\rho_U\), so that, whenever
\[
   0<r<r_U,
   \qquad
   B_{\Lambda_U r}(z)\Subset U,
\]
every nonnegative \(L_\omega\)-harmonic function on \(B_{\Lambda_U r}(z)\) satisfies the Harnack inequality with constant \(H_U\) on \(B_r(z)\) and \(B_{2r}(z)\setminus \overline{B_r(z)}\).

We first treat \(\delta=0\). For \(B_{2r}(p)\Subset\Omega\) and \(0<r<r_U\), we have the local capacity estimate
\begin{equation}\label{eq:ball-capacity-section3}
   C_U^{-1}r^{m-2}
   \le
   \operatorname{Cap}_\omega(\overline{B_r(p)},\Omega)
   \le
   C_Ur^{m-2}.
\end{equation}
Indeed, the upper bound follows from a Lipschitz cutoff in \(B_{2r}(p)\), while the lower bound follows from the local Sobolev inequality and the Ahlfors lower bound.

Let \(G=G_{\Omega,0}^\omega(p,\cdot)\). Moreover, we have
\begin{equation}\label{eq:green-level-identity-section3}
   t\,\operatorname{Cap}_\omega(\{G\ge t\},\Omega)=1,
   \qquad t>0 .
\end{equation}
This is the \(p=2\) case of the superlevel-set identity in \cite[Lemma~3.4 and Theorem~9.3]{BjornBjornLehrback2020}. 
Let \(y\in B_R(x_0)\cap\mathcal R\), \(y\ne p\), and set
\[
   \rho=d(p,y),
   \qquad
   A_\rho(p)=B_{2\rho}(p)\setminus\overline{B_{\rho/2}(p)} .
\]
Write
\[
   M_\rho=\sup_{A_\rho(p)}G,
   \qquad
   m_\rho=\inf_{A_\rho(p)}G .
\]
Since \(B_{4\rho}(p)\subset B_{9R}(x_0)\Subset\Omega\), the Harnack inequality gives \(M_\rho\le H_Um_\rho\). The maximum principle on \(\Omega\setminus B_{2\rho}(p)\) gives \(\{G>M_\rho\}\subset B_{2\rho}(p)\). Using \eqref{eq:ball-capacity-section3} and \eqref{eq:green-level-identity-section3},
\[
   \frac1{M_\rho}
   =
   \operatorname{Cap}_\omega(\{G>M_\rho\},\Omega)
   \le
   \operatorname{Cap}_\omega(B_{2\rho}(p),\Omega)
   \le
   C_U\rho^{m-2}.
\]
Thus \(M_\rho\ge c_U\rho^{2-m}\), and the Harnack inequality gives
\[
   G(y)\ge m_\rho\ge H_U^{-1}M_\rho\ge c_U\rho^{2-m}.
\]

For the upper bound, the minimum principle on \(B_{\rho/2}(p)\setminus\{p\}\) gives
\[
   B_{\rho/2}(p)\subset\{G\ge m_\rho\}
\]
up to a set of zero capacity. Hence
\[
   \frac1{m_\rho}
   =
   \operatorname{Cap}_\omega(\{G\ge m_\rho\},\Omega)
   \ge
   \operatorname{Cap}_\omega(B_{\rho/2}(p),\Omega)
   \ge
   c_U\rho^{m-2}.
\]
Therefore \(m_\rho\le C_U\rho^{2-m}\), and the Harnack inequality gives
\[
   G(y)\le M_\rho\le H_Um_\rho\le C_U\rho^{2-m}.
\]
This proves \eqref{eq:local-green-estimate-section3} for \(\delta=0\).

The same argument gives
\[
   G_{\Omega,0}^\omega(q,z)
   \le
   C_Ud(q,z)^{2-m}
\]
for \(q\in B_R(x_0)\) and \(z\in\Omega\cap\mathcal R\setminus\{q\}\). Indeed, the preceding estimate applies when \(d(q,z)<2R\), while the maximum principle on \(\Omega\setminus\overline{B_R(q)}\) gives \(G_{\Omega,0}^\omega(q,z)\le C_UR^{2-m}\) when \(d(q,z)\ge R\). Since \(d(q,z)\le(\Lambda_U+1)R\), the claim follows after enlarging \(C_U\).

Next we deal with \(\delta>0\). By the maximum principle \(0<G_{\Omega,\delta}^\omega(p,y)\le G_{\Omega,0}^\omega(p,y)\), so the upper bound follows from the case \(\delta=0\).

For the lower bound, the resolvent identity for the Green functions constructed in Lemma~\ref{Lem:Green representation} gives
\[
   G_{\Omega,\delta}^\omega(p,y)
   =
   G_{\Omega,0}^\omega(p,y)
   -
   \delta
   \int_\Omega
      G_{\Omega,0}^\omega(p,z)
      G_{\Omega,\delta}^\omega(z,y)\,d\mu_\omega(z).
\]
Since \(G_{\Omega,\delta}^\omega\le G_{\Omega,0}^\omega\), symmetry and the preceding upper bound give
\[
   G_{\Omega,\delta}^\omega(p,y)
   \ge
   G_{\Omega,0}^\omega(p,y)
   -
   C_U\delta
   \int_\Omega
      d(p,z)^{2-m}d(z,y)^{2-m}\,d\mu_\omega(z).
\]

Decompose
\[
   \Omega=E_p\cup E_y,
   \qquad
   E_p=\{z\in\Omega:\ d(p,z)\le d(y,z)\},
   \qquad
   E_y=\Omega\setminus E_p .
\]
On \(E_p\), the triangle inequality gives \(d(y,z)\ge \frac{\rho}{2}\) and on \(E_y\) it gives \(d(p,z)\ge \frac{\rho}{2}\). Hence
\[
\begin{aligned}
&\int_\Omega
   d(p,z)^{2-m}d(z,y)^{2-m}\,d\mu_\omega(z)  \\
&\qquad\le
C\rho^{2-m}
\left(
   \int_\Omega d(p,z)^{2-m}\,d\mu_\omega(z)
   +
   \int_\Omega d(y,z)^{2-m}\,d\mu_\omega(z)
\right).
\end{aligned}
\]
Note that if \(q\in B_R(x_0)\), then
\[
   \Omega\subset B_{(\Lambda_U+1)R}(q).
\]
By the Ahlfors upper bound and a dyadic decomposition around \(q\),
\[
\begin{aligned}
   \int_\Omega d(q,z)^{2-m}\,d\mu_\omega(z)
   &\le
   \sum_{j=0}^\infty
   (2^{-j-1}(\Lambda_U+1)R)^{2-m}
   \mu_\omega\bigl(B_{2^{-j}(\Lambda_U+1)R}(q)\bigr) \\
   &\le
   C_U
   \sum_{j=0}^\infty
   (2^{-j}R)^2
   \le
   C_U R^2 .
\end{aligned}
\]
Applying this with \(q=p\) and \(q=y\) gives
\[
   \int_\Omega
      d(p,z)^{2-m}d(z,y)^{2-m}\,d\mu_\omega(z)
   \le
   C_U R^2 d(p,y)^{2-m}.
\]

Hence the \(\delta=0\) case gives
\[
   G_{\Omega,\delta}^\omega(p,y)
   \ge
   C_U^{-1}d(p,y)^{2-m}
   -
   C_U\delta R^2d(p,y)^{2-m}.
\]
After decreasing \(r_U\) so that \(C_U\delta r_U^2\le\frac12C_U^{-1}\) and using \(R<r_U\), we obtain \(G_{\Omega,\delta}^\omega(p,y)\ge c_Ud(p,y)^{2-m}\). This proves the lemma.
\end{proof}

In the following, we construct global Green functions for $P_{\omega,\delta}$ by exhaustion.

\begin{proposition}\label{prop:global-resolvent-green-section3}
Assume that $\mathcal S\cap U$ has zero $2$-capacity for any open $U\Subset X$. For any \(\delta>0\), there exists a positive Green function
\[
   G_\delta^\omega(p,y)
\]
for \(P_{\omega,\delta}\) on \(X\).  For every \(p\in X\), we have \[P_{\omega,\delta}G_\delta^\omega(p,\cdot)=\boldsymbol\delta_p\] in the weak sense, and \(G_\delta^\omega(p,\cdot)\) is smooth on \(\mathcal R\setminus\{p\}\).
Moreover, for every compact \(K\Subset X\), there exist constants \(r_K>0\) and \(C_K\ge1\) such that
\begin{equation}\label{eq:global-green-local-asymp-section3}
   C_K^{-1}d(p,y)^{2-m}
   \le
   G_\delta^\omega(p,y)
   \le
   C_Kd(p,y)^{2-m}
\end{equation}
whenever
\[
   p\in K,\qquad y\in X,\qquad 0<d(p,y)<r_K .
\]
Finally, if \(K\Subset W\Subset X\), then there is a constant \(C_{K,W}<\infty\) such that
\begin{equation}\label{eq:global-green-exterior-bound-section3}
   G_\delta^\omega(p,y)\le C_{K,W}
\end{equation}
for all $p\in K$ and $y\in X\setminus W$.
\end{proposition}

\begin{proof}
Choose an exhaustion by precompact open sets
\[
   V_1\Subset V_2\Subset\cdots,
   \qquad
   \bigcup_{\ell=1}^{\infty}V_\ell=X.
\]
Let \(G_{\ell,p}\) be the Dirichlet Green function for \(P_{\omega,\delta}\) on \(V_\ell\), with pole \(p\in V_\ell\).  By maximum principle we have \( G_{\ell,p}\le G_{\ell+1,p}\). Define
\[
   G_\delta^\omega(p,y)
   :=
   \lim_{\ell\to\infty}G_{\ell,p}(y).
\]

As a preparation, we establish some uniform estimates for $G_{\ell,p}$.
Let \(h_\ell\in W^{1,2}_0(V_\ell)\) solve
\[
   P_{\omega,\delta}h_\ell=1
   \qquad\text{in }V_\ell .
\]
The maximum principle gives \( 0\le h_\ell\le\delta^{-1}\). By the Green representation formula from Lemma \ref{Lem:Green representation}, we have
\[
   h_\ell(p)
   =
   \int_{V_\ell}G_{\ell,p}(y)\,d\mu_\omega(y),
\]
and so
\begin{equation}\label{eq:global-green-l1-bound-section3}
   \int_{V_\ell}G_{\ell,p}(y)\,d\mu_\omega(y)
   \le
   \delta^{-1}.
\end{equation}
This \(L^1\)-bound, together with Harnack inequalities, gives local uniform bounds for \(G_{\ell,p}\) away from the pole $p$. By maximum principle, we have
$$G_{\ell_0,p}\leq G_{\ell,p}\leq G_{\ell_0,p}+C\quad\text{ in }V_{\ell_0}$$
for all $\ell\geq \ell_0$,
and
$$G_{\ell,p}\leq C_r\quad\text{ in } V_\ell\setminus B_r(p),\quad \forall r>0,$$
where $C>0$ is a constant independent of $\ell_0$ and $\ell$, and  $C_r>0$ is a constant depending only on $\delta$, $r$, and AM-PI constants around $p$. Since AM-PI constants are uniform around any compact subset $K$, we actually have
\begin{equation}\label{Eq: upper bound Glp}
    \sup_{p\in K}\sup_{y\in V_\ell\setminus B_r(p)}G_{\ell,p}(y)\leq C_r,
\end{equation}
where $C_r>0$ is a constant depending only on $\delta$, $r$, and $K$.

We claim that 
$$\|G_{\ell,p}\|_{W^{1,\theta}(U)}\leq C_U \qquad \theta\in(1,\frac{m}{m-1})$$ for any open $U\Subset X$. Indeed, the previous discussion yields a uniform constant $C$ such that $\tilde G_{\ell,p}:=(G_{\ell,p}-C)_+$ is supported in $U$ for all $\ell$. From the same argument in the proof of Lemma \ref{Lem:Green representation}, we obtain $\|G_{\ell,p}\|_{W^{1,\theta}(U)}\leq C_U$.

Let us verify the desired properties of $G_\delta^\omega(p,\cdot)$. For any Lipschitz function $\eta$ with compact support, we take bounded $U\Subset X$ such that $\eta$ is supported in $U$. From the claim above, $G_{\ell,p}(\cdot)$ converges to $G_\delta^\omega(p,\cdot)$ weakly in $W^{1,\theta}(U)$. In particular, we have
\[
\begin{split}
   & \int_X\langle\nabla_y G_\delta^\omega(p,y),\nabla_y\eta(y)\rangle_g+\delta G_\delta^\omega(p,y)\eta(y)\,d\mu_\omega\\
   =& \lim_{\ell\to+\infty}\int_X\langle\nabla_y G_{\ell,p}(y),\nabla_y\eta(y)\rangle_g+\delta G_{\ell,p}(y)\eta(y)\,d\mu_\omega\\
   =&\lim_{\ell\to+\infty}\eta(p)=\eta(p).
\end{split}
\]
This verifies $P_{\omega,\delta}G_\delta^\omega(p,\cdot)=\boldsymbol\delta_p$ in the weak sense. The smoothness of $G_\delta^\omega(p,\cdot)$ on $\mathcal R\setminus\{p\}$ comes from the smooth convergence of $G_{\ell,p}$ due to uniform bound and standard elliptic estimates.

Next we focus on \eqref{eq:global-green-local-asymp-section3}. For any \(K\Subset X\), we fix $V_{\ell_0}\supset K$. Choose \(r_K>0\) so small such that Lemma~\ref{lem:local-green-estimate-section3} applies on
\(B_{4r_K}(p)\subset V_{\ell_0}\) for every \(p\in K\).
By the maximum principle, for $\ell\geq \ell_0$ we have 
\[
  G_{\ell,p}(y)\ge G_{\ell_0,p}(y) \ge  C_K^{-1}d(p,y)^{2-m},
   \qquad 0<d(p,y)<r_K .
\]
Letting \(\ell\to\infty\) gives the lower bound in \eqref{eq:global-green-local-asymp-section3}. On the other hand, from the upper bound \eqref{Eq: upper bound Glp} and the maximum principle, we have
$$G_{\ell,p}\leq G_{\ell_0,p}+C_{r_K}.$$
Combined with
$$  G_{\ell_0,p}(y)
   \le
   C_Kd(p,y)^{2-m},$$
   we conclude
   $$G_{\ell,p}\leq C_Kd(p,y)^{2-m}$$
   after increasing $C_K$ slightly. Letting \(\ell\to\infty\) proves the upper bound in \eqref{eq:global-green-local-asymp-section3}.

   Finally we deal with \eqref{eq:global-green-exterior-bound-section3}. Given $K\Subset W\Subset X$, we take $W_0$ such that $K\Subset W_0\Subset W$. Since $K$ and $X\setminus W_0$ have positive distance, from \eqref{Eq: upper bound Glp} we have 
\[
   G_{\ell,p}\le C_{K,W}
   \qquad
   \text{on }\partial W_0,\quad p\in K .
\]
Passing to the limit gives the exterior bound \eqref{eq:global-green-exterior-bound-section3}.

\end{proof}

\begin{proposition}\label{prop:global-packing-green-blowup-section3}
Let \(X=\mathcal R\sqcup\mathcal S\) be an \(m\)-dimensional AM--PI space, and let \(\omega\in\mathcal W_{\rm adm}(X)\).   Assume
\[
   0<\tau<\tau_{\rm pack}(\mathcal S)=\operatorname{codim}_A(\mathcal S)-2,
   \qquad
   \alpha=\frac2\tau .
\]
Then there exists a nonnegative function
\[
   \Theta\in C^\infty(\mathcal R)
\]
and constants \(A>0\), \(B\ge0\) such that
\begin{equation}\label{eq:global-packing-green-diff-section3}
   -L_\omega\Theta
   \ge
   A\Theta^{1+\alpha}-B
\end{equation}
on \(\mathcal R\).  Moreover, for every compact \(K\Subset X\) there exist \(c_K>0\) and \(r_K>0\) such that
\begin{equation}\label{eq:global-packing-green-growth-section3}
   \Theta(x)
   \ge
   c_K\operatorname{dist}(x,\mathcal S)^{-\tau}
\end{equation}
whenever \(x\in K\cap\mathcal R\) and \(0<\operatorname{dist}(x,\mathcal S)<r_K\).
\end{proposition}

\begin{proof}
Choose \(d_*\) with \(d_*<m-2-\tau\) such that \(\mathsf P_{d_*}(\mathcal S)\) holds, and set
\[
   \theta=m-2-\tau-d_*>0 .
\]
Since \(d_*<m-2\), Lemma~\ref{lem:packing-zero-capacity-section3} implies that \(\mathcal S\cap U\) has zero capacity for any 
open $U\Subset X$, and then    Proposition~\ref{prop:global-resolvent-green-section3} applies, where we fix some constant \(\delta>0\) below.

Since AM-PI spaces are locally compact and paracompact, we can choose a locally finite family of precompact open sets \(U_i\Subset X\) such that
\[
   \mathcal S\subset \bigcup_i U_i,
   \qquad
   U_i\cap\mathcal S\ne\emptyset .
\]
Moreover, we can choose compact sets \(K_i\Subset X\) such that
\[
   \overline U_i\Subset K_i^\circ,
\]
and such that the family \(\{K_i\}\) is locally finite as well.  Put
\[
   S_i=\mathcal S\cap\overline U_i .
\]

Let
\[
   r_i>0,\qquad C_i\ge1
\]
be the constants from Proposition~\ref{prop:global-resolvent-green-section3} applied to \(K_i\).  After decreasing $r_i$ and increasing \(C_i\), we can assume that the Ahlfors estimates holds for all balls $B_r(y)$ with $y\in K_i$ and $r<r_i$, and that we have the off-diagonal bound
\begin{equation}\label{Eq: off-diagonal}
    G_\delta^\omega(p,y)\le C_i
\end{equation}
whenever $p,y\in K_i$ and $d(p,y)\geq r_i/4$.  From the definition of \(\mathsf P_{d_*}(\mathcal S)\), by increasing $C_i$ and decreasing $r_i$ properly, we can assume the packing estimate \eqref{eq:local-packing-condition-section3} holds on $K_i$ with constants $C_i$ and $r_i$. In particular, after shrinking $U_i$ with diameter less than $r_i/2$, we have the packing estimate
\begin{equation}\label{Eq: packing}
    \#(P_s\cap B_R(x))\leq C_i\left(\frac{R}{s}\right)^{d_*}
\end{equation}
whenever $x\in K_i$ with $\dist(x,U_i)<r_i/2$, $0<s<\min\{R,r_i\}$, and $P_s\subset \mathcal S_i$ is $s$-separated. In the following, we always use $C_i$ to denote a constant depending only on $i$, whose value may change from line to line.

Choose \(s_{i,0}\in(0,r_i/80)\) small enough so that
\[
   \bigcup_{z\in \mathcal S_i}
   B_{20s_{i,0}}(z)\subset K_i .
\]
Set \[s_{i,j}=2^{-j}s_{i,0}.\]  For each \(j\), choose a maximal \(s_{i,j}\)-separated set
\[
   P_{i,j}\subset S_i .
\]
For each \(z\in P_{i,j}\), choose
\[
   \eta_{i,j,z}\in C^\infty(\mathcal R),
   \qquad
   0\le\eta_{i,j,z}\le1,
\]
such that \(\eta_{i,j,z}=1\) on \(B_{10s_{i,j}}(z)\cap\mathcal R\), and $\eta_{i,j,z}=0$ outside $B_{20s_{i,j}}(z)\cap\mathcal R$.

\vspace{3mm}Define the ball potential
\[
   u_{i,j,z}(x)
   =
   \int_{\mathcal R}
   G_\delta^\omega(y,x)\eta_{i,j,z}(y)\,d\mu_\omega(y).
\]
Then we have \[P_{\omega,\delta}u_{i,j,z}=\eta_{i,j,z}\quad\text{ on }\mathcal R.\]

\vspace{3mm}
\noindent{\bf Step 1: Estimate for each $u_{i,j,z}$.}
\vspace{3mm}

\noindent{\bf Claim:} We have the estimates
\begin{equation}\label{eq:global-source-upper-section3}
   u_{i,j,z}(x)
   \le
   C_i s_{i,j}^m
   \bigl(d(x,z)+s_{i,j}\bigr)^{2-m},\quad \forall x\in K_i\cap \mathcal R,
\end{equation}
and
\begin{equation}\label{eq:global-source-inner-section3}
   c_i s_{i,j}^2
   \le
   u_{i,j,z}(x)
   \le
   C_i s_{i,j}^2
   \qquad
   \text{if }d(x,z)\le5s_{i,j}.
\end{equation}

\vspace{3mm} For simplicity,  we write
\[
   s=s_{i,j},\qquad\eta=\eta_{i,j,z},\qquad u=u_{i,j,z},\qquad R=d(x,z),
\]
where \(x\in K_i\cap\mathcal R\). 

First we show \eqref{eq:global-source-upper-section3} by dividing the discussion into three cases.

\vspace{3mm} {\it Case 1. $R\leq 40s$.} 

In this case, we have \(\operatorname{spt}\eta \subset B_{60s}(x)\). By \eqref{eq:global-green-local-asymp-section3} and the Ahlfors estimates, we have
\[
\begin{aligned}
   u(x)
   &\le
   C_i\int_{B_{60s}(x)}
      d(x,y)^{2-m}\,d\mu_\omega(y)  \\
   &\le
   C_i\sum_{\ell=0}^{\infty}
      (2^{-\ell}s)^{2-m}
      \mu_\omega(B_{2^{-\ell}60s}(x))
   \le
   C_i s^2 .
\end{aligned}
\]
Since \(s\le R+s\le41s\), this gives \eqref{eq:global-source-upper-section3}.

\vspace{3mm} {\it Case 2. $40s\leq R<r_i/2$.}

In this case, we have 
\[\frac12R\le d(x,y)\le\frac32R<r_i,\quad\forall y\in\operatorname{spt}\eta. \] 
By \eqref{eq:global-green-local-asymp-section3} and the Ahlfors estimates, we have
\[
   u(x)
   \le
   C_i R^{2-m}\mu_\omega(B_{20s}(z))
   \le
   C_i s^m R^{2-m}
   \le
   C_i s^m(R+s)^{2-m},
\]
establishing \eqref{eq:global-source-upper-section3}.

\vspace{3mm}{\it Case 3. $R\geq r_i/2$.}

In this case, we have
\[
   d(x,y)\ge R-20s\ge \frac14r_i
   \qquad
   \text{for }y\in\operatorname{spt}\eta.
\]
The off-diagonal bound \eqref{Eq: off-diagonal} and the Ahlfors estimates imply \[u(x)\le C_i\mu_\omega(B_{20s}(z))\le C_i s^m.\] Since \(x,z\in K_i\), the quantity \(R+s\) is uniformly bounded from above.  After further increasing \(C_i\), we again get \eqref{eq:global-source-upper-section3}.

\vspace{3mm}
It remains to show \eqref{eq:global-source-inner-section3}. Since we have \(d(x,z)\le5s\), the upper bound of \eqref{eq:global-source-inner-section3} follows from \eqref{eq:global-source-upper-section3} directly. For the lower bound, note that we have
\(B_{s/4}(x)\subset B_{10s}(z)\) when \(d(x,z)\le5s\), and so \(\eta=1\) on \(B_{s/4}(x)\cap\mathcal R\).  In particular, the lower bound in \eqref{eq:global-green-local-asymp-section3} and the Ahlfors estimates give
\[
   u(x)\ge
   c_i\int_{B_{s/4}(x)}
      d(x,y)^{2-m}\,d\mu_\omega(y)
   \ge
   c_i s^{2-m}
   \mu_\omega\bigl(B_{s/4}(x)\bigr)
   \ge
   c_i s^2 .
\]
We complete the proof of the claim.

\vspace{3mm}
\noindent{\bf Step 2: Estimate for weighted summation of $u_{i,j,z}$.}
\vspace{3mm}

Define the weighted summation
\[
   \Psi_i
   =
   \sum_j\sum_{z\in P_{i,j}}
   s_{i,j}^{-\tau-2}u_{i,j,z}.
\]
Since \(P_{\omega,\delta}u_{i,j,z}=\eta_{i,j,z}\), we have
\begin{equation}\label{eq:block-source-equation-section3}
   P_{\omega,\delta}\Psi_i
   =
   \sum_j\sum_{z\in P_{i,j}}
   s_{i,j}^{-\tau-2}\eta_{i,j,z}
   \ge0 .
\end{equation}

\noindent{\bf Claim:} there are constants $A_i>0$ and $B_i>0$ such that
\begin{equation}\label{eq:block-P-diff-section3}
   P_{\omega,\delta}\Psi_i
   \ge
   A_i\Psi_i^{1+2/\tau}-B_i
   \qquad\text{on }\mathcal R.
\end{equation}

\vspace{3mm} 
For \(x\in K_i\cap\mathcal R\) and a fixed scale \(s=s_{i,j}\), write
\[
   \Psi_{i,s}(x)
   =
   \sum_{z\in P_{i,j}}s^{-\tau-2}u_{i,j,z}(x).
\]
Using \eqref{eq:global-source-upper-section3}, we obtain
\[
   \Psi_{i,s}(x)
   \le
   C_i\sum_{z\in P_{i,j}}
   s^{m-2-\tau}\bigl(d(x,z)+s\bigr)^{2-m}.
\]

Write \[r=\operatorname{dist}(x,S_i)<r_i/2.\]
We investigate the estimates for $\Psi_{i,s}$ as follows. 
 If \(s\le r\), we make the decomposition
 $$P_{i,j}=\bigcup_{\ell=0}^{+\infty}P_{i,j}\cap \mathcal A_\ell,\quad \mathcal A_\ell:=B_{2^{\ell+1}r}(x)\setminus B_{2^\ell r}(x).$$
The packing condition \eqref{Eq: packing} gives
\[
\sum_{z\in P_{i,j}\cap \mathcal A_{\ell}}
   s^{m-2-\tau}\bigl(d(x,z)+s\bigr)^{2-m}
   \le C_i2^{-\ell(\tau+\theta)}s^\theta r^{-\tau-\theta}.
\]
Summing over \(\ell\) yields
\begin{equation}\label{Eq: Psi is r large}
    \Psi_{i,s}(x)\le C_i s^\theta r^{-\tau-\theta},\quad s\leq r.
\end{equation}
If \(s\ge r\), the same estimate with \(s\) in place of \(r\) gives 
\begin{equation}\label{Eq: Psi is s large}
    \Psi_{i,s}(x)\le C_i s^{-\tau},\quad s\geq r.
\end{equation}

We are ready to estimate $\Psi_i$.
First we show
\begin{equation}\label{eq:block-upper-section3}
   c_i\tau^{-\tau}\leq\Psi_i(x)\le C_i r^{-\tau},\quad r\in (0,s_{i,0}).
\end{equation}
The upper bound follows directly from summing $\Psi_{i,s}$ over \(j\) and using the estimates \eqref{Eq: Psi is r large} and \eqref{Eq: Psi is s large}. For the lower bound, we choose $j$ such that \(s_{i,j+1}\le r<s_{i,j}\). Pick \(z_0\in S_i\) with \(d(x,z_0)\le2r\).  By the maximality of \(P_{i,j}\), there is \(p_0\in P_{i,j}\) such that \(d(z_0,p_0)\le s_{i,j}\). Thus we have
\[
   d(x,p_0)\le2r+s_{i,j}<3s_{i,j}<5s_{i,j}.
\]
From \eqref{eq:global-source-inner-section3} we obtain
\[
   \Psi_i(x)
   \ge
   s_{i,j}^{-\tau-2}u_{i,j,p_0}(x)
   \ge
   c_i s_{i,j}^{-\tau}
   \ge
   c_i r^{-\tau}.
\]
On the other hand, since we have \(\eta_{i,j,p_0}(x)=1\), it follows from \eqref{eq:block-source-equation-section3} that
\begin{equation}\label{Eq: Psi is P value}
    P_{\omega,\delta}\Psi_i(x)
   \ge
   s_{i,j}^{-\tau-2}
   \ge
   c_i r^{-\tau-2},\quad r\in (0,s_{i,0}).
\end{equation}

From \eqref{eq:block-upper-section3} and \eqref{Eq: Psi is P value} there is a constant $A_i>0$ such that
\begin{equation*}
   P_{\omega,\delta}\Psi_i
   \ge
   A_i\Psi_i^{1+2/\tau}
\end{equation*}
on  \(\mathcal R\cap\{0<\operatorname{dist}(\cdot,S_i)<s_{i,0}\}\).
The estimate \eqref{eq:block-upper-section3}, the off-diagonal bound for \(G_\delta^\omega\), and interior elliptic regularity imply that \(\Psi_i\in C^\infty(\mathcal R)\).  Moreover, \(\Psi_i\) is bounded on
\[
   \mathcal R\cap
   \{\operatorname{dist}(\cdot,S_i)\ge s_{i,0}\}.
\]
Together with \(P_{\omega,\delta}\Psi_i\ge0\), this implies the existence of a constant $B_i>0$ such that \eqref{eq:block-P-diff-section3} holds.

\vspace{3mm}
\noindent{\bf Step 3: Construction of the function $\Theta$.}
\vspace{3mm}

Choose positive coefficients \(\varepsilon_i\) decreasing sufficiently fast so that
\[
   E_A:=\sum_i\varepsilon_iA_i^{-\tau/2}<\infty,
   \qquad
   E_B:=\sum_i\varepsilon_iB_i<\infty,
\]
and \(\Theta_0=\sum_i\varepsilon_i\Psi_i\) converges in \(C^\infty_{\rm loc}(\mathcal R)\).  Since \(\alpha=2/\tau\), H\"older's inequality gives
\[
   \Theta_0^{1+\alpha}
   \le
   E_A^\alpha
   \sum_i\varepsilon_iA_i\Psi_i^{1+\alpha}.
\]
Therefore we have \(P_{\omega,\delta}\Theta_0\ge E_A^{-\alpha}\Theta_0^{1+\alpha}-E_B\). By the definition of \(P_{\omega,\delta}\), we get
\[
   -L_\omega\Theta_0
   \ge
   E_A^{-\alpha}\Theta_0^{1+\alpha}
   -E_B-\delta\Theta_0 .
\]
Absorbing the linear term into the superlinear term gives
\[
   -L_\omega\Theta_0
   \ge
   A\Theta_0^{1+\alpha}-B
\]
for some \(A>0\) and \(B\ge0\).  Set \(\Theta=\Theta_0\).  This proves \eqref{eq:global-packing-green-diff-section3}.

Finally, we prove \eqref{eq:global-packing-green-growth-section3}. Let \(K\Subset X\).  Choose \(K'\Subset X\) with \(K\Subset (K')^\circ\).  Choose \(r_K>0\) so small that
\[
   \{x\in X:\operatorname{dist}(x,K)<2r_K\}\subset K' .
\]
Let \(I_K\) be the finite set of indices with \(U_i\cap K'\ne\emptyset\). 
Decrease \(r_K\) further so that
\[
   2r_K<s_{i,0}
   \qquad
   \text{for all }i\in I_K .
\]
If \(K'\cap\mathcal S=\emptyset\), decrease \(r_K\) once more so that \(\operatorname{dist}(K,\mathcal S)>r_K\). Then there is nothing to prove.  Otherwise set
\[
   c_K=
   2^{-\tau}\min\{\varepsilon_i c_i:\ i\in I_K\}.
\]
Let
\[
   x\in K\cap\mathcal R,
   \qquad
   0<\operatorname{dist}(x,\mathcal S)<r_K .
\]
Choose \(z\in\mathcal S\) with \(d(x,z)\le2\operatorname{dist}(x,\mathcal S)\). Then \(z\in K'\), and hence \(z\in U_i\) for some \(i\in I_K\).  Since \(z\in S_i\),
\[
   \operatorname{dist}(x,S_i)
   \le
   d(x,z)
   \le
   2\operatorname{dist}(x,\mathcal S)
   <
   2r_K
   <
   s_{i,0} .
\]
Using \eqref{eq:block-upper-section3}, we obtain
\[
   \Theta(x)
   \ge
   \varepsilon_i\Psi_i(x)
   \ge
   \varepsilon_i c_i\operatorname{dist}(x,S_i)^{-\tau}
   \ge
   c_K\operatorname{dist}(x,\mathcal S)^{-\tau}.
\]
We complete the proof.
\end{proof}

\subsection{Packing exponents and conformal blow-up}
\label{subsec:packing-exponents-conformal-blowup}

\begin{lemma}
\label{lem:conformal-absorption-section3}
Let
\[
   \alpha>0,\qquad
   \mathfrak K>0,\qquad
   A>0,\qquad
   B\ge0,
\]
and let \(\sigma\in\mathbb R\).  For every \(\zeta>0\), there exists \(\varepsilon_0>0\), depending only on
\[
   \alpha,\mathfrak K,A,B,\sigma,\zeta,
\]
such that, for all \(0<\varepsilon<\varepsilon_0\) and all \(T\ge0\), if
\[
   w=1+\varepsilon T ,
\]
then
\begin{equation}\label{eq:absorption-qf-section3}
   \mathfrak K A\varepsilon T^{1+\alpha}w^{-1}
   -
   \mathfrak K B\varepsilon w^{-1}
   -
   \frac12\sigma(w^\alpha-1)
   +
   \frac{\zeta}{2}w^\alpha
   \ge0 .
\end{equation}
\end{lemma}

\begin{proof}
Set \( z=\frac{\varepsilon T}{1+\varepsilon T}\). Then \(0\le z<1\), \(w=(1-z)^{-1}\), and
\[
   \varepsilon T^{1+\alpha}w^{-1}
   =
   \varepsilon^{-\alpha}z^{1+\alpha}w^\alpha .
\]
Also
\[
   w^\alpha-1
   =
   w^\alpha\bigl(1-w^{-\alpha}\bigr)
   =
   w^\alpha\bigl(1-(1-z)^\alpha\bigr).
\]
Since \(w^\alpha-1\ge0\), writing \(\sigma_+=\max\{\sigma,0\}\), we have
\[
\begin{split}
&\mathfrak K A\varepsilon T^{1+\alpha}w^{-1}
-\frac12\sigma(w^\alpha-1) \\
\ge&
w^\alpha
\left[
   \mathfrak K A\varepsilon^{-\alpha}z^{1+\alpha}
   -
   \frac12\sigma_+
   \bigl(1-(1-z)^\alpha\bigr)
\right].
\end{split}
\]
Using
\[
   1-(1-z)^\alpha\le C_\alpha z,
   \qquad 0\le z<1,
\]
and
\[
   \sup_{z\ge0}\{az-Mz^{1+\alpha}\}
   \le
   C_\alpha a^{1+1/\alpha}M^{-1/\alpha},
\]
with \(a=\frac{1}{2}C_\alpha\sigma_+\) and \(M=\mathfrak K A\varepsilon^{-\alpha}\), we get
\[
   \mathfrak K A\varepsilon T^{1+\alpha}w^{-1}
   -
   \frac12\sigma(w^\alpha-1)
   \ge
   -C\varepsilon w^\alpha .
\]
Since \(w\ge1\),
\[
   \mathfrak K B\varepsilon w^{-1}
   \le
   \mathfrak K B\varepsilon w^\alpha .
\]
Therefore the left hand side of \eqref{eq:absorption-qf-section3} is bounded below by
\[
   \left(
      \frac{\zeta}{2}
      -C\varepsilon
      -\mathfrak K B\varepsilon
   \right)w^\alpha .
\]
Choosing \(\varepsilon_0>0\) sufficiently small proves the lemma.
\end{proof}

\begin{lemma}
\label{lem:conformal-completeness-section3}
Let \(\alpha>0\), let \(\Theta\ge0\), and suppose that for every compact \(K\Subset X\) there exist \(c_K>0\) and \(r_K>0\) such that
\[
   \Theta(x)
   \ge
   c_K\operatorname{dist}(x,\mathcal S)^{-2/\alpha}
\]
whenever
\[
   x\in K\cap\mathcal R,
   \qquad
   0<\operatorname{dist}(x,\mathcal S)<r_K .
\]
If
\[
   w=1+\varepsilon\Theta,
   \qquad
   \widetilde g=w^\alpha g,
\]
then \((\mathcal R,\widetilde g)\) is complete.
\end{lemma}

\begin{proof}
Let
\[
   \rho(x)=\operatorname{dist}(x,\mathcal S).
\]
On each compact region near \(\mathcal S\),
\[
   w^\alpha\ge(\varepsilon c)^\alpha\rho^{-2},
   \qquad
   ds_{\widetilde g}\ge(\varepsilon c)^{\alpha/2}\rho^{-1}ds_g .
\]
Let \(\gamma:[s_0,S)\to\mathcal R\) be a curve approaching \(\mathcal S\), parametrized by \(g\)-arclength on a final finite-length segment.  Since \(\rho(\gamma(s))\to0\) as \(s\to S\), and since \(\rho\) is \(1\)-Lipschitz,
\[
   \rho(\gamma(s))\le S-s .
\]
Therefore
\[
   L_{\widetilde g}(\gamma)
   \ge
   (\varepsilon c)^{\alpha/2}
   \int_{s_0}^{S}\frac{ds}{\rho(\gamma(s))}
   \ge
   (\varepsilon c)^{\alpha/2}
   \int_{s_0}^{S}\frac{ds}{S-s}
   =
   \infty .
\]
\end{proof}

We use the notation from Section~\ref{sec:weighted-scalar-curvature}.  Set
\begin{equation}\label{eq:tau-m-lambda-section3}
   \tau_m(\lambda)
   :=
   \frac2{\alpha_m(\lambda)}
   =
   \frac{m-2}{1+q_m(\lambda)}
\end{equation}
for the pointwise conformal formula, and
\begin{equation}\label{eq:qf-tau-m-lambda-section3}
   \overline\tau_m(\lambda)
   :=
   \frac2{\overline\alpha_m(\lambda)}
   =
   \frac{m-2}{1+\overline q_m(\lambda)}
\end{equation}
for the quadratic-form conformal formula.
\begin{proposition}
\label{prop:packing-conformal-blowup-scalar}
Assume that \(X=\mathcal R\sqcup\mathcal S\) is an \(m\)-dimensional AM--PI space, and that
\[
   L_{f,\lambda}^{(m)}=L_\omega
\]
for some \(\omega\in\mathcal W_{\rm adm}(X)\).  Assume
\[
   \lambda<\frac1m,
   \qquad
   \tau_m(\lambda)<\tau_{\rm pack}(\mathcal S),
\]
and
\[
   \Scal_f^{m,\lambda}(g)\ge\sigma
   \qquad\text{on }\mathcal R ,
\]
where \(\sigma\) is constant.  Then, for every \(\zeta>0\), there exists \(w\in C^\infty_{\rm loc}(\mathcal R)\), \(w>0\), such that
\[
   \widetilde g=w^{\alpha_m(\lambda)}g,
   \qquad
   \widetilde f=f+\beta_m(\lambda)\log w
\]
satisfy
\[
   \Scal_{\widetilde f}^{m,\lambda}(\widetilde g)
   \ge
   \sigma-\zeta
   \qquad\text{on }\mathcal R .
\]
Moreover, \((\mathcal R,\widetilde g)\) is complete toward \(\mathcal S\).
\end{proposition}

\begin{proof}
Put \(\alpha=\alpha_m(\lambda)\), \(\beta=\beta_m(\lambda)\), and \(\tau=\tau_m(\lambda)=2/\alpha\).  By Proposition~\ref{prop:global-packing-green-blowup-section3}, there are \(\Theta\in C^\infty_{\rm loc}(\mathcal R)\), \(A>0\), and \(B\ge0\) such that
\begin{equation}\label{eq:pointwise-blowup-theta-section3}
   -L_\omega\Theta\ge A\Theta^{1+\alpha}-B,
   \qquad
   \Theta(x)\ge c_K\operatorname{dist}(x,\mathcal S)^{-2/\alpha}
\end{equation}
for \(x\in K\cap\mathcal R\) sufficiently close to \(\mathcal S\), for every \(K\Subset X\).

For \(\varepsilon>0\), set
\[
   w=1+\varepsilon\Theta,
   \qquad
   \widetilde g=w^\alpha g,
   \qquad
   \widetilde f=f+\beta\log w .
\]
The pointwise conformal formula gives
\[
   w^\alpha
   \Scal_{\widetilde f}^{m,\lambda}(\widetilde g)
   =
   \Scal_f^{m,\lambda}(g)
   -
   \mathcal K_m(\lambda)w^{-1}L_{f,\lambda}^{(m)}w .
\]
Since
\[
   L_{f,\lambda}^{(m)}=L_\omega,
   \qquad
   L_\omega w=\varepsilon L_\omega\Theta,
\]
we get from \eqref{eq:pointwise-blowup-theta-section3} and the fact $w\geq 1$ that
\[
   \Scal_{\widetilde f}^{m,\lambda}(\widetilde g)
   \ge
   \sigma w^{-\alpha}
   +
   \mathcal K_m(\lambda)A\varepsilon\Theta^{1+\alpha}w^{-1-\alpha}
   -
   \mathcal K_m(\lambda)B\varepsilon .
\]
As in the proof of Lemma~\ref{lem:conformal-absorption-section3}, we have
\[
   \sigma_+\bigl(1-w^{-\alpha}\bigr)
   -
   \mathcal K_m(\lambda)A
   \varepsilon\Theta^{1+\alpha}w^{-1-\alpha}
   \le C\varepsilon .
\]
Therefore \( \Scal_{\widetilde f}^{m,\lambda}(\widetilde g)\ge \sigma-C\varepsilon-\mathcal K_m(\lambda)B\varepsilon\). Choosing \(\varepsilon>0\) sufficiently small gives \( \Scal_{\widetilde f}^{m,\lambda}(\widetilde g)\ge\sigma-\zeta\). Completeness follows from the growth estimate in \eqref{eq:pointwise-blowup-theta-section3} and Lemma~\ref{lem:conformal-completeness-section3}.
\end{proof}

\begin{proposition}
\label{prop:qf-conformal-blowup-section3}
Assume that \(X=\mathcal R\sqcup\mathcal S\) is an \(m\)-dimensional AM--PI space, and that
\[
   \overline L_{f,\lambda}^{(m)}=L_\omega
\]
for some \(\omega\in\mathcal W_{\rm adm}(X)\).  Assume
\[
   \lambda<\frac1{m+1},
   \qquad
   \overline\tau_m(\lambda)<\tau_{\rm pack}(\mathcal S),
\]
and assume that
\[
   Q_{a_{m+1}(\lambda),\sigma}^{g,f}(\varphi)\ge0
   \qquad
   \text{for all }\varphi\in C_c^\infty(\mathcal R),
\]
where \(\sigma\) is constant.  Then, for every \(\zeta>0\), there exists \(w\in C^\infty_{\rm loc}(\mathcal R)\), \(w>0\), such that
\[
   \widetilde g=w^{\overline\alpha_m(\lambda)}g,
   \qquad
   \widetilde f=f+\overline\beta_m(\lambda)\log w
\]
satisfy
\[
   Q_{a_{m+1}(\lambda),\,\sigma-\zeta}^{\widetilde g,\widetilde f}(\psi)\ge0
   \qquad
   \text{for all }\psi\in C_c^\infty(\mathcal R).
\]
Moreover, \((\mathcal R,\widetilde g)\) is complete.
\end{proposition}

\begin{proof}
Put
\[
   \alpha=\overline\alpha_m(\lambda),\qquad
   \beta=\overline\beta_m(\lambda),\qquad
   s=\overline s_m(\lambda),
   \qquad
   \tau=\overline\tau_m(\lambda)=2/\alpha .
\]
By Proposition~\ref{prop:global-packing-green-blowup-section3}, there are
\(\Theta\in C^\infty_{\rm loc}(\mathcal R)\), \(A>0\), and \(B\ge0\) such that
\begin{equation}\label{eq:qf-blowup-theta-section3}
   -L_\omega\Theta\ge A\Theta^{1+\alpha}-B,
   \qquad
   \Theta(x)\ge c_K\operatorname{dist}(x,\mathcal S)^{-2/\alpha}
\end{equation}
for \(x\in K\cap\mathcal R\) sufficiently close to \(\mathcal S\), for every
\(K\Subset X\).

For \(\varepsilon>0\), set
\[
   w=1+\varepsilon\Theta,
   \qquad
   \widetilde g=w^\alpha g,
   \qquad
   \widetilde f=f+\beta\log w .
\]
For \(\varphi\in C_c^\infty(\mathcal R)\), the quadratic-form conformal formula gives
\[
\begin{split}
Q_{a_{m+1}(\lambda),\sigma}^{\widetilde g,\widetilde f}
(w^{-s}\varphi)  = &
Q_{a_{m+1}(\lambda),\sigma}^{g,f}(\varphi)
-
\frac{\overline{\mathcal K}_m(\lambda)}2
\int_{\mathcal R} e^{-f}w^{-1}
\overline L_{f,\lambda}^{(m)}w\,\varphi^2\,d\mu_g  \\
&
-\frac12 \int_{\mathcal R} e^{-f}\sigma(w^\alpha-1)\varphi^2\,d\mu_g .
\end{split}
\]
Using
\[
   Q_{a_{m+1}(\lambda),\sigma}^{g,f}\ge0,
   \qquad
   \overline L_{f,\lambda}^{(m)}=L_\omega,
   \qquad
   L_\omega w=\varepsilon L_\omega\Theta,
\]
and \eqref{eq:qf-blowup-theta-section3}, we obtain
\[
\begin{split}
&Q_{a_{m+1}(\lambda),\sigma-\zeta}^{\widetilde g,\widetilde f}
(w^{-s}\varphi)  \\
\ge &
\int_{\mathcal R} e^{-f}
\left[
   \frac{\overline{\mathcal K}_m(\lambda)}2
   A\varepsilon\Theta^{1+\alpha}w^{-1}
   -
   \frac{\overline{\mathcal K}_m(\lambda)}2
   B\varepsilon w^{-1}
   -
   \frac12\sigma(w^\alpha-1)
   +
   \frac{\zeta}{2}w^\alpha
\right]\varphi^2\,d\mu_g .
\end{split}
\]
Applying Lemma~\ref{lem:conformal-absorption-section3} with \(\mathfrak K=\frac{\overline{\mathcal K}_m(\lambda)}2\) shows that the bracket is nonnegative for \(\varepsilon>0\) sufficiently small. Hence
\[
   Q_{a_{m+1}(\lambda),\sigma-\zeta}^{\widetilde g,\widetilde f}
   (w^{-s}\varphi)\ge0
   \qquad
   \text{for all }\varphi\in C_c^\infty(\mathcal R).
\]
Since multiplication by \(w^{-s}\) is a bijection on \(C_c^\infty(\mathcal R)\), the required quadratic-form inequality follows. Completeness follows from \eqref{eq:qf-blowup-theta-section3} and Lemma~\ref{lem:conformal-completeness-section3}.
\end{proof}

\begin{remark}
\label{rem:codimension-thresholds-section3}
Suppose that \(\mathcal S\) satisfies the local packing condition of codimension \(c\).  Then \(\tau_{\rm pack}(\mathcal S)\ge c-2\).

For the pointwise blow-up, the condition \(\tau_m(\lambda)<c-2\) is automatic for all \(\lambda\leq 1/m\) if \(c\ge (m+2)/2\).  If \(3-2/m<c<(m+2)/2\), then it is equivalent to
\[
   \lambda<\lambda_{m,c}^{\rm pt},
   \qquad
   \lambda_{m,c}^{\rm pt}
   =
   \frac{c^2-2c-m+2}{(c-1)(m(c-3)+2)} .
\]

For the quadratic-form blow-up, the condition \(\overline\tau_m(\lambda)<c-2\) is automatic for all \(\lambda\leq 1/(m+1)\) if \(c\ge (m+2)/2\).  If
\[
   2+\frac{m-2}{1+\sqrt{m(m-1)/2}}<c<\frac{m+2}{2},
\]
then it is equivalent to
\[
   \lambda<\lambda_{m,c}^{\rm qf},
   \qquad
   \lambda_{m,c}^{\rm qf}
   =
   \frac{c^2-2m}{(m+1)c^2-4mc+2m}.
\]
In particular, for \(c=7\), the condition is automatic when \(m\le12\), while for \(m>12\) it becomes
\[
   \lambda<\lambda_{m,7}^{\rm qf}
   =
   \frac{49-2m}{23m+49}.
\]
\end{remark}

\section{Cube inequality}
\label{sec:cube-inequality}

Throughout this section, we denote \(I=[-1,1]\).

Let \(Y^r\) be either a point, in which case \(r=0\), or a connected complete oriented smooth manifold without boundary. Let \(k\geq 0\) and set \(n=r+k\). Let \((X^n,\partial X,g)\) be a connected complete oriented smooth Riemannian manifold with boundary, and let \(F^0=(\eta,\theta_1,\ldots,\theta_k):X\to Y\times I^k\) be a proper continuous map, smooth on the interior, such that \(F^0:(X,\partial X)\to (Y\times I^k, Y\times\partial I^k)\) has nonzero degree. If \(Y\) is a point, we omit the factor \(\eta\).

For \(i=1,\ldots,k\), set
\[
C_i^\pm=(F^0)^{-1}(Y\times I^{i-1}\times\{\pm 1\}\times I^{k-i}),\qquad d_i=\operatorname{dist}_g(C_i^-,C_i^+).
\]
Fix auxiliary numbers \(0<\bar{d}_i<d_i\). We choose smooth functions \(\tau_i: X\to [-1,1]\), smooth on the interior, such that \(\tau_i=\pm 1\) on a neighborhood of \(C_i^\pm\) and \(\operatorname{Lip}_g(\tau_i)\leq 2/\bar{d}_i\). 

Set \(F=(\eta,\tau_1,\ldots,\tau_k):X\to Y\times I^k\). Then \(F\) is proper, and the straight-line homotopy in the cubical coordinates from \(F^0\) to \(F\) is a proper homotopy of pairs, hence \(F\) has the same nonzero degree as \(F^0\).  Since we let \(\bar{d}_i\uparrow d_i\) in the final step for the cube inequality, we rename \(\bar{d}_i\) by \(d_i\) for simplicity.

Unless explicitly stated otherwise, all closures in this section are taken in the ambient space \(X\), with respect to the topology induced by \(g\).

\subsection{Descent data}
For \(m\geq r\), set \(B_m=Y\times I^{m-r}\). If \(r=0\), this means \(B_m=I^m\). We write \(\partial B_m=Y\times \partial I^{m-r}\), and when \(m=r\), we use \(I^0=\{\operatorname{pt}\}\), so \(B_r=Y\) and \(\partial B_r=\emptyset\). We write
\[
   F_m=(\eta,\tau_1,\ldots,\tau_{m-r}):X\to B_m.
\]
By a harmless abuse of notation, the same symbol \(F_m\) denotes its restriction to any subset of \(X\) whenever the domain is clear. 

We use the usual locally finite relative degree for proper maps to possibly noncompact manifolds with boundary.
\begin{definition}
     If a map of pairs
\[
   G:(M,\partial M)\to (N,\partial N)
\]
is given and \(U\subset N\) is open, we say that \(G\) is \emph{proper over} \(U\) if the restricted map
\[
   G:G^{-1}(U)\to U
\]
is proper, with the relative boundary understood as \(\partial M\cap G^{-1}(U)\to \partial N\cap U\).  When \(U\) is connected and oriented, the degree of \(G\) over \(U\), denoted \(\deg_U(G)\), is defined by
\[
   G_*[G^{-1}(U),\partial M\cap G^{-1}(U)]_{\rm lf}
   =
   \deg_U(G)\cdot\,[U,\partial N\cap U]_{\rm lf}.
\]
\end{definition}

 The degree has a natural interpretation in differential topology as well.
For a regular value \(y\in U\), we write
\[
   \deg_y(G)=\sum_{p\in G^{-1}(y)}\operatorname{sgn}_p(\mathrm dG).
\]
If \(G\) is proper over \(U\), this number is finite and equals the degree over \(U\). For background on locally finite homology, see, for example, \cite[Chs.~10--11]{Geoghegan2008}.

\begin{definition}
    Let $m\geq r$. An $m$-dimensional datum controlled over \(B_m\) consists of a smooth embedded \(m\)-manifold with boundary \(X_m^\circ\subset X\), together with a metric \(g_m\), a smooth function \(f_m\), and a constant \(\sigma_m\).
    We set
\[
   X_m=\overline{X_m^\circ},
   \qquad
   D_m=X_m\setminus X_m^\circ .
\]
The datum is required to satisfy the following conditions:
\begin{itemize}
    \item \(X_m\cap \tau_i^{-1}(\pm 1)=\emptyset\) for all \(i>m-r\), and
    \[
       \dim_{\mathcal H}D_m\le m-7 .
    \]    
    \item $X_m^\circ$ is properly embedded in $X\setminus D_m$. 
 
    \item The metric \(g_m\) is smooth and complete on \(X_m^\circ\), and
    \[
       \Scal^{m,\lambda}_{f_m}(g_m)\geq \sigma_m .
    \]
    Moreover, for every point \(p\in D_m\) there are \(c_p,r_p>0\) such that
    \[
       g_m\geq c_p d_g(\cdot,p)^{-2}g
       \qquad \text{on}\qquad
       \{x\in X_m^\circ:d_g(x,p)<r_p\}.
    \]
      
    \item The restriction
\[
   F_m:(X_m^\circ,\partial X_m^\circ)\to(B_m,\partial B_m)
\]
is proper over \(B_m\setminus F_m(D_m)\) and has nonzero (locally finite relative) degree over this open set.
    \item We have
    \[
       \operatorname{Lip}_{g_m}(\tau_i)\leq \frac2{d_i},
       \qquad 1\leq i\leq m-r .
    \]
\end{itemize}
It follows from the preceding assumptions that \(F_m(D_m)\) is closed in \(B_m\), and
\[
   \dim_{\mathcal H}F_m(D_m)\le m-7 .
\]
\end{definition}

\subsection{Finite-width localization}

We shall use the following fact.
\begin{lemma}\label{lem:good-coordinate-levels}
Let \(U\) be a smooth manifold, let \(A\subset U\) be closed with \(\dim_{\mathcal H}A\le d\), and let \(\tau:U\to\mathbb R^q\) be smooth.  Then for almost every
\(a\in\mathbb R^q\),
\[
   \dim_{\mathcal H}
   \bigl(A\cap\tau^{-1}(a)\bigr)
   \le d-q .
\]
\end{lemma}

\begin{proof}
This is the standard slicing theorem for Hausdorff dimension; see Mattila~\cite[Theorem~7.7]{Mattila1995}.
\end{proof}

\begin{corollary}\label{cor:slicing}
Let an \(m\)-dimensional datum be controlled over \(B_m=Y\times I^{m-r}\), and suppose \(m>r\).  Put \(j=m-r\).  Then for almost every \(a\in(-1,1)\), the set \(X_m\cap\tau_j^{-1}(a)\) has the regular-singular decomposition
\[
   X_m\cap\tau_j^{-1}(a)
   =
   (X_m^\circ\cap\tau_j^{-1}(a))
   \sqcup
   (D_m\cap\tau_j^{-1}(a)),
\]
where \(X_m^\circ\cap\tau_j^{-1}(a)\) is a smooth hypersurface in \(X_m^\circ\),
\[
   \overline{X_m^\circ\cap\tau_j^{-1}(a)}
   \subset
   X_m\cap\tau_j^{-1}(a)
   \quad\text{in }X,
\]
and \(\dim_{\mathcal H}(D_m\cap\tau_j^{-1}(a))\le m-8\).
\end{corollary}

From now on, fix such a good level \(a\in(-1,1)\), and set
\[
   \widetilde X_{m-1}=X_m\cap\tau_j^{-1}(a),\quad
   \widetilde X_{m-1}^\circ=X_m^\circ\cap\tau_j^{-1}(a),\quad
   \widetilde D_{m-1}=D_m\cap\tau_j^{-1}(a).
\]

We have the following lemma on finite-width localization.

\begin{lemma}\label{lem:finite-width-localization}
For every \(W>0\), the neighborhood
\[
   N_W=\{x\in X_m^\circ:
      d_{g_m}(x,\widetilde X_{m-1}^\circ)\le W\}
\]
satisfies \(\overline{N_W}\cap D_m\subset\widetilde D_{m-1}\).
\end{lemma}

\begin{proof}
Suppose \(q\in D_m\) and \(q_i\in N_W\) with \(q_i\to q\).  For each \(i\), there are \(p_i\in\widetilde X_{m-1}^\circ\) and a curve \(\gamma_i\subset X_m^\circ\) from \(q_i\) to \(p_i\) with \(L_{g_m}(\gamma_i)\le W+1\).  We claim $p_i\to q$ as $i\to \infty$. Once this is proved, then we have
    $$q\in \overline{\widetilde X_{m-1}^\circ} \cap D_m\subset\widetilde D_{m-1}.$$
     To prove the claim, we argue by contradiction. Suppose that $p_i\to p\neq q$. Let us parameterize the curve $$\gamma_i:[0,s_i]\to X_m^\circ$$ to have $g$-unit speed. For $i$ large enough, we have \(s_i\geq \frac{1}{2}d_g(p,q)\).  Recall that there are positive constants $c_q,r_q>0$ such that 
    \[g_m\geq c_q d_g(\cdot,q)^{-2}g\qquad\text{on}\qquad \{x\in X_m^\circ:d_g(x,q)<r_q\}.\]
    
    Without loss of generality, we may assume $r_q<d_g(p,q)/2$. Then for $i$ large enough we can compute
    \[
        L_{g_m}(\gamma_i)\geq \int_0^{r_q-d_g(q_i,q)} |\gamma_i'(s)|_{g_m}\,\mathrm ds
        \geq \sqrt {c_q}\log\frac{r_q}{2d_g(q_i,q)}>W.
    \]
    This gives the desired contradiction, and so the claim is proved.
\end{proof}

\subsection{The \(\mu\)-bubble in a non-compact band}\label{subsec:mu-bubble-noncompact-band}

Assume \(m>r\) and fix a good level \(a\in(-1,1)\) as above. For \(W>0\), set
\[
   N_W=\{x\in X_m^\circ\mid
      d_{g_m}(x,\widetilde X_{m-1}^\circ)\le W\}.
\]
After smoothing this finite-width metric band and the signed distance to \(\widetilde X_{m-1}^\circ\), for every \(L<W\) we obtain a smooth \(L\)-band \(\Omega\subset N_W\) and a smooth function
\[
   s:(\Omega,\partial_\pm\Omega)\to([-L,L],\pm L),\quad
   s^{-1}(\pm L)=\partial_\pm\Omega,\quad
   s^{-1}(0)=\widetilde X_{m-1}^\circ\cap\Omega,\quad
   |ds|_{g_m}\le1.
\]
The side boundary is
\[
   \partial_s\Omega
   =\partial\Omega\cap F_{m-1}^{-1}(\partial B_{m-1}),
\]
and the faces \(\partial_\pm\Omega\) meet \(\partial_s\Omega\) in acute inner angles.

Let \(a_L=\pi/(2L)\) and \(\Phi(s)=-\mu a_L\tan(a_Ls)\), where \(\mu=2(1-\lambda)\).  Then \(\Phi(s)\to+\infty\) as \(s\to -L\), and
\(\Phi(s)\to-\infty\) as \(s\to L\). Since \(|d\Phi|_{g_m}\le \mu a_L^2(1+\tan^2(a_Ls))\), for every unit vector \(\xi\), we have
\begin{equation}\label{eq:tangent-potential-new}
\begin{split}
2\langle\nabla\Phi,\xi\rangle+\frac1{1-\lambda}\Phi^2
&\ge
-2\mu a_L^2(1+\tan^2(a_Ls))
+\frac{\mu^2}{1-\lambda}a_L^2\tan^2(a_Ls)\\
&\ge -2\mu a_L^2 .
\end{split}
\end{equation}

Fix an exhaustion of $\Omega$  by compact subbands, denoted by
\[\Omega_1\Subset\Omega_2\Subset\cdots\Subset\Omega\]
with \(\bigcup_j\Omega_j=\Omega\). Moreover, we can require \[\partial\Omega_j=\Gamma_j^-\cup\Gamma_j^+\cup\Gamma_j^s,\] where \(\Gamma_j^\pm=\partial\Omega_j\cap\partial_\pm\Omega\), and \(\Gamma_j^s=\overline{\partial\Omega_j\setminus(\Gamma_j^+\cup\Gamma_j^-)}\) meets \(\Gamma_j^\pm\) in acute  inner angles.  Set
\[E_0=\Omega\cap\{s\leq 0\},\] and let \(\mathcal C\) be the collection of all locally finite perimeter sets \(E\) such that \((E\Delta E_0)\cap\Omega_j\Subset\Omega^\circ\) for all \(j\).

\begin{proposition}\label{prop:mububble-in-band}
There exists a locally finite perimeter set \(E\in\mathcal C\) which is a local minimizer of
\[
   \mathcal F(E)
   =
   \int_{\partial^*E}e^{-f_m}\,d\mathcal H^{m-1}_{g_m}
   -
   \int_\Omega(\chi_E-\chi_{E_0})\Phi e^{-f_m}\,d\mathcal H^m_{g_m}.
\]
Let \(\Sigma=\partial E\cap\Omega^\circ\).  With respect to the unit normal \(\nu\) pointing from \(E\) to its complement, on \(\Sigma_{\rm reg}=\partial^*E\) one has \(H_\Sigma-\langle\nabla f_m,\nu\rangle=\Phi\).  The weighted stability inequality holds on \(\Sigma_{\rm reg}\).  Moreover, $\Sigma_{\mathrm{reg}}$ is a properly embedded hypersurface in $\Omega$, \((\Sigma,d_{g_m}|_\Sigma,\mathcal H^{m-1}_{g_m})\) is a local AM--PI space of dimension \(m-1\), and the singular set \(D^\Sigma=\Sigma\setminus\Sigma_{\rm reg}\) satisfies the local scale-uniform codimension-seven packing estimate \(\mathsf P_{m-8}^{\Sigma}(D^\Sigma)\).  In particular \(\dim_{\mathcal H}D^\Sigma\le m-8\). If $\Scal^{m,\lambda}_{f_m}(g_m)\geq \sigma_m$, then the quadratic form
\[
   Q_{a_m(\lambda),\sigma_L}^{g_\Sigma,f_\Sigma}
\]
is nonnegative on \(C_c^\infty(\Sigma_{\rm reg})\), where $g_\Sigma$ is the induced metric on $\Sigma_{\rm reg}$ from $(\Omega,g_m)$, $f_\Sigma=f_m|_{\Sigma_{\rm reg}}$, and
$$\sigma_L=\sigma_m-(1-\lambda)\left(\frac{\pi}{L}\right)^2.$$

\end{proposition}

\begin{proof}
For \(E\in\mathcal C_j:=\{E\in\mathcal C:
E\Delta E_0\Subset\Omega_j\setminus(\Gamma_j^+\cup\Gamma_j^-)\}\), define
\[
   \mathcal A_{f_m}(E;\Omega_j)
   =
   \int_{\partial^*E\cap\Omega_j^\circ}e^{-f_m}\,d\mathcal H^{m-1}_{g_m},
   \quad
   \mathcal V_{\Phi,f_m}(E;\Omega_j)
   =
   \int_{\Omega_j}(\chi_E-\chi_{E_0})\Phi e^{-f_m}\,d\mathcal H^m_{g_m}.
\]
We minimize
\[
   \mathcal F_j(E)=\mathcal A_{f_m}(E;\Omega_j)
   -
   \mathcal V_{\Phi,f_m}(E;\Omega_j).
\]
It follows from \cite[Proof of Theorem 2]{Gromov-Zhu} that there is a minimizer $E_j\in \mathcal C_j$ of the functional $\mathcal F_j$. By comparison, we know that for each fixed compact subband \(\Omega_l\) there is a positive constant \(c_l\), independent of \(j\), such that
\[P(E_j;\Omega_l)\leq c_l.\]
As a consequence, after passing to a subsequence, \(E_j\) converge to a locally finite perimeter set $E$.

Let us show $E\in \mathcal C$. It suffices to show that for any point $p\in \partial_\pm \Omega$ there is a neighborhood such that $\partial E_j$ does not enter for all $j$ large. To see this, we take a small geodesic ball $B_p$ of $p$ in $\partial\Omega$ and a smooth interior positive function $h_p$ on $B_p$ with $h_p=0$ on $\partial B_p$. Clearly, the $(th_p)$-graphs $L_t$ over $B_p$ for $t$ small form a smooth family of hypersurfaces satisfying
$$H_{L_t}-\langle \nabla f_m,\nu_{L_t}\rangle<\Phi\mbox{ if }p\in \partial_-\Omega$$
and
$$H_{L_t}-\langle \nabla f_m,\nu_{L_t}\rangle>\Phi\mbox{ if }p\in \partial_+\Omega.$$
From a barrier argument we conclude that for $j$ large enough $\partial E_j$ cannot intersect $L_t$ and so avoid a neighborhood of $p$. 

To sum up, we have constructed a locally finite perimeter set $E\in \mathcal C$ such that $E$ is a local minimizer of the functional 
\[
\mathcal F(E):=\int_{\partial^*E}e^{-f_m}\,\mathrm d\mathcal H^{m-1}_{g_m}-\int_{\Omega}(\chi_E-\chi_{E_0})\Phi e^{-f_m}\,\mathrm d\mathcal H^m_{g_m}.
\]
Then the first and second variation give the desired mean curvature equation and the weighted stability inequality on $\Sigma_{\rm reg}$. We remark that the barrier argument above also implies locally uniform bound for the mean curvature of $\partial E$. Combined with the local perimeter bound and the monotonicity formula, we conclude that $\Sigma_{\mathrm{reg}}$ is a properly embedded hypersurface in $\Omega$. The conclusion that \((\Sigma,d_{g_m}|_\Sigma,\mathcal H^{m-1}_{g_m}) \) is AM--PI follows from the argument of Bombieri--Giusti~\cite{BombieriGiusti1972}. Finally, the quantitative estimates of Naber--Valtorta~\cite[Theorems~1.3 and~1.6]{NaberValtorta2020} give the corresponding tubular-neighborhood volume bound for \(D^\Sigma\), and in the relative case we use Edelen~\cite{Edelen2020}.  Since \(\Sigma\) is locally Ahlfors \((m-1)\)-regular, this tubular estimate is equivalent to \(\mathsf P_{m-8}^{\Sigma}(D^\Sigma)\). The last statement follows from Proposition~\ref{prop:smooth-descent-pointwise} and the estimate for
\(\Phi\).
\end{proof}

\subsection{Descent}

We first record the admissible range of the parameter.  For a positive integer \(n\), set
\[
  \Lambda_n:=
  \begin{cases}
    \displaystyle \frac1n, & n\le13,\\[0.8em]
    \displaystyle \frac{51-2n}{23n+26}, & n\ge14 .
  \end{cases}
\]
This is the range needed for the conformal blow-up steps in the descent. Indeed, in the step \(m\to m-1\), the new \(\mu\)-bubble singular set lies in an  \((m-1)\)-dimensional AM--PI space and has codimension seven.  Thus the
quadratic-form blow-up is applied in dimension \(m-1\).  By Remark~\ref{rem:codimension-thresholds-section3}, the condition is automatic when \(m-1\le12\), while for \(m-1\ge13\) it is
\[
   \lambda
   <
   \lambda_{m-1,7}^{\rm qf}
   =
   \frac{51-2m}{23m+26}.
\]
The most restrictive case is \(m=n\).  Hence \(\lambda<\Lambda_n\) guarantees that every intermediate blow-up is allowed.

Let
\[
   m_0=
   \begin{cases}
      1, & r=0,\\
      r, & r\ge1 .
   \end{cases}
\]

\begin{proposition}\label{prop:controlled-descent}
Assume \(\lambda<\Lambda_n\) and
\[
   \Scal_f^{n,\lambda}(g)\ge \sigma
\]
on \(X\). For each \(m_0\le m\le n\) and every \(\varepsilon>0\), there is an \(m\)-dimensional datum controlled over
\[
   B_m=Y\times I^{m-r}
\]
such that
\[
   \sigma_m
   \ge
   \sigma
   -
   4\pi^2(1-\lambda)
   \sum_{i=m-r+1}^{k}d_i^{-2}
   -
   \varepsilon,
\]
where the sum is empty if \(m-r+1>k\).
\end{proposition}

\begin{proof}
Let
\[
   c=\deg(F)\ne0 .
\]
For \(m=n\), take
\[
   X_n^\circ=X,
   \qquad
   g_n=g,
   \qquad
   f_n=f,
   \qquad
   \sigma_n=\sigma .
\]
Then \(D_n=\emptyset\), no target set is deleted from $B_n$, and \(F_n=F\) has degree \(c\).

Assume \(m>m_0\), put
\[
   j=m-r,
   \qquad
   B_m=Y\times I^j,
\]
and suppose that an \(m\)-dimensional datum controlled over \(B_m\) has been constructed with
\[
   \sigma_m
   \ge
   \sigma
   -
   4\pi^2(1-\lambda)
   \sum_{i=j+1}^{k}d_i^{-2}
   -
   \frac{\varepsilon}{4},
\]
and such that the restricted map in the datum has degree \(c\).  Since \(m>m_0\), one has \(j\ge1\).

\vspace{3mm}
\noindent {\bf Step 1: Slicing.}
\vspace{3mm}

Choose \(a\in(-1,1)\), with \(|a|\) later taken arbitrarily small, from the full-measure set of good levels given by Corollary~\ref{cor:slicing}.  Set
\[
   \widetilde X_{m-1}=X_m\cap\tau_j^{-1}(a),
   \qquad
   \widetilde X_{m-1}^\circ=X_m^\circ\cap\tau_j^{-1}(a),
   \qquad
   \widetilde D_{m-1}=D_m\cap\tau_j^{-1}(a).
\]
We identify the target slice \(\{t_j=a\}\subset Y\times I^j\) with \(B_{m-1}=Y\times I^{j-1}\).

\vspace{3mm}
{\bf Claim:} The restriction \[F_{m-1}|_{\widetilde X_{m-1}^\circ}\]  is proper over \(B_{m-1}\setminus F_{m-1}(\widetilde D_{m-1})\).
\vspace{3mm}

Indeed, for any compact \(K\subset B_{m-1}\setminus F_{m-1}(\widetilde D_{m-1})\), if \(K\times\{a\}\) met \(F_m(D_m)\) at some point, then the corresponding preimage in $\widetilde X_{m-1}^\circ$ would lie in \(D_m\cap\tau_j^{-1}(a)=\widetilde D_{m-1}\), yielding a contradiction. Hence \(K\times\{a\}\subset B_m\setminus F_m(D_m)\), and
\[
   \left(F_{m-1}|_{\widetilde X_{m-1}^\circ}\right)^{-1}(K)
   =
   X_m^\circ\cap F_m^{-1}(K\times\{a\})
\]
is compact by the properness of the map $F_m:(X_m^\circ,\partial X_m^\circ)\to (B_m,\partial B_m)$ over $B_m\setminus F_m(D_m)$ in the \(m\)-dimensional datum. We orient \(\widetilde X_{m-1}^\circ\) so that, for regular values \(y\notin F_{m-1}(\widetilde D_{m-1})\),
\[
   \deg_y\!\left(F_{m-1}|_{\widetilde X_{m-1}^\circ}\right)
   =
   \deg_{(y,a)}(F_m).
\] Thus
\[
   F_{m-1}:\widetilde X_{m-1}^\circ\to B_{m-1}
\]
has degree \(c\) over \(B_{m-1}\setminus F_{m-1}(\widetilde D_{m-1})\).

\vspace{3mm}
\noindent{\bf Step 2: Finite-width localization.}
\vspace{3mm}

Choose \(W<\frac{1-|a|}{2}d_j\) and set
\[
   N_W=\{x\in X_m^\circ:
   d_{g_m}(x,\widetilde X_{m-1}^\circ)\le W\}.
\]
Since \(\operatorname{Lip}_{g_m}(\tau_j)\le2d_j^{-1}\), the closure $\overline{N_W}$ of \(N_W\) is disjoint from the two faces \(\{\tau_j=\pm1\}\).

As in the beginning of Subsection~\ref{subsec:mu-bubble-noncompact-band}, for every \(L<W\) we choose a closed smooth \(L\)-band \(\Omega\subset N_W\), with side boundary \(\partial_s\Omega=\partial\Omega\cap F_{m-1}^{-1}(\partial B_{m-1})\), and a smooth function
\[
   s:(\Omega,\partial_\pm\Omega)\to([-L,L],\pm L),
   \qquad
   s^{-1}(0)=\widetilde X_{m-1}^\circ\cap\Omega,
   \qquad
   |ds|_{g_m}\le1 .
\]
The smoothing is chosen so that \(\partial_\pm\Omega\) meet \(\partial_s\Omega\) in acute inner angles.

\vspace{3mm}
{\bf Claim:} the restriction of \(F_{m-1}\) to the finite-width band \(\Omega\) is proper over \(B_{m-1}\setminus F_{m-1}(\widetilde D_{m-1})\).  
\vspace{3mm}

Let \(K\subset B_{m-1}\setminus F_{m-1}(\widetilde D_{m-1})\) be compact.  Every sequence in \(\Omega\cap F_{m-1}^{-1}(K)\) has a subsequence converging in \(X\), as it lies in the compact set \(F^{-1}(K\times I^{k-j+1})\).  The limit belongs to the  \(\overline{\Omega}\subset X_m\). 
If it were outside \(\Omega\), then it would lie in \(D_m\) for \(\Omega\) is closed in \(X_m^\circ\).  Since \(\Omega\subset N_W\), Lemma~\ref{lem:finite-width-localization} then forces such a limit point to lie in \(\widetilde D_{m-1}\), contradicting \( K\cap F_{m-1}(\widetilde D_{m-1})=\emptyset \). Thus the restriction of \(F_{m-1}\) to \(\Omega\) is proper over \(B_{m-1}\setminus F_{m-1}(\widetilde D_{m-1})\).  

\vspace{3mm}
\noindent{\bf Step 3. $\mu$-bubble construction for $(m-1)$-dimensional datum.}
\vspace{3mm}

By Proposition~\ref{prop:mububble-in-band}, there exists a locally finite perimeter set \(E\) in the band such that \(\Sigma=\partial E\cap\Omega^\circ\) is a weighted \(\mu\)-bubble. On \(\Sigma_{\rm reg}=\partial^*E\), one has
\[
   H_\Sigma-\langle\nabla f_m,\nu\rangle=\Phi,
\]
and the weighted stability inequality holds. Moreover, \(\Sigma=\Sigma_{\rm reg}\sqcup D^\Sigma\) is an \((m-1)\)-dimensional AM--PI space satisfying \(\mathsf P_{m-8}^{\Sigma}(D^\Sigma)\). The quadratic form
\[
   Q_{a_m(\lambda),\sigma_L}^{g_\Sigma,f_\Sigma},
   \qquad
   \sigma_L=\sigma_m-(1-\lambda)\left(\frac{\pi}{L}\right)^2,
\]
is nonnegative on \(C_c^\infty(\Sigma_{\rm reg})\), where \(g_\Sigma\) is induced by \((\Omega,g_m)\) and \(f_\Sigma=f_m|_{\Sigma_{\rm reg}}\).

\vspace{3mm}
{\bf Claim:} $\Sigma_{\mathrm{reg}}$ is a properly embedded hypersurface in $X\setminus (\widetilde D_{m-1}\cup D^\Sigma)$.
\vspace{3mm}

Let $K\subset X\setminus(\widetilde D_{m-1}\cup D^\Sigma)$ be compact. We have the relation
$$\Sigma_{\mathrm{reg}}\cap K=\Sigma_{\mathrm{reg}}\cap\left(X_m^\circ\cap(K\cap\Omega)\right).$$
Note that $K\cap\Omega$ is a compact set in $X\setminus D_m$. Indeed, from the same argument as in the previous step, we have $K\cap \Omega=K\cap \overline \Omega$. The proper embeddedness of $X_m^\circ\subset X\setminus D_m$ and $\Sigma_{\mathrm{reg}}\subset \Omega\subset X_m^\circ$ implies that $\Sigma_{\mathrm{reg}}\cap K$ is a compact subset of $\Sigma_{\mathrm{reg}}$ with respect to the topology from the induced Riemannian metric. \vspace{3mm}

Let $\mathcal N$ be an open tubular neighborhood of $\partial X$, for which $\partial \mathcal N\setminus \partial X$ is transverse to $\Sigma_{\mathrm{reg}}$. Set
$$X_{m-1}^\circ=\Sigma_{\mathrm{reg}}\setminus\mathcal N,\quad X_{m-1}=\overline{X_{m-1}^\circ},\quad D_{m-1}=X_{m-1}\setminus X_{m-1}^\circ.$$
Let us verify the desired properties of our descent datum related to $X_{m-1}^\circ$.
Since $X_{m-1}\subset \overline{N_W}\subset X_m$, we know that \(X_{m-1}\cap \tau_i^{-1}(\pm 1)=\emptyset\) for all \(i\geq j\).
Note that we have $D_{m-1}\subset \tilde D_{m-1}\cup D^\Sigma$, and so we have 
$$\dim_{\mathcal H}(D_{m-1})\leq m-8=m-1-7$$ as required. Let us check that $X_{m-1}^\circ$ is properly embedded in $X\setminus D_{m-1}$. For any compact $K$ in $X\setminus D_{m-1}$, we have
$$X_{m-1}^\circ\cap K=\Sigma_{\mathrm{reg}}\cap (X_{m-1}\cap K).$$
Since  $X_{m-1}\cap K$ is a compact set in $X\setminus (\widetilde D_{m-1}\cup D^\Sigma)$, the properness of $\Sigma_{\mathrm{reg}}$ above yields that $X_{m-1}^\circ\cap K$ is a compact subset of $X_{m-1}^\circ$.

Recall that \(\tau_i=\pm 1\) on a neighborhood $\tilde C_i^{\pm}$ of \(C_i^\pm\). The neighborhood $\mathcal N$ can be taken small enough such that we have 
$$F_{m-1}( \Sigma\cap \overline{\mathcal N})\subset F_{m-1}(\cup_{i=1}^{j-1}\tilde C_i^\pm)\subset \partial B_{m-1}.$$
From the properness of $X_{m-1}^\circ$ in $X\setminus D_{m-1}$, we see that
$$F_{m-1}:(X_{m-1}^\circ,\partial X_{m-1}^\circ)\to (B_{m-1},\partial B_{m-1} )$$
is proper over $B_{m-1}\setminus F_{m-1}(D_{m-1})$.

Next we compute the degree of this restriction. Choose
\[
   y\in(\operatorname{int}B_{m-1})\setminus F_{m-1}(D_{m-1})
\]
to be a common regular value of the restrictions of \(F_{m-1}\) to
\[
   \widetilde X_{m-1}^\circ,\qquad \Omega^\circ,\qquad X_{m-1}^\circ,
\]
and such that
\[
   \deg_y\!\left(F_{m-1}|_{\widetilde X_{m-1}^\circ}\right)=c.
\]

Put \(\Gamma_y=\Omega\cap F_{m-1}^{-1}(y)\). Then \(\Gamma_y\) is an oriented one-dimensional manifold, with boundary lying
on \(\partial_-\Omega\cup\partial_+\Omega\). If \(H\subset\Omega\) is an oriented hypersurface meeting \(\Gamma_y\) transversely in finitely many points, we denote the corresponding oriented intersection number by
\[
   I(\Gamma_y,H).
\]
Since \(y\in (\operatorname{int}B_{m-1})\setminus F_{m-1}(D_{m-1})\), \(\Gamma_y\) does not meet the singular set \(D^\Sigma\) and $\Sigma\cap\mathcal N$. Hence we have \(\Gamma_y\cap\Sigma=\Gamma_y\cap X_{m-1}^\circ\).

Recall \(E_0=\Omega\cap\{s\le0\}\). We have \(\partial E_0\cap\Omega^\circ=\widetilde X_{m-1}^\circ\cap\Omega\). Since \(E\in\mathcal C\), replacing \(E_0\) by \(E\) does not change the oriented intersection number of their boundaries with $\Gamma_y$.
Hence
\[
   I(\Gamma_y,X_{m-1}^\circ)
   =
   I(\Gamma_y,\widetilde X_{m-1}^\circ)
   =
   \deg_y\!\left(F_{m-1}|_{\widetilde X_{m-1}^\circ}\right)
   =
   c.
\]
This means the map $$F_{m-1}:(X_{m-1}^\circ,\partial X_{m-1}^\circ)\to (B_{m-1},\partial B_{m-1} )$$
has degree $c$ over $B_{m-1}\setminus F_{m-1}(D_{m-1})$.

\vspace{3mm}
\noindent{\bf Step 4. Conformal blow-up.}
\vspace{3mm}

If \(m\le7\), then \(D^\Sigma=\emptyset\) and we set \(w\equiv1\). If \(m\ge8\), we apply Proposition~\ref{prop:qf-conformal-blowup-section3} in dimension \(m-1\) to $X_{m-1}$ with \(\lambda\le\Lambda_n\). We point out that $X_{m-1}$ is separated from the side boundary \(\partial_s\Omega\). Hence the AM--PI and packing conclusions of Proposition~\ref{prop:mububble-in-band} are ordinary interior statements there, and the Green-function construction of Section~\ref{sec:conformal-blowup-exceptional} applies.  Thus, for every \(\delta>0\), there is a smooth positive function \(w\) on  $X_{m-1}^\circ$ such that the metric and the weight
\[
   \bar g_{m-1}
   =
   w^{\overline\alpha_{m-1}(\lambda)}g_\Sigma,
   \qquad
   \bar f_{m-1}
   =
   f_\Sigma+\overline\beta_{m-1}(\lambda)\log w .
\]
satisfy that \(\bar g_{m-1}\) is smooth and complete on $X_{m-1}^\circ$, and
\[
   Q_{a_m(\lambda),\,\sigma_L-\delta}^{\bar g_{m-1},\bar f_{m-1}}(\varphi)
   \ge0
   \quad\text{ for all}\quad \varphi \in C_c^\infty(X_{m-1}^\circ) .
\]
By the Allegretto--Piepenbrink principle \cite{Allegretto1974,Piepenbrink1974,FischerColbrieSchoen1980}, applied on an exhaustion of $X_{m-1}^\circ$, there exists a positive smooth function \(\psi\) on $X_{m-1}^\circ$ such that
\begin{equation}\label{Eq: psi}
    \mathcal L_{m,\lambda,\sigma_L-\delta}^{\bar g_{m-1},\bar f_{m-1}}\psi
   \ge0 .
\end{equation}
Strictly speaking, to ensure the positivity of \(\psi\), we need to shrink \(X_{m-1}^\circ\) slightly further. Since this can be done by the same argument as before, we omit this routine technicality.

Set
\[
   g_{m-1}=\bar g_{m-1},
   \qquad
   f_{m-1}=\bar f_{m-1}-\log\psi ,\qquad \sigma_{m-1}=\sigma_L-\delta.
\]
We are ready to verify the rest desired properties of $(m-1)$-dimensional datum. As mentioned before, $g_{m-1}$ is smooth and complete on $X_{m-1}^\circ$. From \eqref{Eq: psi} and
Lemma \ref{lem:positive-supersolution} we have
\[
   \Scal_{f_{m-1}}^{m-1,\lambda}(g_{m-1})
   \ge
   \sigma_{m-1}
\]
pointwise on $X_{m-1}^\circ$. Let us check the quadratic growth of $g_{m-1}$ near \(D_{m-1}\).  If \(p\in D^\Sigma\), the quadratic-form conformal blow-up gives, after shrinking \(r_p\),
\[
   g_{m-1}\ge c_p\,d_{g}(\cdot,p)^{-2}\,g|_{X_{m-1}^\circ}
\]
near \(p\), by \eqref{eq:qf-blowup-theta-section3}.  If \(p\in\widetilde D_{m-1}\), the datum at level \(m\) gives the corresponding quadratic lower bound for \(g_m|_{X_{m-1}^\circ}\).  Since \(g_{m-1}\ge g_m|_{X_{m-1}^\circ}\), the same bound holds for \(g_{m-1}\). Since $g_{m-1}\geq g|_{X_{m-1}^\circ}$, we have
$$\mathrm{Lip}_{g_{m-1}}\tau_i\leq \frac{2}{d_i},
       \qquad 1\leq i\leq j-1.$$
\vspace{3mm}
We have constructed the desired $(m-1)$-dimensional datum. Since we can make $L$ arbitrarily close to $d_j/2$ and $\delta$ arbitrarily close to $0$, we can guarantee
$$\sigma_{m-1}\geq \sigma_m-4\pi^2(1-\lambda)d_j^{-2}-\frac{3\varepsilon}{4}\geq \sigma
   -
   4\pi^2(1-\lambda)
   \sum_{i=j}^{k}d_i^{-2}
   -
   \varepsilon. $$
   This completes the induction and hence the proof of the proposition.
\end{proof}

\subsection{Terminal cases}

\begin{lemma}\label{lem:one-dimensional-case}
Let \(X_1^\circ\), together with \(g_1,f_1\), and \(\sigma_1\), be a one-dimensional datum controlled over \(I\), and suppose \(\operatorname{Lip}_{g_1}(\tau_1)\le2/d_1\).  Then
\[
   \sigma_1\le 4\pi^2(1-\lambda)d_1^{-2}.
\]
\end{lemma}

\begin{proof}
Since \(\dim_{\mathcal H}D_r\le r-7\), we have \(D_1=\emptyset\).  The nonzero degree condition gives a connected component \(\Gamma\subset X_1\) carrying nonzero degree.  Thus \(\Gamma\) joins the two faces \(\tau_1^{-1}(-1)\) and \(\tau_1^{-1}(1)\).  If \(\mathcal L\) is the \(g_1\)-length of \(\Gamma\), then \(\mathcal L\ge d_1\).

Parametrize \(\Gamma\) by arclength \(s\in(-\mathcal L/2,\mathcal L/2)\), and set \(F=f_1|_\Gamma\).  The lower bound \(\Scal_{f_1}^{1,\lambda}(g_1)\ge\sigma_1\) gives \(2F''-(1-\lambda)^{-1}(F')^2\ge\sigma_1\). Put \(a=\pi/\mathcal L\) and \(\Phi(s)=2(1-\lambda)a\tan(as)\).  Since \(\Phi\to-\infty\) at the left end and \(\Phi\to+\infty\) at the right end, \(F'-\Phi\) has a first zero \(p\).  At this point \(F'(p)=\Phi(p)\) and \(F''(p)\le\Phi'(p)\).  Hence
\[
   \sigma_1
   \le
   2F''(p)-\frac1{1-\lambda}(F'(p))^2
   \le
   2\Phi'(p)-\frac1{1-\lambda}\Phi(p)^2
   =
   4\pi^2(1-\lambda)\mathcal L^{-2}
   \le
   4\pi^2(1-\lambda)d_1^{-2}.
\]
\end{proof}

\begin{corollary}\label{cor:pure-cube-inequality}
Assume \(Y=\{\mathrm{pt}\}\), so \(r=0\) and \(k=n\), and assume \(\lambda<\Lambda_n\).  If \(\Scal_f^{n,\lambda}(g)\ge\sigma\), then
\[
   \sigma
   \le
   4\pi^2(1-\lambda)\sum_{i=1}^{n}d_i^{-2}.
\]
\end{corollary}

\begin{proof}
Apply Proposition~\ref{prop:controlled-descent} with \(m=1\), then apply Lemma~\ref{lem:one-dimensional-case}, and let \(\varepsilon\to0\).
\end{proof}

\begin{corollary}\label{cor:terminal-Y-datum}
Assume \(r\ge1\) and \(\lambda<\Lambda_n\).  For every \(\varepsilon>0\), the descent produces an \(r\)-dimensional datum controlled over \(Y\), with generated sets \(X_r\) and \(D_r\), such that \(\eta_r:=F_r:X_r^\circ\to Y\) is proper over \(Y\setminus F_r(D_r)\) and has nonzero degree there, and
\[
   \Scal_{f_r}^{r,\lambda}(g_r)
   \ge
   \sigma
   -
   4\pi^2(1-\lambda)\sum_{i=1}^{k}d_i^{-2}
   -
   \varepsilon.
\]
If \(r\le6\), then \(D_r=\emptyset\).  Consequently \(X_r=X_r^\circ\) is a smooth complete \(r\)-dimensional manifold, and \(\eta_r:X_r\to Y\) is proper and has nonzero degree.
\end{corollary}

\begin{proof}
The scalar lower bound is Proposition~\ref{prop:controlled-descent} with \(m=r\).  The degree and properness are part of the terminal controlled datum.  Since \(\dim_{\mathcal H}D_r\le r-7\), the condition \(r\le6\) forces \(D_r=\emptyset\).
\end{proof}

\begin{remark}
\label{rem:controlled-target-version}
We record a flexibility in the descent.  If \(N^n\) is an oriented manifold with boundary and \(A\subset N\) is closed, write \(A\in\mathcal N_q(N)\) when \(\mathcal H^{n-q}(A)=0\).  

At the \(m\)-dimensional stage, write
\[
   B_m=B_{m-1}\times I,
   \qquad
   F_m=(F_{m-1},\tau_j),
   \qquad
   j=m-r .
\]
Let \(A_m\in\mathcal N_q(B_m)\), \(q\le6\), be closed, with \( F_m(D_m)\subset A_m \). Assume that \(F_m\) is proper over \(B_m\setminus A_m\) and has nonzero locally finite relative degree there.  For a good level \(a\in(-1,1)\), put
\[
   H_a=B_{m-1}\times\{a\},
   \qquad
   A_{m-1}^{\rm old}
   =
   \{y\in B_{m-1}:(y,a)\in A_m\}.
\]
By \cite[Theorem~7.7]{Mattila1995}, for almost every such \(a\) we have \( A_{m-1}^{\rm old}\in\mathcal N_q(B_{m-1})\).

The key ingredient is finite-width localization.  Let \(Z_a=X_m^\circ\cap \tau_j^{-1}(a)\) and, for \(W<\infty\), let
\[
   N_W=N_W^{g_m}(Z_a)
   =
   \{x\in X_m^\circ:d_{g_m}(x,Z_a)\le W\}.
\]
Let \(\Omega\subset N_W\) be a smooth closed finite-width band used for the \(\mu\)-bubble construction.  We require
\(\overline{\Omega}\cap D_m\subset F_m^{-1}(A_m\cap H_a)\).
This implies that
\[
   F_{m-1}\big|_\Omega:\Omega\to B_{m-1}
\]
is proper over \(B_{m-1}\setminus A_{m-1}^{\rm old}\). For singularities produced in the conformal blow-up descent, this localization is supplied by Lemma~\ref{lem:finite-width-localization}.

After the \(\mu\)-bubble step, let \(D^\Sigma\) be the new singularities and define
\[
   A_{m-1}
   =
   A_{m-1}^{\rm old}
   \cup
   F_{m-1}\!\left(
      D^\Sigma\cap N_W\cap
      F_{m-1}^{-1}(B_{m-1}\setminus A_{m-1}^{\rm old})
   \right).
\]
As before, the finite-width localization condition makes \(A_{m-1}\) closed. Moreover, \(\dim_{\mathcal H}D^\Sigma\le m-8\).  Since \(q\le6\), this gives
\[
   A_{m-1}\in\mathcal N_q(B_{m-1}).
\]
The intersection argument used before then still gives a nonzero degree. Hence the descent is unchanged if \(F_m(D_m)\) is replaced by such target exceptional sets \(A_m\).

In particular, suppose that \(X_n\) is the closure of \(X_n^\circ\) in a metric space \((X,d)\), and write \(D_n=X_n\setminus X_n^\circ\). Assume that \(X_n^\circ\) is smooth with metric \(g\), and that \(d\) is locally compatible with \(g\) on \(X_n^\circ\) in the sense of the definition of AM--PI spaces. Assume also that \(F\) is Lipschitz on \(X_n\), that \(F(D_n)\subset A_n\), and that the complete metric \(g_n\) satisfies, for every \(q\in D_n\), a local conformal blow-up estimate
    \[
       g_n\ge c_q d(\cdot,q)^{-2}g
    \]
    near \(q\). Then the same argument as in the proof of Lemma~\ref{lem:finite-width-localization} gives the finite-width localization property required above.  Consequently, in the case \(B_n=I^n\), if \(A_n\in\mathcal N_1(I^n)\), then the cube inequality holds for \(X_n^\circ\).
\end{remark}

\begin{remark}\label{rem:cohomological-degree-descent}
The descent in this section can also be run over a lower-dimensional cubical target.  In that case the ordinary degree is replaced by a cohomological degree. Let \(F:(M^m,\partial M)\to(B^{m-r},\partial B)\) be proper over \(U=B\setminus A\), and let \(u\in H^r(M;\mathbb Z)\).  We say that \(F\) has nonzero \(u\)-degree over \(U\) if
\[
   F_*\bigl([F^{-1}(U),\partial F^{-1}(U)]_{\rm lf}\cap u\bigr)
   =
   c[U,\partial U]_{\rm lf}
\]
for some \(c\ne0\).  When \(r=0\) and \(u=1\), this is the locally finite degree used above.

If \(F:X\to B\) is proper over \(U\), then all locally finite classes are restricted to \(F^{-1}(U)\).  This restriction is compatible with cap product:
\[
   (\xi\cap u)|_{F^{-1}(U)}
   =
   \xi|_{F^{-1}(U)}\cap u|_{F^{-1}(U)} .
\]
Thus equality of locally finite classes over \(U\) remains true after capping with \(u\), and properness over \(U\) allows the resulting classes to be pushed forward to \(U\).

Suppose now that the initial datum has cubical part \(\tau=(\tau_1,\ldots,\tau_k):X_n^\circ\to B_k=I^k\), where \(k=n-r\), and let \(u\in H^r(X_n^\circ;\mathbb Z)\).  Assume that \(\tau\) has nonzero \(u\)-degree over \(B_k\).

At an \(m\)-dimensional step, the remaining target is \(B_{m-r}=I^{m-r}\).  Write \(B_{m-r}=B_{m-r-1}\times I\) and \(F_m=(F_{m-1},\tau_j)\).  After \(A_{m-r-1}\) has been defined, put \(U=B_{m-r-1}\setminus A_{m-r-1}\).
As before, finite-width localization gives properness of \(P=F_{m-1}|_\Omega\) over \(U\).  The finite-perimeter region used in the \(\mu\)-bubble replacement shows that the good slice and the $\mu$-bubble represent the same locally finite relative class in \(P^{-1}(U)\).  Capping with the restriction of \(u\) preserves this equality, and the properness of \(P\) allows us to push the resulting classes forward to \(U\).  Hence the \(u\)-degree is unchanged.

Iterating the descent gives an \(r\)-dimensional terminal datum with nonzero \(u\)-degree.  If \(\iota_r:X_r^\circ\hookrightarrow X_n^\circ\) denotes the inclusion map, then \(\left\langle\iota_r^*u,[X_r^\circ,\partial X_r^\circ]_{\rm lf}\right\rangle\ne0 \). The same argument works with coefficients in an abelian group.
\end{remark}

\section{Enlargeable bases and two-systoles}
\label{sec:enlargeable-bases}

\subsection{Proofs of Theorems~A,~B and~C}

We use the notation and the admissible constant \(\Lambda_n\) from Section~\ref{sec:cube-inequality}.  Recall the definition of enlargeability  from the introduction.

\begin{proof}[Proof of Theorem~B]
It suffices to prove the result when
\(\lambda<\Lambda_n\). Indeed, the case when $\lambda\leq \Lambda_n$ can be dealt with by approximation based on the fact $\Scal_f^{n,\lambda'}(g)\ge \Scal_f^{n,\lambda}(g)$ for any $\lambda'<\lambda$.
If \(Y=\{\mathrm{pt}\}\), then Theorem~B is exactly the cube inequality established in Corollary~\ref{cor:pure-cube-inequality}.  Therefore, we assume
\[
   r=\dim Y>0 .
\]

Fix \(\varepsilon>0\).  By enlargeability of \(Y\), choose a Riemannian cover
\[
   \widehat Y\to Y
\]
and a smooth \(\varepsilon\)-Lipschitz map
\[
   \phi:\widehat Y\to S^r
\]
which is constant outside a compact set and has nonzero degree.  Pull back \(X\) by the \(Y\)-component of \(F\), and write
\[
   \widehat X=X\times_Y\widehat Y .
\]
The lifted map
\[
   \widehat F:\widehat X\to\widehat Y\times I^k
\]
is proper and has the same nonzero  relative degree.

Choose a regular value \(y_0\in S^r\), different from the constant value around infinity, such that
\[
   \deg(\phi;y_0)\ne0 .
\]
Let \(C\Subset S^r\) be a small coordinate cube around \(y_0\), away from the constant value around infinity.  Then \(\phi^{-1}(C)\) is compact and the restriction map \(\phi:\phi^{-1}(C)\to C\) has nonzero relative degree.  
In particular, the preimage
\[
   X_C
   :=
   \widehat F^{-1}\bigl(\phi^{-1}(C)\times I^k\bigr)
\]
is compact, and the map
\[
   X_C\to C\times I^k
\]
has nonzero relative degree as well.  After identifying \(C\) with \(I^r\), this gives a cubical map
\[
   X_C\to I^r\times I^k=I^n .
\]

The \(k\) cubical widths coming from the original \(I^k\)-coordinates are at least the corresponding \(d_i\).  The additional \(r\) widths coming from \(C\subset S^r\) tend to infinity as \(\varepsilon\downarrow0\).  More precisely, let \(L_Y\) denote a Lipschitz constant for \(F_Y\), and \(c_C>0\) denote the minimum distance between opposite faces of \(C\), then the new widths are bounded below by \(\frac{c_C}{\varepsilon L_Y}\). Applying Corollary~\ref{cor:pure-cube-inequality} to the cubical map above
gives
\[
   \sigma
   \le
   4\pi^2(1-\lambda)
   \left(
      \sum_{i=1}^{k}d_i^{-2}
      +
      O(\varepsilon^2)
   \right).
\]
Letting \(\varepsilon\downarrow0\) gives
\[
   \sigma
   \le
   4\pi^2(1-\lambda)\sum_{i=1}^{k}d_i^{-2}.
\]
This proves Theorem~B.
\end{proof}

\begin{proof}[Proof of Theorem~A]
Take \(Y=\mathbb T^n\) and \(k=0\) in Theorem~B.
\end{proof}

\begin{proof}[Proof of Theorem~C]
As before, it suffices to deal with \(\lambda<\Lambda_n\).  Applying Corollary~\ref{cor:terminal-Y-datum} with \(Y=S^2\), we obtain a smooth closed oriented surface \((\Sigma,g_\Sigma)\), with a map \(\rho_\Sigma:\Sigma\to S^2\) of nonzero degree, and for every \(\eta>0\),
\[
   \Scal_{f_\Sigma}^{2,\lambda}(g_\Sigma)\ge\sigma_*-\eta .
\]
From $\lambda<\Lambda_n$ we have \(a_2(\lambda)>0\), and then it follows, on any component \(\Sigma_0\subset\Sigma\),
\[
\begin{split}
   (\sigma_*-\eta)\operatorname{Area}_{g_\Sigma}(\Sigma_0)
   &\le
   \int_{\Sigma_0}\Scal_{f_\Sigma}^{2,\lambda}(g_\Sigma)  \\
   &=
   \int_{\Sigma_0}
   \bigl(R_{g_\Sigma}
      +2\Delta f_\Sigma
      -a_2(\lambda)|df_\Sigma|^2\bigr)  \\
   &\le
   \int_{\Sigma_0}R_{g_\Sigma}
   =
   4\pi\chi(\Sigma_0).
\end{split}
\]
Since \(\rho_\Sigma\) has nonzero degree, some component \(\Sigma_0\) carries nonzero degree over \(S^2\).  For this component the left-hand side is positive, so we know \(\chi(\Sigma_0)>0\) and \(\Sigma_0\simeq S^2\).  Therefore
\[
   \operatorname{Area}_{g_\Sigma}(\Sigma_0)
   \le
   \frac{8\pi}{\sigma_*-\eta}.
\]
This component gives an integral \(2\)-cycle in \(X\) pairing nontrivially with \(u_\rho\).  Since the conformal factors in the descent are at least \(1\), its area with respect to the original metric is no larger than $\operatorname{Area}_{g_\Sigma}(\Sigma_0)$.  Letting \(\eta\downarrow0\) proves the estimate.
\end{proof}

\subsection{The \(S^2\)-factor over an enlargeable base}

\begin{proposition}\label{prop:enlargeable-s2-systole}
Let \(Y^{n-2}\) be a closed oriented enlargeable manifold, and let \(X^n\) be a closed oriented smooth manifold.  Suppose
\[
   F=(F_Y,\rho):X\to Y\times S^2
\]
has nonzero degree, and set \(u_\rho=\rho^*\omega_{S^2}\). If
\[
   \Scal_f^{n,\lambda}(g)\ge\sigma>0,
   \qquad
   \lambda\le\Lambda_n,
\]
then
\begin{equation}\label{eq:enlargeable-s2-systole}
   \operatorname{sys}_2(X,u_\rho;g)\le\frac{8\pi}{\sigma}.
\end{equation}
Moreover, if equality holds in \eqref{eq:enlargeable-s2-systole}, then
\[
   f\equiv\mathrm{constant},
   \qquad
   R_g\equiv\sigma,
   \qquad
   \operatorname{Ric}_g\ge0,
\]
and the universal cover splits isometrically as 
\[
   \widetilde X\cong S^2_\sigma\times\mathbb R^{n-2},
\]
where \(S^2_\sigma\) is the round \(2\)-sphere with scalar curvature \(\sigma\).
\end{proposition}

\begin{proof}
We first prove the estimate \eqref{eq:enlargeable-s2-systole}.  It suffices to deal with \(\lambda<\Lambda_n\) as before.

Fix \(\varepsilon>0\).  Choose a cover \(\widehat Y\to Y\) and an \(\varepsilon\)-Lipschitz map \(\phi:\widehat Y\to S^{n-2}\), constant outside a compact set and of nonzero degree.  Pull back \(X\) to \(\widehat X=X\times_Y\widehat Y\), and write \(\widehat F=(\widehat F_Y,\widehat\rho)\).  Let \(C\Subset S^{n-2}\) be a small coordinate cube around a regular value of \(\phi\), disjoint from the constant value around infinity, such that
\[
   \phi:\phi^{-1}(C)\to C
\]
has nonzero relative degree.  Then
\[
   X_C:=\widehat F_Y^{-1}(\phi^{-1}(C))
\]
is compact, and
\[
   (\widehat\rho,\phi\circ\widehat F_Y):X_C\to S^2\times C
\]
has nonzero relative degree.  Identifying \(C\) with \(I^{n-2}\), we apply Theorem~C. Let \(D_j\) be the cubical widths coming from \(C\), then \(D_j\ge c_C/(\varepsilon L_Y)\), where \(L_Y\) denotes a Lipschitz constant for \(F_Y\), and \(c_C>0\) denotes the minimum distance between opposite faces of \(C\). Therefore, we have \(\sum_jD_j^{-2}=O(\varepsilon^2)\) and
\[
   \operatorname{sys}_2(X_C,\widehat u_\rho;g)
   \le
   \frac{8\pi}
   {\sigma-4\pi^2(1-\lambda)O(\varepsilon^2)} .
\]
Pushing forward cycles to \(X\) does not increase mass and preserves the pairing with \(u_\rho\).  Letting \(\varepsilon\downarrow0\) gives
\[
   \operatorname{sys}_2(X,u_\rho;g)\le\frac{8\pi}{\sigma}.
\]

Now we deal with the rigidity and assume the equality holds.  Put
\[
   a=a_n(\lambda),
   \qquad
   V=\Scal_f^{n,\lambda}(g).
\]
Consider
\[
   P_{g,f}
   =
   -\frac4a
   \bigl(\Delta_g-a\langle\nabla f,\nabla\cdot\rangle_g\bigr)
   +V,
\]
which is self-adjoint with respect to the weighted measure \(e^{-af}d\mu_g\). Denote its first eigenvalue by \(\theta(g,f)\).  

If \(V>\sigma\) somewhere, then \(\theta(g,f)>\sigma\). For a positive first eigenfunction \(u\), set \(\psi=-2a^{-1}\log u\).  Then
\[
   \Scal_{f+\psi}^{n,\lambda}(g)
   =
   u^{-1}P_{g,f}u
   =
   \theta(g,f)>\sigma .
\]
Applying the estimate \eqref{eq:enlargeable-s2-systole} to \((g,f+\psi,\lambda)\) gives
\[
   \operatorname{sys}_2(X,u_\rho;g)
   \le
   \frac{8\pi}{\theta(g,f)}
   <
   \frac{8\pi}{\sigma}.
\]
This leads to a contradiction, and so we have \(V\equiv\sigma\).

If \(df\not\equiv0\), choose \(\lambda_-<\lambda\).  Since \(a_n(\lambda)\) increases as $\lambda$ increases,
\[
   \Scal_f^{n,\lambda_-}(g)
   =
   \sigma+
   \bigl(a_n(\lambda)-a_n(\lambda_-)\bigr)|df|_g^2
\]
is no less than \(\sigma\), and greater than \(\sigma\) somewhere. Repeating the preceding argument with \(\lambda_-\) gives a weight \(\psi_-\) and a
constant \(\theta_->\sigma\) with \(\Scal_{f+\psi_-}^{n,\lambda_-}(g)=\theta_-\), again contradicting the estimate.  Thus we have \(df\equiv0\), and consequently
\[
   f\equiv\mathrm{constant},
   \qquad
   R_g\equiv\sigma .
\]

We next show \(\operatorname{Ric}_g\ge0\).  For a metric \(\bar g\), set
\[
   P_{\bar g}=-\frac4a\Delta_{\bar g}+R_{\bar g},
\]
and denote its first eigenvalue by \(\theta(\bar g)\).  Since \(R_g\equiv\sigma\), constant functions are first eigenfunctions with respect to the first eigenvalue
\(\theta(g)=\sigma\) for the operator $P_{\bar g}$.  If \(\operatorname{Ric}_g\) has a negative direction somewhere, then we can choose a $2$-tensor \(h\ge0\), supported in a small ball, with
\[
   \int_X\langle\operatorname{Ric}_g,h\rangle_g\,d\mu_g<0.
\]
For \(g_t=g+th\), one has \(g_t\ge g\) and
\[
   \operatorname{sys}_2(X,u_\rho;g_t)\ge\frac{8\pi}{\sigma}.
\]
On the other hand, the first variation formula gives
\[
   \left.\frac d{dt}\right|_{t=0}\theta(g_t)
   =
   -\frac1{\operatorname{Vol}_g(X)}
   \int_X\langle\operatorname{Ric}_g,h\rangle_g\,d\mu_g>0.
\]
Thus \(\theta(g_t)>\sigma\) for small \(t>0\).  Let $u_t$ be the first eigenfunction of the operator $P_{g_t}$ 
and \(\psi_t=-2a^{-1}\log u_t\), then we have
\[
   \Scal_{\psi_t}^{n,\lambda}(g_t)=\theta(g_t)>\sigma .
\]
Applying the estimate \eqref{eq:enlargeable-s2-systole} to \((g_t,\psi_t,\lambda)\) gives \(\operatorname{sys}_2(X,u_\rho;g_t) <\frac{8\pi}{\sigma}\), leading to a contradiction.  Therefore, we have \(\operatorname{Ric}_g\ge0\).

Finally, by Cheeger--Gromoll \cite{CheegerGromoll1971}, we know that the universal cover $\widetilde X$ splits isometrically as
\[
   \widetilde X\cong\mathbb R^\ell\times Z
\]
with \(Z\) compact and simply-connected. Since $Y^{n-2}$ is enlargeable, the argument in \cite[Proposition~IV.5.8]{LawsonMichelsohnSpinGeometry} (see also \cite[Proposition 3.3]{Gromov-Zhu}) gives \(\ell\ge n-2\).  Since \(R_g\equiv\sigma>0\), the compact factor has dimension at least \(2\), so 
we have \(\ell=n-2\) and \(\dim Z=2\).  The product structure of $\widetilde X$ yields that \(Z\) is the round sphere \(S^2_\sigma\) with constant scalar curvature $\sigma$. 
Thus
\[
   \widetilde X\cong S^2_\sigma\times\mathbb R^{n-2}.
\]
The proof is completed.
\end{proof}

\begin{remark}
The \(S^2\)-factor in Proposition~\ref{prop:enlargeable-s2-systole} can be replaced by a cohomology class.  Let \(u\in H^2(X;\mathbb Z)\), and assume that \(F_Y:X\to Y\) has nonzero \(u\)-degree over \(Y\), in the sense of Remark~\ref{rem:cohomological-degree-descent}.  

The only additional point in the proof is that this \(u\)-degree remains nonzero by passing to the enlargeability cover.  Indeed, if \(\widehat Y\to Y\) is such a cover, \(q:\widehat X\to X\) is the induced cover, and \(\widehat F_Y:\widehat X\to\widehat Y\) is the lifted map, then \(\widehat F_Y\) has nonzero \(q^*u\)-degree over \(\widehat Y\).  If \(C\subset S^{n-2}\) is the coordinate cube used in the proof and \(\phi:\phi^{-1}(C)\to C\) has relative degree \(d\ne0\), then
\[
   \deg^{q^*u}_{C}(\phi\circ\widehat F_Y)
   =
   d\,\deg^u_Y(F_Y)
   \ne0 .
\]
Thus the cubical descent of Remark~\ref{rem:cohomological-degree-descent} applies.  It gives a two-dimensional cycle pairing nontrivially with \(q^*u\). Pushing this cycle to \(X\) does not increase mass and preserves the pairing with \(u\).  Hence \(\operatorname{sys}_2(X,u;g)\le\frac{8\pi}{\sigma}\) under the same scalar-curvature assumptions. Moreover, the rigidity statement follows from the same proof.
\end{remark}

\section{Obstructions and rigidity for AM--PI spaces}
\label{sec:ampi-obstruction-rigidity}

Throughout this section, \( X=\mathcal R\sqcup\mathcal S\) is a compact connected oriented \(n\)-dimensional AM--PI space in the sense of Section~\ref{sec:conformal-blowup-exceptional}.  On \(\mathcal R\), we write the chosen admissible measure as \(e^{-f}d\mu_g \). We assume throughout this section that \(f\) is bounded on \(\mathcal R\) so that every measure \(e^{-af}d\mu_g\), \(a\in\mathbb R\), is again admissible.

\subsection{Positive weighted scalar curvature obstruction}

\begin{theorem}\label{thm:ampi-weighted-psc-obstruction}
Let \(X=\mathcal R\sqcup\mathcal S\) be compact, oriented, and enlargeable, and assume
\[
   c_A^X(\mathcal S)>3-\frac2n .
\]
Let \(\lambda\leq 1/n\).  Then \(\Scal_f^{n,\lambda}(g)\) cannot be nonnegative on \(\mathcal R\) and positive on a nonempty open subset of \(\mathcal R\).
\end{theorem}

\begin{proof}
Assume otherwise.  Choose
\[
   3-\frac2n<c<c_A^X(\mathcal S).
\]
By Remark~\ref{rem:codimension-thresholds-section3}, choose
\(\bar\lambda\le\lambda\) such that
\[
   \bar\lambda<\Lambda_n,
   \qquad
   \tau_n(\bar\lambda)<c-2 .
\]
Since \(a_n\) is increasing on \((-\infty,1/n)\), the function
\[
   V_0:=\Scal_f^{n,\bar\lambda}(g)
   =
   \Scal_f^{n,\lambda}(g)
   +
   \bigl(a_n(\lambda)-a_n(\bar\lambda)\bigr)|df|_g^2
\]
is nonnegative and positive on a nonempty open subset of \(\mathcal R\). Choose a smooth ball \(B\Subset\mathcal R\), a constant \(b>0\), and a function \(\chi\in C_c^\infty(B)\), \(0\le\chi\le1\), \(\chi\not\equiv0\), such that \(V_0\ge b\) on \(B\).  Put \(W=b\chi\).  Then
\(0\le W\le V_0\), \(W\not\equiv0\), and \(W\) is bounded with compact support in \(\mathcal R\).

Put \(a=a_n(\bar\lambda)\).  Since \(f\) is bounded, the measure \(e^{-af}d\mu_g\) is admissible.  Define
\[
   \eta_0
   =
   \inf_{\varphi\not\equiv0}
   \frac{
      \frac4a\int_{\mathcal R}|\nabla\varphi|_g^2e^{-af}\,d\mu_g
      +
      \int_{\mathcal R}W\varphi^2e^{-af}\,d\mu_g
   }{
      \int_{\mathcal R}\varphi^2e^{-af}\,d\mu_g
   } .
\]
The compact AM--PI hypothese guarantees a minimizer \(u\ge0\).  Since \(X\) is connected and \(W\ge0\) is not identically zero, \(\eta_0>0\).  The  Harnack inequalities give \(0<c_0\le u\le C_0<\infty\).  On \(\mathcal R\), elliptic regularity gives \(u\in C^\infty(\mathcal R)\) and
\[
   \left[
      -\frac4a
      \bigl(\Delta_g-a\langle\nabla f,\nabla\cdot\rangle_g\bigr)
      +W
   \right]u
   =
   \eta_0u .
\]
Set \(\psi=-2a^{-1}\log u\).  Then \(\psi\) is bounded.  Since \(W\le V_0\), a direct calculation gives
\[
\begin{aligned}
   \Scal_{f+\psi}^{n,\bar\lambda}(g)
   &=
   u^{-1}
   \left[
      -\frac4a
      \bigl(\Delta_g-a\langle\nabla f,\nabla\cdot\rangle_g\bigr)
      +V_0
   \right]u  \\
   &=
   \eta_0+(V_0-W)
   \ge \eta_0>0 .
\end{aligned}
\]
Moreover, since \(f\) and \(\psi\) are bounded, the weight \(e^{-\gamma_n(\bar\lambda)(f+\psi)}\) is comparable to \(e^{-f}\), and is therefore admissible by Lemma~\ref{lem:bounded-weight-preserves-ampi-section3}.

We now apply Proposition~\ref{prop:packing-conformal-blowup-scalar} to \((g,f+\psi)\) and obtain a complete smooth metric \(g_0\) on \(\mathcal R\) and a smooth density \(f_0\) such that \(\Scal_{f_0}^{n,\bar\lambda}(g_0)\ge\eta_1>0 \).

Let \(\widehat X\to X\) be a cover, and let \(\Phi:\widehat X\to S^n\) be \(\varepsilon\)-Lipschitz, constant outside a compact set, and of nonzero degree.  Choose a regular value \(y_0\in S^n\), different from the constant value at infinity, such that \(\deg(\Phi;y_0)\ne0\). Let \(C\Subset S^n\) be a small coordinate cube around \(y_0\), disjoint from the constant value at infinity.  Then \(\Phi^{-1}(C)\) is compact and \(\Phi:\Phi^{-1}(C)\to C\) has nonzero relative degree.

Let \(\widehat{\mathcal R}\) and \(\widehat{\mathcal S}\) denote the lifted regular and singular parts, and set
\[
   A=\Phi\bigl(\widehat{\mathcal S}\cap\Phi^{-1}(C)\bigr)\subset C .
\]
Since \(c_A^X(\mathcal S)>3-2/n\), one has \(\dim_A\mathcal S<n-1\), and hence \(\mathcal H^{n-1}(A)=0\).
The blown-up regular part maps properly to \(C\setminus A\) and has nonzero degree. Identifying \(C\) with \(I^n\), Remark~\ref{rem:controlled-target-version} gives
\[
   \eta_1
   \le
   4\pi^2(1-\bar\lambda)\sum_{i=1}^nD_i^{-2},
\]
where \(D_i\) are the cubical widths measured in \(g_0\).  By the construction of the pointwise conformal blow-up, \(g_0\ge g\).  Since \(\Phi\) is \(\varepsilon\)-Lipschitz with respect to \(g\), there is a constant \(c_C>0\), depending only on \(C\), such that \(D_i\ge c_C\varepsilon^{-1}\) for all \(i\).  Letting \(\varepsilon\downarrow0\) contradicts \(\eta_1>0\).
\end{proof}

\begin{corollary}\label{cor:ampi-weight-constant}
Under the hypotheses of Theorem~\ref{thm:ampi-weighted-psc-obstruction}, if
\[
   \Scal_f^{n,\lambda}(g)\ge0
\]
on \(\mathcal R\) for some \(\lambda\leq 1/n\), then \(df\equiv0\) and \(R_g\equiv0\) on \(\mathcal R\). 
\end{corollary}

\begin{proof}
Theorem~\ref{thm:ampi-weighted-psc-obstruction} gives \(\Scal_f^{n,\lambda}(g)\equiv0\). If \(df\not\equiv0\), choose \(\lambda'<\lambda\).  Since \(a_n\) is strictly increasing,
\[
   \Scal_f^{n,\lambda'}(g)
   =
   \Scal_f^{n,\lambda}(g)
   +
   \bigl(a_n(\lambda)-a_n(\lambda')\bigr)|df|_g^2
\]
is nonnegative and positive somewhere, contradicting Theorem~\ref{thm:ampi-weighted-psc-obstruction}.  Thus \(df\equiv0\), and then \(R_g=\Scal_f^{n,\lambda}(g)\equiv0\).
\end{proof}

\begin{proposition}\label{prop:ampi-ricci-flat-reduction}
Let \(X=\mathcal R\sqcup\mathcal S\) satisfy the hypotheses of Theorem~\ref{thm:ampi-weighted-psc-obstruction}.  If
\[
   \Scal_f^{n,\lambda}(g)\ge0
\]
on \(\mathcal R\) for some \(\lambda\leq 1/n\), then \(df\equiv0\) and \(\operatorname{Ric}_g\equiv0\) on \(\mathcal R\).
\end{proposition}

\begin{proof}
By Corollary~\ref{cor:ampi-weight-constant}, \(df\equiv0\) and \(R_g\equiv0\) on \(\mathcal R\).  It remains to rule out
\(\operatorname{Ric}_g\not\equiv0\).

Assume \(\operatorname{Ric}_g\not\equiv0\).  Since \(R_g=0\), the Ricci tensor has a negative direction somewhere.  Choose
\(h\in C_c^\infty(\operatorname{Sym}^2T^*\mathcal R)\), supported in a smooth ball \(B\Subset\mathcal R\), with \(h\ge0\) and
\[
   \int_{\mathcal R}\langle\operatorname{Ric}_g,h\rangle_g\,d\mu_g<0 .
\]
For small \(t\ge0\), set \(g_t=g+th\).  Let \(a=a_n(\lambda)\), and let \(\theta(t)\) be the first eigenvalue of
\(P_{g_t}=-\frac4a\Delta_{g_t}+R_{g_t}\).  As in the deformation argument used in Proposition~\ref{prop:enlargeable-s2-systole}, constants realize \(\theta(0)=0\), and
\[
   \theta'(0)
   =
   -\frac1{\mu_g(X)}
   \int_{\mathcal R}\langle\operatorname{Ric}_g,h\rangle_g\,d\mu_g
   >0 .
\]
Hence \(\theta(t)>0\) for all small \(t>0\).

Let \(u_t>0\) be a first eigenfunction.  By the same compactness, regularity and Harnack argument as above, \(u_t\) is smooth on \(\mathcal R\) and \(0<c_t\le u_t\le C_t\).  With \(\psi_t=-2a^{-1}\log u_t\), we get
\[
   \Scal_{\psi_t}^{n,\lambda}(g_t)
   =
   u_t^{-1}P_{g_t}u_t
   =
   \theta(t)>0
   \quad\text{on }\mathcal R .
\]
The deformation is supported in \(B\Subset\mathcal R\), so \(g_t=g\) near \(\mathcal S\); in particular the AM--PI assumptions, the Assouad-codimension condition, and admissibility of the weight are unchanged.  Applying Theorem~\ref{thm:ampi-weighted-psc-obstruction} to \((X,g_t,\psi_t,\lambda)\) gives a contradiction.  Therefore \(\operatorname{Ric}_g\equiv0\) on \(\mathcal R\).
\end{proof}

\subsection{The QL property}

The argument in this subsection is inspired by Dai--Wang--Wang--Wei \cite{DaiWangWangWei2025} and Honda--Sun \cite{HondaSun2026}.

We use the following equivalent \(L^\infty\)-form of the QL condition of Honda--Sun~\cite[Definition~2.29]{HondaSun2026}; the equivalence follows from the standard local boundedness estimate for harmonic functions on PI spaces.
\begin{definition}
    We say that \((X,d,\mu_g)\) satisfies \emph{QL} if, for every compact \(K\Subset X\), there exist constants \(C_K,r_K>0\) such that, whenever \(x\in K\), \(0<r<r_K\), and \(u\) is bounded and weakly harmonic on \(B_{2r}(x)\), one has
\[
   \operatorname*{ess\,sup}_{B_r(x)\cap\mathcal R}|\nabla u|_g
   \le
   \frac{C_K}{r}\|u\|_{L^\infty(B_{2r}(x))}.
\]

\end{definition}

\begin{lemma}\label{lem:kato-removable-subsolution-ql-ampi}
Assume \(\operatorname{Ric}_g\ge0\) on \(\mathcal R\).  Let \(U\subset X\) be open, let \(u\) be harmonic on \(U\cap\mathcal R\), and set \(h=|\nabla u|_g\) and \(q_n=\frac{n-2}{n-1}\). For every \(q>q_n\), the function \(v=h^q\) satisfies, weakly on \(U\cap\mathcal R\),
\begin{equation}\label{eq:kato-subsolution-weak-ampi}
   \Delta_g v
   \ge
   \kappa_q v^{-1}|\nabla v|_g^2,
   \qquad
   \kappa_q=\frac{q-q_n}{q}>0 .
\end{equation}

Assume moreover that, for some \(\beta>0\),
\[
   v\le A\,\operatorname{dist}(\cdot,\mathcal S)^{-\beta}
   \qquad\text{on }U\cap\mathcal R,
\]
and \(c_A^X(\mathcal S)>2+\beta\). Then \(v\) is weakly subharmonic on \(U\), that is, for every \(\varphi\in\operatorname{Lip}_c(U)\), \(\varphi\ge0\),
\[
   \int_U\langle\nabla v,\nabla\varphi\rangle\,d\mu_g\le0 .
\]
\end{lemma}

\begin{proof}
Put \[s=|\nabla u|_g^2.\]  For \(\varepsilon>0\), set \[v_\varepsilon=(s+\varepsilon)^{q/2}.\] The Bochner formula gives \(\frac12\Delta_gs\ge|\nabla^2u|^2\), and the improved Kato inequality gives
\[
   |\nabla s|^2\le \frac{4(n-1)}{n}s|\nabla^2u|^2 .
\]
Hence
\[
\begin{aligned}
   \Delta_gv_\varepsilon
   &=
   \frac q2(s+\varepsilon)^{q/2-1}\Delta_gs
   +
   \frac{q(q-2)}4(s+\varepsilon)^{q/2-2}|\nabla s|^2  \\
   &\ge
   q(s+\varepsilon)^{q/2-2}
   \left[
      (s+\varepsilon)|\nabla^2u|^2
      +\frac{q-2}{4}|\nabla s|^2
   \right]  \\
   &\ge
   \frac q4
   \left(q-\frac{n-2}{n-1}\right)
   (s+\varepsilon)^{q/2-2}|\nabla s|^2 .
\end{aligned}
\]
Since
\[
   v_\varepsilon^{-1}|\nabla v_\varepsilon|^2
   =
   \frac{q^2}{4}(s+\varepsilon)^{q/2-2}|\nabla s|^2,
\]
we get \[\Delta_gv_\varepsilon\ge\kappa_q v_\varepsilon^{-1}|\nabla v_\varepsilon|^2.\]
From integration by parts, the gradients $\nabla v_\epsilon$ are uniformly bounded in $L^2_{loc}(\mathcal R)$ and so we have $\nabla v_\epsilon\to \nabla v$ weakly as $\varepsilon\to 0$. Combined with the monotone convergence theorem applied to the right-hand side, we can obtain \eqref{eq:kato-subsolution-weak-ampi} for \(v=s^{q/2}=|\nabla u|^q\) by letting $\varepsilon\to 0$.

It remains to extend the weak subharmonicity across \(\mathcal S\).  Put \(d(x)=\operatorname{dist}(x,\mathcal S)\) and
\(N_\rho=N_\rho(\mathcal S)\).  Let \(K\Subset U\) and choose \(c\) with \(2+\beta<c<c_A^X(\mathcal S)\).  By the local Ahlfors regularity, and the bound \(v\le A d^{-\beta}\), we have, for small \(\rho\),
\[
   \mu_g(N_\rho\cap K)\le C\rho^c,\qquad
   \int_{N_\rho\cap K}v\,d\mu_g\le C\rho^{c-\beta}.
\]
Choose \(K'\Subset U\) with \(K\Subset K'\), and take \(\psi\in\operatorname{Lip}_c(U)\) with \(\psi\equiv1\) near \(K\) and \(\operatorname{spt}\psi\subset K'\).  Let \(\eta_\rho=1\) on \(N_{2\rho}\), \(\operatorname{spt}\eta_\rho\subset N_{4\rho}\), and \(|\nabla\eta_\rho|_g\le C\rho^{-1}\).  For \(0<\delta<\rho\), let \(\theta_\delta=0\) on \(N_\delta\), \(\theta_\delta=1\) outside \(N_{2\delta}\), and \(|\nabla\theta_\delta|_g\le C\delta^{-1}\).  Testing \eqref{eq:kato-subsolution-weak-ampi} with \(\zeta_{\rho,\delta}:=\psi\eta_\rho\theta_\delta\), which equals \(1\) on
\((N_{2\rho}\setminus N_{2\delta})\cap K\), gives
\[
\begin{aligned}
   \int_{(N_{2\rho}\setminus N_{2\delta})\cap K}
   v^{-1}|\nabla v|_g^2\,d\mu_g
   &\le
   C(\rho^{-2}+1)\int_{N_{4\rho}\cap K'}v\,d\mu_g
   +
   C\delta^{-2}\int_{N_{2\delta}\cap K'}v\,d\mu_g  \\
   &\le
   C\rho^{c-\beta-2}+C\delta^{c-\beta-2}.
\end{aligned}
\]
Since \(c-\beta-2>0\), letting \(\delta\downarrow0\) yields
\[
   \int_{N_{2\rho}\cap K}
   v^{-1}|\nabla v|_g^2\,d\mu_g
   \le C\rho^{c-\beta-2}.
\]

Now fix \(\varphi\in\operatorname{Lip}_c(U)\), \(\varphi\ge0\), and take \(K\Subset U\) with \(\operatorname{spt}\varphi\subset K\).  Let \(\chi_\rho=0\) on \(N_\rho\), \(\chi_\rho=1\) outside \(N_{2\rho}\), and \(|\nabla\chi_\rho|\le C\rho^{-1}\).  Since \(\chi_\rho\varphi\in\operatorname{Lip}_c(U\cap\mathcal R)\), the subharmonicity of $v$ on \(U\cap\mathcal R\) gives
\[
   \int_U\langle\nabla v,\nabla(\chi_\rho\varphi)\rangle\,d\mu_g\le0 .
\]
Hence
\[
\begin{aligned}
   \int_U\langle\nabla v,\nabla\varphi\rangle\,d\mu_g
   &\le
   \int_U(1-\chi_\rho)\langle\nabla v,\nabla\varphi\rangle\,d\mu_g
   -
   \int_U\varphi\langle\nabla v,\nabla\chi_\rho\rangle\,d\mu_g .
\end{aligned}
\]
By H\"older's inequality, the preceding estimates imply
\[
\left|
   \int_U(1-\chi_\rho)\langle\nabla v,\nabla\varphi\rangle\,d\mu_g
\right|
\le C\rho^{c-\beta-1}\to0,
\]
and
\[
\left|
   \int_U\varphi\langle\nabla v,\nabla\chi_\rho\rangle\,d\mu_g
\right|
\le C\rho^{c-\beta-2}\to0.
\]
Since \(c-\beta>2\), letting \(\rho\downarrow0\) yields \(\int_U\langle\nabla v,\nabla\varphi\rangle\,d\mu_g\le0\).
\end{proof}

\begin{proposition}\label{prop:QL-assouad-codim3-ampi}
Assume \(\operatorname{Ric}_g\ge0\) on \(\mathcal R\),
and
\[
   c_A^X(\mathcal S)>3-\frac1{n-1}.
\]
Then \((X,d,\mu_g)\) satisfies QL.
\end{proposition}

\begin{proof}
Let \(u\) be bounded and weakly harmonic on \(B_{2R}(p)\), and we are going to establish
\[\operatorname*{ess\,sup}_{B_R(p)}|\nabla u|_g\le\frac{C}{R}\|u\|_{L^\infty(B_{2R}(p))}.\]
Put
\[
   M=\|u\|_{L^\infty(B_{2R}(p))},
   \qquad
   \rho(x)=\operatorname{dist}(x,\mathcal S).
\]
The function \(u+2M\) is positive and harmonic on $B_{2R}(p)\mathcal R$.  If \(x\in B_{3R/2}(p)\cap\mathcal R\) and \(\rho(x)<R/8\), then \(B_{\rho(x)/4}(x)\subset B_{2R}(p)\cap\mathcal R\).
The Cheng--Yau estimate on this regular ball gives
\[
   |\nabla u|(x)\le C\rho(x)^{-1}M.
\]
Away from \(N_{R/8}(\mathcal S)\), the ordinary interior estimate gives the bound \(CR^{-1}M\).  Hence, after increasing \(C\), we have
\[
   |\nabla u|(x)\le C\rho(x)^{-1}M
   \qquad
   \text{on }B_{3R/2}(p)\cap\mathcal R.
\]

Choose
\[
   \frac{n-2}{n-1}<q<c_A^X(\mathcal S)-2.
\]
This is possible by the codimension assumption.  Set \(v=|\nabla u|^q\). The preceding gradient estimate gives
\(v\le C\rho^{-q}M^q\). Since \(c_A^X(\mathcal S)>2+q\), Lemma~\ref{lem:kato-removable-subsolution-ql-ampi} applies with \(\beta=q\), and so \(v\) is weakly subharmonic on \(B_{3R/2}(p)\). The Moser iteration then gives
\[
   \operatorname*{ess\,sup}_{B_R(p)}v
   \le
   C R^{-n}\int_{B_{3R/2}(p)}v\,d\mu_g .
\]
Since \(q<2\), Caccioppoli estimate for \(u\), H\"older inequality, and local Ahlfors regularity give
\[
   \int_{B_{3R/2}(p)}v\,d\mu_g
   \le
   C R^{n-q}M^q .
\]
Therefore, we arrive at \(\operatorname*{ess\,sup}_{B_R(p)}|\nabla u|_g\le\frac{C}{R}M\), which is QL.
\end{proof}

\subsection{Length-space replacement and rigidity}

\begin{lemma}
\label{lem:punctured-local-quasiconvexity}
Let \(X=\mathcal R\sqcup\mathcal S\) be an \(n\)-dimensional AM--PI space, and assume \(\mathsf P_{n-2}^{X}(\mathcal S)\). Then, for every compact \(K\Subset X\), there exist constants \(C_K,r_K>0\) such that, whenever \(x,y\in K\cap\mathcal R\) and \(R=d(x,y)<r_K\), there is a curve
\[
   \gamma \subset\mathcal R
\]
joining \(x\) to \(y\), with \(L_g(\gamma)\le C_KR\).
\end{lemma}

\begin{proof}
We use the argument in the proof of annular quasiconvexity, cf.~\cite[Theorem~3.3]{Korte2007}, \cite[Section~9.4]{HKST2015}.

Take a larger compact set $K_*$ such that $K\Subset K_*^\circ$. Since $X$ is AM-PI and \(\mathsf P_{n-2}^{X}(\mathcal S)\), there are constants $\widetilde C\geq 1$, $\widetilde \lambda\geq 1$, and $\widetilde r>0$ such that the Ahlfors estimate \eqref{eq:local-ahlfors-section3}, the Poincar\'e inequality \eqref{eq:local-pi-section3}, and also the packing estimate \eqref{eq:local-packing-condition-section3} hold for $B_d(z,r)$ whenever $z\in \widetilde K$ and $r\in (0,\widetilde r)$. In particular, up to decreasing $\widetilde r$, we have
\begin{equation}\label{Eq: neighborhood estimate}
     \mu_g\bigl(N_s(\mathcal S)\cap B_d(z,r)\bigr)
   \le
   C_Kr^{n-2}s^2
\end{equation}
whenever \(z\in K_*\), \(0<s<r<\widetilde r\). Since AM-PI spaces are local PI spaces, we have local quasiconvexity for small balls. Namely, by increasing $\widetilde C$ and decreasing $\widetilde r$ properly, for any points $x,y\in B_d(z,r)$ we can find a curve $\gamma_{x,y}$ connecting $x$ and $y$ with length no greater than $\widetilde Cd(x,y)$, whenever $z\in K_*$ and $r\in (0,\widetilde r)$.

 Put $\delta(z)=d(z,\mathcal S)$ and denote $R:=d(x,y)$.

\vspace{3mm}
\noindent{\bf Claim:} the lemma holds for any $x,y\in K_*\cap \mathcal R$ with $R<\widetilde r$ under the extra assumptions $\delta(x)\geq \beta R$ and $\delta(y)\geq  \beta R$ for some $\beta>0$, where the constant in the length control depends on $\beta$.
\vspace{3mm}

Set $\kappa=\beta (2\widetilde C)^{-1}$. The proof of the claim will be divided into two steps.

\vspace{3mm}
{\it Step 1.} There is a constant $L>1$
 such that for any \(x,y\in K\cap\mathcal R\) with $R:=d(x,y)< \widetilde r$, there is a curve $$\gamma:[0,1]\to\mathcal R$$
connecting $B_d(x,\kappa R)$ to $ B_d(y,\kappa R)$   with $L_g(\gamma)\leq L R$.

\vspace{3mm}
For simplicity, we denote
\[
   E_0=B_d(x, \kappa R),
   \qquad
   F_0=B_d(y,\kappa R).
\]
With $A>0$ to be determined later, we put
\[
   E=E_0\setminus N_{R/A}(\mathcal S),
   \qquad
   F=F_0\setminus N_{R/A}(\mathcal S).
\]
Clearly, we have the measure estimates
$$\mu_g(E_0)\geq cR^n,\qquad \mu_g(F_0)\geq cR^n,$$
and
\[
   \mu_g(E_0\cap N_{R/A}(\mathcal S))
   +
   \mu_g(F_0\cap N_{R/A}(\mathcal S))
   \le CA^{-2}R^n.
\]
Here and in the following, $c$ and $C$ denote positive constants depending on $\widetilde C$, $\widetilde \lambda$, $\widetilde r$, $\kappa$ and $n$, whose value may change from line to line.
By taking $A$ large enough, we can guarantee \[\min\{\mu_g(E),\mu_g(F)\}\ge cR^n.\]

Denote $B=B_d(x,2\widetilde\lambda R)$. Let \(\Gamma\) be the family of curves in \(B\) joining \(E\) to \(F\). In the following, we use the notion of the modulus of a curve family. Recall that the $2$-modulus of $\Gamma$ is defined as
$$\operatorname{Mod}_2(\Gamma):=\inf\int_X\rho^2d\mu_g,$$
where the infimum is taken over all Borel functions $\rho:X\to [0,\infty]$ satisfying
$$\int_\gamma \rho \,ds\geq 1,\quad \forall \gamma\in \Gamma.$$
By \cite[Corollary 9.3.2]{HKST2015} and the  Poincar\'e inequality, we know
\[
   \operatorname{Mod}_2(\Gamma)\ge cR^{n-2}.
\]

In order to guarantee length control and singularity avoidance, we remove two bad subfamilies of $\Gamma$. First let
\[
   \Gamma_{\rm long}
   =
   \{\gamma\in\Gamma:\ L_g(\gamma)>LR\}.
\]
Since \((LR)^{-1}\mathbf 1_B\) is admissible for \(\Gamma_{\rm long}\), we obtain \[\operatorname{Mod}_2(\Gamma_{\rm long})\le CL^{-2}R^{n-2}.\]
Next let \(0<\rho<R/(2A)\), write \(N_s=N_s(\mathcal S)\), and set
\[
   \Gamma_{\mathcal S,\rho}
   =
   \{\gamma\in\Gamma:\gamma\cap N_\rho\ne\emptyset\}.
\]
Every curve in \(\Gamma_{\mathcal S,\rho}\) starts outside \(N_{R/A}\) and meets \(N_\rho\), hence crosses \(N_{R/A}\setminus N_\rho\).  Put
\[
   L_{\rho,A}=\log\frac{R}{A\rho},
\]
and define
\[
   \varrho_{\rho,A}
   =
   \frac{C}{L_{\rho,A}}\,
   \delta^{-1}\,
   \mathbf 1_{N_{R/A}\setminus N_\rho}.
\]
After increasing \(C\), this function is admissible for \(\Gamma_{\mathcal S,\rho}\).  Let \(s_j=2^{-j}R/A\), and choose \(J\) so that
\(s_{J+1}<\rho\le s_J\).  Then
\[
\begin{split}
   \operatorname{Mod}_2(\Gamma_{\mathcal S,\rho})
   &\le
   \int_B\varrho_{\rho,A}^2\,d\mu_g  \\
   &\le
   \frac{C}{L_{\rho,A}^2}
   \sum_{j=0}^{J}
   s_j^{-2}\mu_g(N_{s_j}\cap B)  \\
   &\le
   \frac{C}{L_{\rho,A}^2}
   \sum_{j=0}^{J}R^{n-2}
   \le
   \frac{CR^{n-2}}{\log(R/(A\rho))}.
\end{split}
\]
Choose \(L\) large and then \(\rho\) small so that
\[
   \operatorname{Mod}_2(\Gamma_{\rm long})
   +
   \operatorname{Mod}_2(\Gamma_{\mathcal S,\rho})
   <
   \operatorname{Mod}_2(\Gamma).
\]
Hence some curve $\gamma_*$ in \(\Gamma\) avoids \(N_\rho(\mathcal S)\) and has \(g\)-length at most \(LR\). 

\vspace{3mm}
{\it Step 2.} there is a curve $\gamma_{x,y}\subset \mathcal R$ connecting $x$ and $y$ with $g$-length no greater than $(L+\beta)R$.
\vspace{3mm}

Assume that $\gamma_*$ connects $x_*\in B_d(x,\kappa R)$ and $y_*\in B_d(y,\kappa R)$. By local quasiconvexity, $x$ and $x_*$ can be connected by a curve $\gamma_{x,x_*}$, and $y$ and $y_*$ can be connected by a curve $\gamma_{y,y_*}$, both with length no greater than $\beta R/2$. Since we have assumed $\delta(x)\geq \beta R$ and $\delta(y)\geq \beta R$, both $\gamma_{x,x_*}$ and $\gamma_{y,y_*}$ are contained in $\mathcal R$. The concatenation of the three curves gives us the desired curve connecting $x$ and $y$.

\vspace{3mm}

Finally, we prove the lemma based on scale transition.
From the Ahlfors estimate \eqref{eq:local-ahlfors-section3} and the neighborhood estimate \eqref{Eq: neighborhood estimate}, there is a constant $\eta>0$ such that for any $r\in (0,\widetilde r)$ there is a point $x_r\in B_d(x,r)$ with $\delta(x_r)\geq \eta r$.

Set $r_j=2^j\delta(x)$ and let $j_0$ be the first index such that $r_{j_0}\geq R$. From the previous discussion, we can take $x_j\in B_{d}(x,r_j)$ with $\delta(x_j)\geq \eta r_j$ for $j\geq 1$, and $x_0=x$. Clearly, we have
$$d(x_{j-1},x_j)\leq 2r_j,\quad \delta(x_{j-1})\geq \frac{\eta}{2}r_j,\quad \delta(x_j)\geq \eta r_j.$$
By taking $r_K$ small enough such that $N_{r_K}(K)\subset K_*$, we can apply the claim to $(x_{j-1},x_j)$ with $\beta=\eta/4$ and construct a curve $\gamma_j\subset \mathcal R$ connecting $x_{j-1}$ and $x_j$ with $g$-length no greater than $C r_j$. Denote $x_*=x_{j_0}$ The concatenation of these curves gives a curve $\gamma_{x,x_{*}}\subset \mathcal R$ connecting $x$ and $x_*$ with $g$-length no greater than $CR$, where $\delta(x_*)\geq \eta R$. Similarly, there is a curve $\gamma_{y,y_*} \subset \mathcal R$ connecting $y$ and $y_*$ with $g$-length no greater than $CR$, where $\delta(y_*)\geq \eta R$. Note that we have
$$d(x_*,y_*)\leq (2C+1)R,\quad \delta(x_*)\geq \eta R,\quad \delta(y^*)\geq \eta R.$$
By taking $r_K$ small enough such that $d(x_*,y_*)\leq (2C+1)r_K\leq \widetilde r$, we can apply the claim to $(x_*,y_*)$ with $\beta=\eta/(2C+1)$ to conclude that $x_*$ and $y_*$ can be connected by a curve $\gamma_*\subset \mathcal R$ with $g$-length no greater than $CR$. The  concatenation of $\gamma_{x,x_*}$, $\gamma_*$, and $\gamma_{y,y_*}$ gives the desired curve with $g$-length no greater than $C_*R$ for some constant $C_*>0$. The proof is completed by taking $C_K=C_*$.

\end{proof}

We recall the form of the Sobolev-to-Lipschitz property used below.  When \(c_A^X(\mathcal S)\geq 2\), Lemma~\ref{lem:packing-zero-capacity-section3} gives zero \(2\)-capacity of \(\mathcal S\).  Thus \(W^{1,2}\)-functions are
represented on the regular part, and their energy is computed by
\[
   \int_{\mathcal R}|\nabla^g u|_g^2\,d\mu_g .
\]
Let \(d_g^\ell\) be the intrinsic length distance of \((\mathcal R,g)\), and set
\[
   X^\ell=\overline{(\mathcal R,d_g^\ell)} .
\]
In this setting, Sobolev-to-Lipschitz means that whenever
\(u\in W^{1,2}(X^\ell)\) satisfies
\[
   |\nabla^g u|_g\le1
   \qquad\text{a.e. on }\mathcal R,
\]
then \(u\) has a \(1\)-Lipschitz representative on \(X^\ell\). 

\begin{lemma}\label{lem:length-replacement-ampi}
Assume \(\mathsf P_{n-2}^{X}(\mathcal S)\). Then the identity map on \(\mathcal R\) extends to a locally bi-Lipschitz homeomorphism
\[
   X^\ell\longrightarrow X .
\]
Moreover, \(X^\ell\) satisfies the above Sobolev-to-Lipschitz property.
\end{lemma}

\begin{proof}
Clearly, \(d(x,y)\le d_g^\ell(x,y)\) on \(\mathcal R\). Conversely, by Lemma~\ref{lem:punctured-local-quasiconvexity}, for every compact \(K\Subset X\) there are constants \(C_K,r_K>0\) such that, whenever \(x,y\in K\cap\mathcal R\) and \(d(x,y)<r_K\), there is a curve \(\gamma\subset\mathcal R\) joining \(x\) to \(y\) with \(L_g(\gamma)\le C_Kd(x,y)\).  Hence
\[
   d_g^\ell(x,y)\le C_Kd(x,y)
\]
for such \(x,y\).  Thus the identity map extends locally bi-Lipschitz to the two completions. 

It remains to prove the Sobolev-to-Lipschitz property.  Let \(u\in W^{1,2}(X^\ell)\) satisfy
\[
   |\nabla^g u|_g\le1
   \qquad\text{a.e. on }\mathcal R .
\]
For \(x,y\in\mathcal R\), choose a regular curve \(\gamma\) with \(L_g(\gamma)\le d^\ell(x,y)+\varepsilon\).  Then
\[
   |u(x)-u(y)|
   \le
   \int_\gamma |\nabla^g u|\,ds_g
   \le
   d^\ell(x,y)+\varepsilon .
\]
Letting \(\varepsilon\downarrow0\), and then using the density of \(\mathcal R\subset X^\ell\), gives a \(1\)-Lipschitz representative.
\end{proof}

\begin{theorem}\label{thm:ampi-flat-rigidity}
Let \(X=\mathcal R\sqcup\mathcal S\) be compact, oriented, and enlargeable.  Assume
\[
   c_A^X(\mathcal S)>3-\frac1{n-1}.
\]
Let \(\lambda\leq 1/n\).  If
\[
   \Scal_f^{n,\lambda}(g)\ge0
   \qquad\text{on }\mathcal R,
\]
where \(e^{-f}d\mu_g\) is the admissible measure and \(f\) is bounded on \(\mathcal R\), then \(f\) is constant, and the length space \(X^\ell\) is a compact flat manifold.
\end{theorem}

\begin{proof}
By Proposition~\ref{prop:ampi-ricci-flat-reduction}, \(df\equiv0\) and \(\operatorname{Ric}_g\equiv0\) on \(\mathcal R \). Thus the admissible measure is a constant multiple of \(d\mu_g\). Replace \(X\) by \(X^\ell\).  By Lemma~\ref{lem:length-replacement-ampi}, the AM--PI structure and the Assouad codimension condition are preserved, and \(X^\ell\) satisfies the Sobolev-to-Lipschitz property.  By Proposition~\ref{prop:QL-assouad-codim3-ampi}, \(X^\ell\) satisfies QL.  Together with PI and Sobolev-to-Lipschitz, the Honda--Sun criterion \cite[Theorem~5.5]{HondaSun2026} gives \((X^\ell,d_g^\ell,\mu_g)\in RCD(0,n)\).

Since \(X^\ell\) is compact \(RCD(0,n)\), Mondino--Wei \cite[Theorem~3.6]{MondinoWei2019} shows its universal cover 
\[
   \widetilde X^\ell\cong Z\times\mathbb R^k,
\]
where \(Z\) is compact.  The argument in \cite[Proposition~IV.5.8]{LawsonMichelsohnSpinGeometry} rules out a nontrivial compact factor.  Hence \(Z\) is a point and \(k=n\).  Thus \(\widetilde X^\ell\cong\mathbb R^n \). Therefore \(X^\ell\) is a compact flat manifold.
\end{proof}

\begin{proof}[Proof of Theorem~D]
The first assertion is Theorem~\ref{thm:ampi-weighted-psc-obstruction}.  The second assertion is Theorem~\ref{thm:ampi-flat-rigidity}.  
\end{proof}

\begin{remark}
By a uniformly Euclidean \(L^\infty\)-metric on a closed smooth manifold \(M\), we mean a measurable positive definite symmetric \(2\)-tensor \(g\) such that, for some smooth background metric \(g_0\) and some constant
\(\Lambda\ge1\),
\[
   \Lambda^{-1}g_0(v,v)
   \le
   g(v,v)
   \le
   \Lambda g_0(v,v)
\]
for almost every \(x\in M\) and every \(v\in T_xM\).  

The theorem applies to such metrics.  Suppose \(M\) is closed and enlargeable, \(g\) is uniformly Euclidean on \(M\), and \(g\) is smooth on \(\Omega=M\setminus S\). Then the length completion associated with \((\Omega,g)\) is an AM--PI space. If, in addition,
\[
   c_A(S)>3-\frac1{n-1},
\]
then \(R_g\ge0\) on \(\Omega\) implies that this length space is a compact flat manifold.

This could be compared with the examples of Cecchini--Frenck--Zeidler~\cite{CecchiniFrenckZeidler}, which show that the positive scalar curvature obstruction for uniformly Euclidean \(L^\infty\)-metrics is not the same as in the smooth case.  
\end{remark}

\bibliographystyle{plain}
\bibliography{geroch}

@article{Assouad1983,
  author  = {Assouad, Patrice},
  title   = {Plongements lipschitziens dans {$\mathbb{R}^n$}},
  journal = {Bulletin de la Soci{\'e}t{\'e} Math{\'e}matique de France},
  volume  = {111},
  year    = {1983},
  pages   = {429--448},
  doi     = {10.24033/bsmf.1997}
}

@article{CecchiniFrenckZeidler,
  author        = {Cecchini, Simone and Frenck, Georg and Zeidler, Rudolf},
  title         = {Positive scalar curvature with point singularities},
  journal       = {Duke Mathematical Journal},
  note          = {To appear},
  year          = {2025},
  archivePrefix = {arXiv},
  eprint        = {2407.20163},
  primaryClass  = {math.DG},
  doi           = {10.48550/arXiv.2407.20163}
}

@article{Edelen2020,
  author  = {Edelen, Nick},
  title   = {A note on the singular set of area-minimizing hypersurfaces},
  journal = {Calculus of Variations and Partial Differential Equations},
  volume  = {59},
  number  = {1},
  year    = {2020},
  eid     = {18},
  pages   = {Paper No. 18},
  doi     = {10.1007/s00526-019-1660-7}
}

@article{Gromov2018,
  author  = {Gromov, Misha},
  title   = {Metric inequalities with scalar curvature},
  journal = {Geometric and Functional Analysis},
  volume  = {28},
  number  = {3},
  year    = {2018},
  pages   = {645--726},
  doi     = {10.1007/s00039-018-0453-z}
}

@article{GromovLawson1980,
  author  = {Gromov, Mikhail and Lawson, Jr., H. Blaine},
  title   = {Spin and scalar curvature in the presence of a fundamental group. {I}},
  journal = {Annals of Mathematics},
  series  = {2},
  volume  = {111},
  number  = {2},
  year    = {1980},
  pages   = {209--230},
  doi     = {10.2307/1971198}
}

@article{LesourdUngerYau,
  author  = {Lesourd, Martin and Unger, Ryan and Yau, Shing-Tung},
  title   = {The positive mass theorem with arbitrary ends},
  journal = {Journal of Differential Geometry},
  volume  = {128},
  number  = {1},
  year    = {2024},
  pages   = {257--293},
  doi     = {10.4310/jdg/1721075263}
}

@article{Lohkamp1999,
  author  = {Lohkamp, Joachim},
  title   = {Scalar curvature and hammocks},
  journal = {Mathematische Annalen},
  volume  = {313},
  number  = {3},
  year    = {1999},
  pages   = {385--407},
  doi     = {10.1007/s002080050266}
}

@article{NaberValtorta2020,
  author  = {Naber, Aaron and Valtorta, Daniele},
  title   = {The singular structure and regularity of stationary varifolds},
  journal = {Journal of the European Mathematical Society},
  volume  = {22},
  number  = {10},
  year    = {2020},
  pages   = {3305--3382},
  doi     = {10.4171/JEMS/987}
}

@article{SchoenYau1979a,
  author  = {Schoen, Richard and Yau, Shing-Tung},
  title   = {On the proof of the positive mass conjecture in general relativity},
  journal = {Communications in Mathematical Physics},
  volume  = {65},
  number  = {1},
  year    = {1979},
  pages   = {45--76},
  doi     = {10.1007/BF01940959}
}

@article{SchoenYau1979b,
  author  = {Schoen, Richard and Yau, Shing-Tung},
  title   = {On the structure of manifolds with positive scalar curvature},
  journal = {Manuscripta Mathematica},
  volume  = {28},
  number  = {1--3},
  year    = {1979},
  pages   = {159--183},
  doi     = {10.1007/BF01647970}
}

@incollection{SchoenYau2022,
  author        = {Schoen, Richard and Yau, Shing-Tung},
  title         = {Positive scalar curvature and minimal hypersurface singularities},
  booktitle     = {Surveys in Differential Geometry 2019. Differential Geometry, Calabi--Yau Theory, and General Relativity. Part 2},
  series        = {Surveys in Differential Geometry},
  volume        = {24},
  publisher     = {International Press},
  address       = {Boston, MA},
  year          = {2022},
  pages         = {441--480},
  eprint        = {1704.05490},
  archivePrefix = {arXiv},
  primaryClass  = {math.DG}
}

@article{ZhuArbitraryEnds,
  author  = {Zhu, Jintian},
  title   = {Positive mass theorem with arbitrary ends and its application},
  journal = {International Mathematics Research Notices},
  volume  = {2023},
  number  = {11},
  year    = {2023},
  pages   = {9880--9900},
  doi     = {10.1093/imrn/rnac117}
}

@book{LawsonMichelsohnSpinGeometry,
  author    = {Lawson, H. Blaine and Michelsohn, Marie-Louise},
  title     = {Spin Geometry},
  series    = {Princeton Mathematical Series},
  volume    = {38},
  publisher = {Princeton University Press},
  year      = {1989}
}

@book{Mattila1995,
  author    = {Mattila, Pertti},
  title     = {Geometry of Sets and Measures in Euclidean Spaces: Fractals and Rectifiability},
  series    = {Cambridge Studies in Advanced Mathematics},
  volume    = {44},
  publisher = {Cambridge University Press},
  year      = {1995}
}

@article {BjornBjornLehrback2020,
    AUTHOR = {Bj\"orn, Anders and Bj\"orn, Jana and Lehrb\"ack, Juha},
     TITLE = {Existence and almost uniqueness for {$p$}-harmonic {G}reen
              functions on bounded domains in metric spaces},
   JOURNAL = {J. Differential Equations},
  FJOURNAL = {Journal of Differential Equations},
    VOLUME = {269},
      YEAR = {2020},
    NUMBER = {9},
     PAGES = {6602--6640},
      ISSN = {0022-0396,1090-2732},
   MRCLASS = {31C45 (30L10 31C12 31C15 31E05 35J08 35J92 49Q20)},
  MRNUMBER = {4107063},
MRREVIEWER = {Benjamin\ Steinhurst},
       DOI = {10.1016/j.jde.2020.04.044},
       URL = {https://doi.org/10.1016/j.jde.2020.04.044},
}

@book {HKST2015,
    AUTHOR = {Heinonen, Juha and Koskela, Pekka and Shanmugalingam,
              Nageswari and Tyson, Jeremy T.},
     TITLE = {Sobolev spaces on metric measure spaces},
    SERIES = {New Mathematical Monographs},
    VOLUME = {27},
      NOTE = {An approach based on upper gradients},
 PUBLISHER = {Cambridge University Press, Cambridge},
      YEAR = {2015},
     PAGES = {xii+434},
      ISBN = {978-1-107-09234-1},
   MRCLASS = {30-02 (30L05 30L10 31E05 46E35)},
  MRNUMBER = {3363168},
MRREVIEWER = {David\ Matthew\ Freeman},
       DOI = {10.1017/CBO9781316135914},
       URL = {https://doi.org/10.1017/CBO9781316135914},
}

@book {BjornBjorn2011,
    AUTHOR = {Bj{\"o}rn, Anders and Bj{\"o}rn, Jana},
     TITLE = {Nonlinear potential theory on metric spaces},
    SERIES = {EMS Tracts in Mathematics},
    VOLUME = {17},
 PUBLISHER = {European Mathematical Society (EMS), Z{\"u}rich},
      YEAR = {2011},
     PAGES = {xii+403},
      ISBN = {978-3-03719-099-9},
   MRCLASS = {31C45 (30L99 49J45)},
  MRNUMBER = {2760645},
       DOI = {10.4171/099},
}

@article {BombieriGiusti1972,
    AUTHOR = {Bombieri, E. and Giusti, E.},
     TITLE = {Harnack's inequality for elliptic differential equations on
              minimal surfaces},
   JOURNAL = {Invent. Math.},
  FJOURNAL = {Inventiones Mathematicae},
    VOLUME = {15},
      YEAR = {1972},
     PAGES = {24--46},
      ISSN = {0020-9910,1432-1297},
   MRCLASS = {53A10 (35J99)},
  MRNUMBER = {308945},
MRREVIEWER = {E.\ F.\ Beckenbach},
       DOI = {10.1007/BF01418640},
       URL = {https://doi.org/10.1007/BF01418640},
}

@article {MondinoWei2019,
    AUTHOR = {Mondino, Andrea and Wei, Guofang},
     TITLE = {On the universal cover and the fundamental group of an {${\rm
              RCD}^*(K,N)$}-space},
   JOURNAL = {J. Reine Angew. Math.},
  FJOURNAL = {Journal f\"ur die Reine und Angewandte Mathematik. [Crelle's
              Journal]},
    VOLUME = {753},
      YEAR = {2019},
     PAGES = {211--237},
      ISSN = {0075-4102,1435-5345},
   MRCLASS = {53C23 (57M10)},
  MRNUMBER = {3987869},
MRREVIEWER = {Luis\ Guijarro},
       DOI = {10.1515/crelle-2016-0068},
       URL = {https://doi.org/10.1515/crelle-2016-0068},
}

@article {HondaSun2026,
    AUTHOR = {Honda, Shouhei and Sun, Song},
     TITLE = {From almost smooth spaces to {RCD} spaces},
   JOURNAL = {Calc. Var. Partial Differential Equations},
  FJOURNAL = {Calculus of Variations and Partial Differential Equations},
    VOLUME = {65},
      YEAR = {2026},
    NUMBER = {4},
     PAGES = {Paper No. 131, 45},
      ISSN = {0944-2669,1432-0835},
   MRCLASS = {53C21 (53C23)},
  MRNUMBER = {5045151},
       DOI = {10.1007/s00526-026-03297-2},
       URL = {https://doi.org/10.1007/s00526-026-03297-2},
}

@article {Gromov-Zhu,
    AUTHOR = {Gromov, Misha and Zhu, Jintian},
     TITLE = {Area and {G}auss-{B}onnet inequalities with scalar curvature},
   JOURNAL = {Comment. Math. Helv.},
  FJOURNAL = {Commentarii Mathematici Helvetici. A Journal of the Swiss
              Mathematical Society},
    VOLUME = {99},
      YEAR = {2024},
    NUMBER = {2},
     PAGES = {355--395},
      ISSN = {0010-2571,1420-8946},
   MRCLASS = {53C21 (53C23)},
  MRNUMBER = {4730349},
MRREVIEWER = {David\ J.\ Wraith},
       DOI = {10.4171/cmh/570},
       URL = {https://doi.org/10.4171/cmh/570},
}

@article {Korte2007,
    AUTHOR = {Korte, Riikka},
     TITLE = {Geometric implications of the {P}oincar\'e{} inequality},
   JOURNAL = {Results Math.},
  FJOURNAL = {Results in Mathematics},
    VOLUME = {50},
      YEAR = {2007},
    NUMBER = {1-2},
     PAGES = {93--107},
      ISSN = {1422-6383,1420-9012},
   MRCLASS = {46E35 (31C15)},
  MRNUMBER = {2313133},
MRREVIEWER = {Jana\ Bj\"orn},
       DOI = {10.1007/s00025-006-0237-x},
       URL = {https://doi.org/10.1007/s00025-006-0237-x},
}

@incollection{LohkampSecret,
  author        = {Lohkamp, Joachim},
  title         = {The Secret Hyperbolic Life of Positive Scalar Curvature},
  booktitle     = {Perspectives in Scalar Curvature},
  editor        = {Gromov, Mikhail L. and Lawson, Jr., H. Blaine},
  volume        = {1},
  chapter       = {5},
  pages         = {611--642},
  publisher     = {World Scientific},
  year          = {2023},
  doi           = {10.1142/9789811273223_0005},
  eprint        = {2203.16403},
  archivePrefix = {arXiv},
  primaryClass  = {math.DG}
}

@book{Geoghegan2008,
  author    = {Geoghegan, Ross},
  title     = {Topological Methods in Group Theory},
  series    = {Graduate Texts in Mathematics},
  volume    = {243},
  publisher = {Springer},
  address   = {New York},
  year      = {2008}
}

@article{Allegretto1974,
  author  = {Allegretto, Walter},
  title   = {On the Equivalence of Two Types of Oscillation for Elliptic Operators},
  journal = {Pacific Journal of Mathematics},
  volume  = {55},
  number  = {2},
  pages   = {319--328},
  year    = {1974},
  doi     = {10.2140/pjm.1974.55.319}
}

@article{Piepenbrink1974,
  author  = {Piepenbrink, John},
  title   = {Nonoscillatory Elliptic Equations},
  journal = {Journal of Differential Equations},
  volume  = {15},
  pages   = {541--550},
  year    = {1974}
}

@article{WangWangXie2026,
  author = {Wang, Jian and Wang, Jinmin and Xie, Zhizhang},
  title = {{$L^\infty$-metrics on tori and Schoen's conjecture}},
  journal = {arXiv preprint},
  year = {2026},
  note = {arXiv:2606.21325 [math.DG]}
}

@article{DaiWangWangWei2025,
  author = {Dai, Xianzhe and Wang, Changliang and Wang, Lihe and Wei, Guofang},
  title = {{Singular metrics with nonnegative scalar curvature and RCD}},
  journal = {arXiv preprint},
  year = {2025},
  note = {arXiv:2412.09185v2 [math.DG]}
}

@article{BiHaoHeShiZhu2026,
  author = {Bi, Yuchen and Hao, Tianze and He, Shihang and Shi, Yuguang and Zhu, Jintian},
  title = {A proof for the {Riemannian} positive mass theorem up to dimension 19},
  journal = {arXiv preprint},
  year = {2026},
  note = {arXiv:2603.02769 [math.DG]}
}

@article{BrendleWang2026,
  author = {Brendle, Simon and Wang, Yipeng},
  title = {A dimension descent scheme for the positive mass theorem in arbitrary dimension},
  journal = {arXiv preprint},
  year = {2026},
  note = {arXiv:2604.08473 [math.DG]}
}

@article{HeShiYu2026,
  author = {He, Shihang and Shi, Yuguang and Yu, Haobin},
  title = {Singularity removal rigidity theorems for minimal hypersurfaces in manifolds with nonnegative scalar curvature},
  journal = {arXiv preprint},
  year = {2026},
  note = {arXiv:2602.23705 [math.DG]}
}

@article{FischerColbrieSchoen1980,
  author  = {Fischer-Colbrie, Doris and Schoen, Richard},
  title   = {The Structure of Complete Stable Minimal Surfaces in 3-Manifolds of Nonnegative Scalar Curvature},
  journal = {Communications on Pure and Applied Mathematics},
  volume  = {33},
  number  = {2},
  pages   = {199--211},
  year    = {1980},
  doi     = {10.1002/cpa.3160330206}
}

@article {CheegerGromoll1971,
    AUTHOR = {Cheeger, Jeff and Gromoll, Detlef},
     TITLE = {The splitting theorem for manifolds of nonnegative {R}icci
              curvature},
   JOURNAL = {J. Differential Geometry},
  FJOURNAL = {Journal of Differential Geometry},
    VOLUME = {6},
      YEAR = {1971/72},
     PAGES = {119--128},
      ISSN = {0022-040X,1945-743X},
   MRCLASS = {53C20},
  MRNUMBER = {303460},
MRREVIEWER = {J.\ R.\ Vanstone},
       URL = {http://projecteuclid.org/euclid.jdg/1214430220},
}

@article{Smale1993,
  author  = {Smale, Nathan},
  title   = {Generic regularity of homologically area minimizing hypersurfaces in eight dimensional manifolds},
  journal = {Communications in Analysis and Geometry},
  volume  = {1},
  number  = {2},
  year    = {1993},
  pages   = {217--228},
  doi     = {10.4310/CAG.1993.v1.n2.a2}
}

@article{ChodoshMantoulidisSchulze,
  author  = {Chodosh, Otis and Mantoulidis, Christos and Schulze, Felix},
  title   = {Generic regularity for minimizing hypersurfaces in dimensions 9 and 10},
  journal = {Publications Math{\'e}matiques de l'IH{\'E}S},
  volume  = {143},
  year    = {2026},
  pages   = {143--188},
  doi     = {10.5802/pmihes.25},
  eprint  = {2302.02253},
  archivePrefix = {arXiv},
  primaryClass  = {math.DG}
}

@article{ChodoshMantoulidisSchulzeWang,
  author  = {Chodosh, Otis and Mantoulidis, Christos and Schulze, Felix and Wang, Zhihan},
  title   = {Generic regularity for minimizing hypersurfaces in dimension 11},
  journal = {arXiv preprint},
  year    = {2025},
  note    = {arXiv:2506.12852 [math.DG]}
}

@incollection{ChodoshICM,
  author    = {Chodosh, Otis},
  title     = {Minimal surfaces and comparison geometry},
  booktitle = {Proceedings of the International Congress of Mathematicians 2026},
  editor    = {Friedlander, Susan and Tschinkel, Yuri},
  volume    = {4},
  pages     = {82--102},
  publisher = {Society for Industrial and Applied Mathematics},
  year      = {2026},
  doi       = {10.1137/25M1804340}
}

\end{document}